\documentclass[12pt]{amsart}
\usepackage[utf8]{inputenc}
 
\title{Proximity to $\ell_p$ and $c_0$ in Banach spaces}
\author[RM Causey]{Ryan M Causey}
\address{Department of Mathematics \\ University of South Carolina \\
Columbia, SC 29208}
\email{CAUSEYRM@mailbox.sc.edu}

\usepackage{amssymb}
\usepackage{bm}
\usepackage{graphicx}
\usepackage[centertags]{amsmath}
\usepackage{amsfonts}
\usepackage{amsthm}
\usepackage{enumerate}
\usepackage{tocvsec2}
\usepackage{xcolor}
\usepackage[margin=1in]{geometry}

\newtheorem{theorem}{Theorem}[section]
\newtheorem*{theorem*}{Theorem}

\newtheorem{corollary}[theorem]{Corollary}
\newtheorem{lemma}[theorem]{Lemma}

\newtheorem{proposition}[theorem]{Proposition}

\theoremstyle{definition}

\newcommand{\fin}{[\mathbb{N}]^{<\omega}}
\newcommand{\infin}{[\mathbb{N}]}

\newcommand{\ord}{\textbf{Ord}}

\newcommand{\rr}{\mathbb{R}}
\newcommand{\nn}{\mathbb{N}}

\newcommand{\aaa}{\mathcal{A}}

\newcommand{\ddd}{\mathcal{D}}
\newcommand{\ccc}{\mathcal{C}}
\newcommand{\eee}{\mathcal{E}}
\newcommand{\fff}{\mathcal{F}}
\newcommand{\gggg}{\mathcal{G}}

\newcommand{\kkk}{\mathcal{K}}

\newcommand{\mt}{\mathcal{MT}}
\newcommand{\ttt}{\mathcal{T}}
\newcommand{\sss}{\mathcal{S}}
\newcommand{\www}{\mathcal{W}}

\newcommand{\ssm}{\mathcal{SM}}
\newcommand{\iii}{\mathcal{I}}
\newcommand{\uuu}{\mathcal{U}}
\newcommand{\jjj}{\mathcal{J}}
\newcommand{\ssj}{\mathcal{SJ}}

\newcommand{\supp}{\mathrm{supp}}

\DeclareMathOperator{\sgn}{sgn}

\DeclareMathOperator{\ran}{ran}

\begin{document}

\begin{abstract} We construct a class of minimal trees and use these trees to establish a number of coloring theorems on general trees.  Among the applications of these trees and coloring theorems are quantification of the Bourgain $\ell_p$ and $c_0$ indices, dualization of the Bourgain $c_0$ index, establishing sharp positive and negative results for constant reduction, and estimating the Bourgain $\ell_p$ index of an arbitrary Banach space $X$ in terms of a subspace $Y$ and the quotient $X/Y$.

\end{abstract} 



\maketitle

\tableofcontents

\addtocontents{toc}{\setcounter{tocdepth}{1}}

\section{Introduction}

Withint the class of Banach spaces, the sequence spaces $\ell_p$ and $c_0$ play a central role.  James \cite{James} famously showed that every Banach space which contains isomorphic copies of either $\ell_1$ or $c_0$ must contain almost isometric copies of these spaces.  With Odell and Schlumprecht's solution of the distortion problem \cite{OS}, it was shown that the corresponding result for $\ell_p$, $1<p<\infty$, is false.   However, it follows from Krivine's theorem \cite{K} that if $\ell_p$, $1<p<\infty$, is crudely finitely representable in $X$, then $\ell_p$ is finitely representable in $X$. The corresponding result for $\ell_1$ and $c_0$ was also shown by James in \cite{James}.   Thus we see a difference between the weakest form of admitting subspaces of a Banach space admitting $\ell_p$ structure, finite representability of $\ell_p$ in $X$, and the strongest notion, admitting isomorphic copies of $\ell_p$ as a subspace.  The Bourgain $\ell_p$ index \cite{Bo} allows a quantification which bridges the gap between these two forms of admitting $\ell_p$ structures within a given Banach space.  In \cite{JO}, transfinite versions of the argument of James were given for countable ordinals which allow a unified approach to the two previously mentioned results of James.  There, they also gave an argument that the transfinite versions of the argument must fail for the $\ell_p$ index when $1<p<\infty$ must break down somewhere between the positive results from Krivine, which represent the case of an index at least $\omega$, the smallest infinite ordinal, and the negative results which follow from Odell and Schlumprecht, and they asked for a quantification of this result.

Another well-known fact is that if $X, Y$ are Banach spaces admitting no subspace isomorphic to $\ell_p$, then $X\oplus Y$ admits no subspace isomorphic to $\ell_p$.  Dodos \cite{D} prove that, in the sense of the Bourgain $\ell_p$ index, this result is uniform.  That is, there exists a function $\phi_p:[1, \omega_1)^2\to [1, \omega_1)$ so that if $X, Y$ have $\ell_p$ indices $\xi, \zeta$, respectively, then the Bourgain $\ell_p$ index of $X\oplus Y$ cannot exceed $\phi_p(\xi, \zeta)$.  Dodos asked for an explicit estimate on the function $\phi_p$.  

This paper is arranged as follows.  In Section $2$, we discuss trees, derivations, and orders and establish several standard methods that we will use throughout to witness estimates of the Bourgain $\ell_p$ index of a Banach space.  In Section $3$, we establish a number of coloring lemmas of independent interest which yield certain dichotomies on trees.  In Section $4$, we define the different types of $\ell_p$ and $c_0$ structures we wish to quantify and the classes of Banach spaces which will be our primary objects of study.  In Section $5$, we discuss the dualization of the Bourgain $\ell_p$ indices.  In Sections $6$ and $7$, we discuss positive and negative results concerning the transfinite analogues of the equivalence of containment of crude $\ell_p$ and almost isometric $\ell_p$ structures mentioned above.  More precisely, in Section $6$ we given precise positive results for $\ell_1$ and $c_0$ structures and establish the sharpness of these results.  In Section $7$, we discuss the implicitions of Krivine's theorem in producing positive results for $1<p<\infty$, and quantify the Odell-Schlumprecht distortion of $\ell_p$ to answer the question raise in \cite{JO} of quantifying where the positive results may fail.  Finally, in Section $8$, we provide explicit estimates on the function $\phi_p$ of Dodos.  More generally, we provide estimates of the Bourgain $\ell_p$ index of an arbitrary Banach space in terms of the Bourgain $\ell_p$ indices of a subspace and the corresponding quotient.  We use this result to estimate the Bourgain $\ell_p$ index of infinite direct sums of Banach spaces.

\section{Trees, derivations, and order}

Given a set $S$, we let $S^{<\omega}$ denote the finite sequences in $S$.  We include in $S^{<\omega}$ the empty sequence, denoted $\varnothing$.  Given $s\in S^{<\omega}$, we let $|s|$ denote the length of $s$.  If $s=(x_i)_{i=1}^k$ and $0\leqslant n\leqslant k$, we let $s|_n=(x_i)_{i=1}^n$.  If $s,t\in S^{<\omega}$, we let $s\verb!^!t$ denote the concatenation of $s$ with $t$.   We order $S^{<\omega}$ with the partial order $\prec$ by letting $s\prec t$ if and only if $s$ is a proper initial segment of $t$.  In this case, we say $s$ is a \emph{predecessor} of $t$ and that $t$ is a \emph{successor} or \emph{extension} of $s$.

We say a subset $T\subset S^{<\omega}$ is a \emph{tree on} $S$ (or simply a \emph{tree}) provided it is downward closed with respect to the order $\prec$.  We say a tree is \emph{hereditary} provided it contains all subsequences of its members. If $S$ is a topological space, we say $T$ is a \emph{closed} tree if for each $n\in \nn$, $T\cap S^n$ is a closed set. Here, $S^n$ is endowed with the product topology.   If $A$ is a subset of a vector space, and if $1\leqslant p\leqslant \infty$, we say $x$ is a $p$-\emph{absolutely convex combination} of $A$ if there exists $(x_i)_{i=1}^n\in A^{<\omega}$ and $(a_i)_{i=1}^n\in S_{\ell_p^n}$ so that $x=\sum_{i=1}^n a_ix_i$.  The $p$-\emph{absolutely convex hull of} $A$ is the set of all $p$-absolutely convex combinations of $A$ and is denoted by $\text{co}_p(A)$.  We say $(y_i)_{i=1}^m$ is a $p$-\emph{absolutely convex block} of $(x_i)_{i=1}^n$ if there exist $0=k_0<\ldots<k_m=n$ so that for each $1\leqslant i\leqslant m$, $y_i\in \text{co}_p (x_j: k_{i-1}<j\leqslant k_i)$.  We say the tree $T$ is $p$-\emph{absolutely convex} provided it contains all $p$-absolutely convex blocks of its members.

If $T\subset S^{<\omega}$ and if $s\in S^{<\omega}$, we let $T(s)=\{t\in S^{<\omega}: s\verb!^!t\in T\}$.  If $T$ is a tree, $T(s)$ is as well.  If $U\subset S^{<\omega}$, we let $C(U)$ denote all finite, non-empty subsets of $U\setminus \{\varnothing\}$ which are linearly ordered with respect to the order $\prec$, and we call the members of $C(U)$ the \emph{segments} of $U$.  We order $C(U)$ by letting $c<c'$ if $s\prec t$ for all $s\in c$ and $t\in c'$.  We also order $C(U)\cup \{\varnothing\}$ by declaring $\varnothing <c$ for each $c\in C(U)$. If $P, Q$ are partially ordered sets with orders $<_P$, $<_Q$, respectively, we say a function $f:P\to Q$ is \emph{order preserving} if for each $x,y\in P$ with $x<_P y$, $f(x)<_Q f(y)$.  We emphasize that this is not an if and only if.  That is, if $x,y$ are incomparable in $P$, $f(x), f(y)$ need not be incomparable in $Q$ in order for $f$ to be order preserving. Of course, the composition of order preserving maps is order preserving.  If $i:\ttt\to \www$ is order preserving, then the image of a segment under $i$ is also a segment.  Thus for any $U\subset\ttt$, $i$ induces a map, which we also call $i$, between $C(U)$ and $C(i(U))$.  That is, for $c\in C(U)$, $i(c)=\{i(t): t\in c\}$.   We say $f:P\to Q$ is an \emph{embedding} if for each $x,y\in P$, $x<_P y$ if and only if $f(x)<_Q f(y)$.  

We let $MAX(T)$ denote the maximal members of $T$ with respect to the order $\prec$.  We let $T'=T\setminus MAX(T)$, and note that if $T$ is a tree, $T'$ is a tree as well. For $\xi\in \ord$, the class of ordinal numbers, we define the transfinite derived trees $d^\xi(T)$ of $T$ by $$d^0(T)=T,$$ $$d^{\xi+1}(T)=(d^\xi(T))',$$ and if $\xi$ is a limit ordinal, $$d^\xi(T)=\bigcap_{\zeta<\xi} d^\zeta(T).$$  Note that for any set $S$ and any tree $T$ on $S$, there must exist $\xi\in \ord$ so that $d^{\xi+1}(T)=d^\xi(T)$.  We say that $T$ is \emph{well-founded} if $d^\xi(T)=\varnothing$ for some $\xi$, and \emph{ill-founded} otherwise. We let $$o(T)=\min \{\xi\in \ord: d^\xi(T)=\varnothing\},$$ where, as a matter of convenience, we write $\min \varnothing = \infty$.  That is, $o(T)=\infty$ means that $T$ is ill-founded.  Also as a matter of convenience, we declare that $\xi<\infty$ for any $\xi\in \ord$, and that $\xi\infty = \infty = \infty \xi$, and $\infty+\xi=\infty=\xi+\infty$.  Note that if $o(T)<\infty$, $o(T)$ must be a successor, since $\varnothing \in d^\xi(T)$ whenever $d^\xi(T)$ is non-empty.  We observe that $T$ is ill-founded if and only if there exists an infinite sequence $(x_i)\subset S$ so that for all $n\in \nn$, $(x_i)_{i=1}^n\in T$.  

We say $T\subset S^{<\omega}\setminus \{\varnothing\}$ is a $B$-\emph{tree} provided $T\cup \{\varnothing\}$ is a tree.  For a number of our purposes, it will be convenient to avoid having to include the empty sequence in our considerations.  Note that all notions above concerning trees have obvious analogous definitions for $B$-trees.

The following items are either trivial or easily shown by induction.  

\begin{proposition} Let $S$ be a set and let $T$ be a tree on $S$.  \begin{enumerate}[(i)]\item For any $\xi\in \emph{\ord}$ and $s\in S^{<\omega}$, $d^\xi(T)(s)=d^\xi(T(s))$. \item For any $\xi, \zeta\in \emph{\ord}$, $d^\zeta(d^\xi(T))=d^{\xi+\zeta}(T)$. \item For any $s\in S^{<\omega}$ and $\xi\in \emph{\ord}$, $\xi<o(T(s))$ if and only if $s\in d^\xi(T)$.  In particular, if $T(s)$ is well-founded, $o(T(s))=\max\{\xi\in \emph{\ord}: s\in d^\xi(T)\}+1$.  

\end{enumerate}

Moreover, (i) and (ii) hold if $T$ is a $B$-tree, and (iii) does as well as long as $s\neq \varnothing$.

\label{derivation prop}

\end{proposition}

We will also want standard ways to construct new $B$-trees with prescribed orders from given $B$-trees of specified orders.  The first method will be to construct a tree of order $\xi+\zeta$ by placing a ``copy'' of a tree of order $\xi$ beneath every maximal member of a $B$-tree of order $\zeta$.  The second method will be to take a totally incomparable union of $B$-trees with orders having supremum $\xi$.

\begin{proposition} \begin{enumerate}[(i)]\item Let $A, B$ be disjoint sets and let $\xi, \zeta\in \emph{\ord}$.  Suppose that $T_\varnothing$ is a $B$-tree on $A$ of order $\zeta$ and that for each $t\in MAX(T_\varnothing)$, $T_t$ is a $B$-tree on $B$ of order $\xi$.  Then $$T=T_\varnothing \cup \{t\verb!^!s: t\in MAX(T_\varnothing), s\in T_t\}$$ is a $B$-tree on $A\cup B$ of order $\xi+\zeta$. 

\item Fix a limit ordinal $\xi$ and a subset $A\subset [0, \xi)$ such that $\sup A=\xi$.  Suppose that for each $\zeta\in A$, there exists $A_\zeta\subset [0,\xi)$ so that the sets $(A_\zeta)_{\zeta\in A}$ are pairwise disjoint.  Suppose that for each $\zeta\in A$, $T_\zeta$ is a $B$-tree on $A_\zeta$ with $o(T_\zeta)=\zeta$.  Then $$T=\bigcup_{\zeta\in A}T_\zeta$$ is a $B$-tree on $[0,\xi)$ with order $\xi$.  

\item Fix a limit ordinal $\xi$ and a subset $A\subset [0, \xi)$ such that $\sup A=\xi$. Suppose that for each $\zeta\in A$, $T_\zeta\subset [0, \zeta)$ is a $B$-tree of order $\zeta$. Then $$T=\bigcup_{\zeta\in A}\{(\zeta), (\zeta)\verb!^!t: t\in T_\zeta\}$$ is a $B$-tree on $[ 0, \xi)$ with order $\xi$. \end{enumerate}

\label{new from old}
\end{proposition}

\begin{proof}(i) Note that for each $t\in MAX(T_\varnothing)$, $T(t)=T_t\cup \{\varnothing\}$.  Therefore using the correspondence $s\leftrightarrow t\verb!^!s$, we deduce $d^\xi(T(t))=\{\varnothing\}$, so that $t\in MAX(d^\xi(T))$.  Thus $d^\xi(T)=T_\varnothing$, whence $o(T)=\xi+\zeta$.  

(ii) Note that the derived set of the union is simply the union of the derived sets, so that for each $\eta<\xi$, $$d^\eta(T)=\bigcup_{\zeta\in A} d^\eta(T_\zeta).$$

(iii) Note that the $B$-trees $S_\zeta=\{(\zeta), (\zeta)\verb!^!t: t\in T_\zeta\}$ are totally incomparable for distinct $\zeta$.  This means $d^\eta(T)=\cup_{\zeta\in A}d^\eta(S_\zeta)$ for each $\eta<\xi$.  It is clear, or follows from (i), that $o(S_\zeta)=o(T_\zeta)+1=\zeta+1<\xi$.  Thus for any $\eta<\xi$, there exists $\zeta\in A$ with $\eta<\zeta$, and so $d^\eta(T)\supset d^\eta(S_\zeta)\neq \varnothing$.  But $d^\xi(T)=\cup_{\zeta\in A} d^\xi(S_\zeta)=\varnothing$.

\end{proof}

For convenience, we will define for each $\xi\in \ord$ a tree $\mt_\xi$ on $[1, \xi]$ which will be minimal in a sense which is made clear in the proposition below.  More precisely, we construct $\mt_\xi$ by induction on $\xi$ so that $\mt_\xi$ is a tree on $[1, \xi]$ consisting of strictly decreasing sequences in $[1, \xi]$.  These trees can be compared to the trees constructed in \cite{JO}.  There a different definition of tree was used.  The sets they define are partially ordered by set inclusion, rather than our order $\prec$.  

We let $$\mt_0=\{\varnothing\},$$ $$\mt_{\xi+1}=\{\varnothing\}\cup \{(\xi+1)\verb!^!t: t\in \mt_\xi\},$$ and if $\xi\in \ord$ is a limit ordinal, $$\mt_\xi=\bigcup_{\zeta<\xi}\mt_{\zeta+1}.$$  For each $\xi\in \ord$, we let $\ttt_\xi=\mt_\xi\setminus \{\varnothing\}$.  Note that for a limit ordinal $\xi$, $\ttt_\xi=\cup_{\zeta<\xi}\ttt_{\zeta+1}$, and that the $B$-trees in this union are totally incomparable.  This is because every member of $\ttt_{\zeta+1}$ is an extension of $(\zeta+1)$.

\begin{proposition} Let $\xi\in \emph{\ord}$. Let $S$ be a set and let $T$ be a $B$-tree on $S$.    \begin{enumerate}[(i)]\item $\ttt_\xi$ is a $B$-tree on $[1, \xi]$ with $o(\ttt_\xi)=\xi$. \item If there exists $f:\ttt\to T$ which is order preserving, $o(T)\geqslant o(\ttt)$.  \item $o(T)\geqslant \xi$ if and only if there exists an order preserving $f:\ttt_\xi\to T$.  \item $o(T)\geqslant \xi$ if and only if there exists a function $f:\ttt_\xi\to S$ so that for each $t\in \ttt_\xi$, $(f(t|_i))_{i=1}^{|t|}\in T$.  
\end{enumerate}

\label{minimal trees}
\end{proposition}

\begin{proof}(i) That $\ttt_\xi$ is a $B$-tree on $[1,\xi]$ is trivial.  That $o(\ttt_\xi)=\xi$ follows from Proposition \ref{new from old}.

(ii) We claim that for each $\zeta<o(\ttt)$ and each $t\in d^\zeta(\ttt)$, $f(t)\in d^\zeta(T)$.  The proof is by induction on $\zeta$ with the base case and limit ordinal case trivial.  If the claim holds for $\zeta$ and if $\zeta+1<o(\ttt)$, then for $t\in d^{\zeta+1}(\ttt)$ we can choose $t\prec s\in d^\zeta(\ttt)$.  Then $f(t)\prec f(s)\in d^\zeta(T)$ and $f(t)\in d^{\zeta+1}(T)$.  

We prove (iii) and (iv) together. First, if $f:\ttt_\xi\to S$ is such that for each $t\in \ttt_\xi$, $(f(t|_i))_{i=1}^{|t|}\in T$, then $g(t)=(f(t|_i))_{i=1}^{|t|}$ defines an order preserving map of $\ttt_\xi$ into $T$, so that $o(T)\geqslant o(\ttt_\xi)=\xi$.  It remains only to show that if $o(T)\geqslant \xi$, there exists such an $f$.  If $\xi=0$, there is nothing to prove.  Suppose we have established the result for some $\xi$ and that $o(T)\geqslant \xi+1$.  Then this means $d^\xi(T)$ cannot be empty, and therefore that there exists some $x\in S$ so that $(x)\in d^\xi(T)$.  But by Proposition \ref{derivation prop}, $o(T(x))\geqslant \xi$.  This means there exists $f':\ttt_\xi\to S$ so that for each $t\in \ttt_\xi$, $(f'(t|_i))_{i=1}^{|t|}\in T(x)$. Define $f:\ttt_{\xi+1}\to S$ by letting $f((\xi+1))=x$ and $f((\xi+1)\verb!^!t)=f'(t)$ for $t\in \ttt_\xi$.  This map is clearly seen to satisfy the requirements. 

Last, suppose the result holds for every $\zeta<\xi$, $\xi$ a limit ordinal, and suppose $o(T)\geqslant \xi$.  Then $o(T)\geqslant \zeta+1$ for every $\zeta<\xi$, so there exists $f_\zeta:\ttt_{\zeta+1}\to S$ so that for each $t\in \ttt_{\zeta+1}$, $(f_\zeta(t|_i))_{i=1}^{|t|}\in T$.  Define $f$ on $\ttt_\xi$ by letting $f|_{\ttt_{\zeta+1}}=f_\zeta$.

\end{proof}

\subsection{Regular families}

If $M$ is any subset of $\nn$, we let $[M]$ (resp. $[M]^{<\omega}$) denote the infinite (resp. finite) subsets of $M$.  Throughout, we identify subsets of $\nn$ with strictly increasing sequences in $\nn$ in the obvious way.  In this way, we can identify $\fin$ with a subset of $\nn^{<\omega}$, and each of the definitions from the previous section concerning trees and $B$-trees can be applied to $\fin$ with this identification.  In particular, $\fff\subset \fin$ is hereditary if it contains all subsets of its members.  For $\fff\subset \fin$, we let $\widehat{\fff}=\fff\setminus \{\varnothing\}$.  

We also identify $\fin$ with a subset of the Cantor set $2^\nn$ by identifying sets with their indicator functions.  We endow $\fin$ with the topology it inherits from this identification, and we say $\fff\subset \fin$ is compact if it is compact with respect to this topology.  

If $(j_i)_{i=1}^l, (k_i)_{i=1}^l\in \fin$ are such that $j_i\leqslant k_i$ for each $1\leqslant i\leqslant l$, then we say $(k_i)_{i=1}^l$ is a \emph{spread} of $(j_i)_{i=1}^l$.  We say $\fff \subset \fin$ is \emph{spreading} provided it contains all spreads of its members.  We say $\fff\subset \fin$ is \emph{regular} provided it is compact, spreading, and hereditary.  

We note that if $\fff$ is spreading, $MAX(\fff)$ is the set of isolated points in $\fff$, and $\fff'$ is the Cantor-Bendixson derivative of $\fff$. Note that $\fff'$ is also regular.   For $\fff$ regular, we define $$\iota(\fff)=\min \{\xi<\omega_1: d^\xi(\fff)\subset \{\varnothing\}\}.$$  Note that since $\fff$ is countable compact Hausdorff, $\fff$ must have countable Cantor-Bendixson index.  If $\varnothing \neq \fff$, then $d^{\iota(\fff)}(\fff)=\{\varnothing\}$ and $\iota(\fff)+1$ is the Cantor-Bendixson index of $\fff$.  

We write $E<F$ provided $\max E<\min F$.  We write $n<E$ (resp. $n\leqslant E$) provided $n<\min E$ (resp. $n\leqslant \min E$).  By convention, $\varnothing <E<\varnothing$ for all $E\in \fin$.   If $E<F$, we write $E\verb!^!F$ in place of $E\cup F$.  We say $(E_i)_{i=1}^n$ is $\fff$-admissible provided $E_1<\ldots <E_n$ and $(\min E_i)_{i=1}^n\in \fff$.  

If $M\in \infin$, there is a natural bijection between $2^\nn$ and $2^M$, which we also denote by $M$.  That is, if $M=(m_n)_{n\in \nn}$, $M(E):=(m_n: n\in E)$.  We let $\fff(M)=\{M(E): E\in \fff\}$.

If $(\gggg_i)_{i=1}^n$ are regular, then $$(\gggg_1, \ldots, \gggg_n):=\{E_1\verb!^! \ldots \verb!^!E_n: E_i\in \gggg_i \text{\ }\forall 1\leqslant i\leqslant n\}$$ is regular.  For $\gggg$ regular and $n\in \nn$, we let $(\gggg)_n=\gggg$ if $n=1$ and $(\gggg)_n=(\gggg, (\gggg)_{n-1})$ if $n>1$.  Note that $(\gggg_1, (\gggg_2, \gggg_3))=(\gggg_1, \gggg_2, \gggg_3)=((\gggg_1, \gggg_2), \gggg_3)$ for any $\gggg_1, \gggg_2, \gggg_3$ regular. Therefore for $\gggg$ regular and $m,n\in \nn$, $((\gggg)_m, (\gggg)_n)=(\gggg)_{m+n}$. It is easy to see that $\iota((\fff, \gggg))=\iota(\gggg)+\iota(\fff)$.  We think of $(\fff, \gggg)$ as being the sum of $\fff$ and $\gggg$.  

If $\fff, \gggg$ are regular, we let $$\fff[\gggg]=\Bigl\{\bigcup_{i=1}^s E_i: (E_i)_{i=1}^s\subset \gggg \text{\ is\ }\fff\text{-admissible}\Bigr\}.$$ This is easily seen to be regular.  If $\gggg$ is regular, we can define $[\gggg]^n=\gggg$ if $n=1$ and $[\gggg]^n=\gggg\bigl[[\gggg]^{n-1}\bigr]$ if $n>1$. We observe that $\gggg_1\bigl[\gggg_2[\gggg_3]\bigr]=(\gggg_1[\gggg_2])[\gggg_3]$ if $\gggg_1, \gggg_2, \gggg_3$ are regular, so that for all regular families $\gggg$ and all $m,n\in \nn$, $[\gggg]^m\bigl[[\gggg]^n\bigr] = [\gggg]^{m+n}$.  We note that $\iota(\fff[\gggg])=\iota(\gggg)\iota(\fff)$.  We think of $\fff$ and $\gggg$ be the product of $\fff$ and $\gggg$.  

If $\fff_n$ is a sequence of regular families with $\iota(\fff_n)\uparrow \xi$, then we can construct a diagonalization of the sequence $$\fff=\{E: \exists n\leqslant n\in \fff_n\}.$$  Then $\fff$ is regular with $\iota(\fff)=\xi$.

We are now in a position to define certain regular families which we will use throughout.  We define $\aaa_n=\{E\in \fin: |E|\leqslant n\}$.  We let $\sss=\{E: n\leqslant E\in \aaa_n\}$.  We also define $\sss_0=\aaa_1$, $\sss_{\xi+1}=\sss[\sss_\xi]$, and if $\xi<\omega_1$ is a limit ordinal with $\sss_\zeta$ defined for each $\zeta<\xi$, we fix $\xi_n\uparrow \xi$ and let $$\sss_\xi=\{E: \exists n\leqslant E\in \sss_{\xi_n+1}\}.$$  We remark that it is known that, in this case, we can choose the ordinals $\xi_n\uparrow \xi$ so that $\sss_{\xi_n+1}\subset \sss_{\xi_{n+1}}$ for all $n\in \nn$ (this was shown in \cite{Causey}).  We observe that the families resulting from this process are not uniquely determined, since we made choices of $\xi_n\uparrow \xi$ at limit ordinals, but this is unimportant for our purposes as long as we have the containment indicated above.  Moreover, it is easy to see that $\iota(\sss_\xi)=\omega^\xi$.  Also, for convenience, we will let $\sss_{\omega_1}=\fin$.  The families $(\sss_\xi)_{\xi<\omega_1}$ are called the \emph{Schreier families}.

\subsection{Cantor normal form and Hessenberg sums}

We collect here definitions and facts about ordinals which will be used throughout. The following facts can be found in \cite{Monk}.  

\begin{proposition} Fix $\alpha, \beta, \xi\in \emph{\ord}$.   \begin{enumerate}[(i)]\item If  $\alpha<\omega^\xi$, $\alpha+\omega^\xi= \omega^\xi$.  \item If $\alpha, \beta<\omega^\xi$, $\alpha+\beta<\omega^\xi$. \item If $0<\alpha<\omega^{\omega^\xi}$, $\alpha\omega^{\omega^\xi}=\omega^{\omega^\xi}$.  \item If $\alpha, \beta<\omega^{\omega^\xi}$, $\alpha\beta<\omega^{\omega^\xi}$. \item If $\alpha<\beta$ and $0<\xi$, $\xi\alpha<\xi\beta$.  \item If $\alpha<\beta$, $\omega^\alpha<\omega^\beta$.  \item If $\xi\geqslant \omega$, then there exists $\zeta\in \emph{\ord}$ so that $\xi=\omega^\zeta$ if and only if for every limit ordinal $\lambda<\xi$, $\lambda 2<\xi$.  \item The function $\zeta\mapsto \xi+\zeta$ is an order preserving bijection between $[\alpha, \beta)$ and $[\xi+\alpha, \xi+\beta)$.  \item Either $\xi=0$ or there exist unique $k\in \nn$, $\alpha_1>\ldots>\alpha_k$, $\alpha_i\in \emph{\ord}$ and $n_i\in \nn$ so that $\xi=\omega^{\alpha_1}n_1+\ldots + \omega^{\alpha_k}n_k$.  This is called the \emph{Cantor normal form} of $\xi$.  

\end{enumerate}

\label{ordinals}
\end{proposition}

We note that $\omega^\zeta> \omega^{\alpha_1}n_1+\ldots + \omega^{\alpha_k}n_k$ if and only if $\zeta>\alpha_1$. Ordinals of the form $\omega^\xi$ are called \emph{gamma numbers}, while ordinals of the form $\omega^{\omega^\xi}$ are called \emph{delta numbers}.  Both gamma numbers and delta numbers will be important in this work.  The only statement above which is not explicitly stated in Monk is (v), so we discuss it.  Suppose $\xi=\omega^{\alpha_1}n_1+\ldots+\omega^{\alpha_k}n_k$ is the Cantor normal form of $\xi$.  First, if $k>1$, then $\lambda=\omega^{\alpha_1}n_1+\ldots+\omega^{\alpha_{k-1}}n_{k-1}<\xi$ is a limit ordinal and $\lambda 2>\lambda + \omega^{\alpha_k}n_k=\xi$.  If $k=1$, then $\alpha_k>0$, or else $\xi$ is finite.  If $n_1>1$, $\lambda=\omega^{\alpha_1}(n_1-1)<\xi$ is a limit and $\lambda 2 = \omega^{\alpha_1}(2n_1-2)\geqslant \xi$.  Thus if $\xi\geqslant \omega$ has the property that for each limit ordinal $\lambda <\xi$, $\lambda 2<\xi$, $\xi$ is a gamma number.

Note that if $f:S\to S_0$ is a bijection, $(x_i)_{i=1}^n\mapsto ((f(x_i))_{i=1}^n$ is a bijection between $S^{<\omega}$ and $S_0^{<\omega}$ taking trees (resp. $B$-trees) to trees (resp. $B$-trees) of the same order.  In particular, for any $\xi\in \ord$ and $n\in \nn$, $\zeta\mapsto \omega^\xi n+ \zeta$ maps $[1, \omega^\xi)$ bijectively onto $[\omega^\xi n+1, \omega^\xi (n+1))$.  Moreover, for any $\xi\in \ord$, $\zeta\mapsto \omega^\xi+\zeta$ takes $[1, \omega^{\xi+1})$ bijectively onto $[\omega^\xi+1, \omega^\xi+\omega^{\xi+1})=[\omega^\xi+1, \omega^{\xi+1})$.  We will use this fact to take given $B$-trees on the same set and transfer them in a natural way to $B$-trees on disjoint sets while preserving the orders of the trees.

If $\xi, \zeta\in \ord$, by allowing $m_i,n_i$ to be zero, we can find $\alpha_1>\ldots >\alpha_k$ and $m_i, n_i\in \nn$ so that $$\xi=\omega^{\alpha_1}m_1+\ldots + \omega^{\alpha_k} m_k, \hspace{5mm} \zeta= \omega^{\alpha_1}n_1+\ldots + \omega^{\alpha_k}n_k.$$  Then we define the \emph{Hessenberg sum}, or \emph{natural sum}, of $\xi$ and $\zeta $ by $$\xi\oplus \zeta= \omega^{\alpha_1}(m_1+n_1)+\ldots + \omega^{\alpha_k}(m_k+n_k).$$  
 
 We observe that for every $\xi\in \ord$, the set $\{(\alpha, \beta): \alpha\oplus \beta=\xi\}$ is finite, and if $\alpha\oplus \beta=\omega^\xi$, either $\alpha=\omega^\xi$ or $\beta=\omega^\xi$.  From this, it is easy to see that if $(\alpha_i, \beta_i)_{i\in I}$ is such that $\sup_{i\in I} \alpha_i \oplus \beta_i=\omega^\xi$, then either $\sup_{i\in I} \alpha_i=\omega^\xi$ or $\sup_{i\in I} \beta_i= \omega^\xi$.  If it were not so, say $\alpha= \sup_{i\in I}\alpha_i, \beta=\sup_{i\in I}\beta_i<\omega^\xi$, then $$\sup_{i\in I} \alpha_i\oplus \beta_i\leqslant \alpha\oplus \beta < \omega^\xi.$$

\begin{proposition} Let $\xi$ be a limit ordinal and suppose that for each $\zeta<\xi$, we have ordinals $\xi_{0, \zeta}$, $\xi_{1,\zeta}$ so that $\xi_{0,\zeta}\oplus \xi_{1, \zeta}=\zeta+1$.  Then there exist a set $A\subset [0,\xi)$, ordinals $\xi_0, \xi_1$ so that $\xi_0\oplus \xi_1=\xi$, and $j\in \{0,1\}$ so that $\xi_j$ is a limit ordinal and $$\sup_{\zeta\in A} \xi_{j,\zeta}= \xi_j \text{\ and\ }\min_{\zeta\in A}\xi_{1-j, \zeta} \geqslant \xi_{1-j}.$$  

\label{splitting technique}

\end{proposition}

\begin{proof} Write $\xi$ in Cantor normal form as $$\xi=\omega^{\alpha_1}n_1+\ldots +\omega^{\alpha_k}n_k,$$ noting that $\alpha_k>0$.  Let $\xi'=\omega^{\alpha_1}n_1+\ldots + \omega^{\alpha_k}(n_k-1)$ and $\xi''= \omega^{\alpha_k}$.  For each $\zeta\in [0, \xi'')$, $$\xi_{0, \xi'+\zeta}\oplus \xi_{1, \xi'+\zeta}= \xi'+\zeta+1\in (\xi', \xi),$$  where $\xi'>\zeta+1$ or $\xi'=0$.  If we denote by $\gamma_{0, \zeta}$ the sum of the terms in the Cantor normal form of $\xi_{0, \xi'+\zeta}$ involving $\omega^{\alpha_1}, \ldots, \omega^{\alpha_k}$ and denote by $\delta_{0, \zeta}$ the sum of the remaining, smaller terms, we obtain an expression $\xi_{0, \xi'+\zeta}= \gamma_{0,\zeta}+\delta_{0,\zeta}$.  If we perform a similar decomposition with $\xi_{1, \xi'+\zeta}= \gamma_{1, \zeta}+\delta_{1, \zeta}$, we deduce that $\xi'=\gamma_{0,\zeta}\oplus \gamma_{1, \zeta}$ and $\zeta+1= \delta_{0,\zeta}\oplus \delta_{1, \zeta}$.  

Since $\xi'= \gamma_{0, \zeta}\oplus \gamma_{1, \zeta}$, and since there are only finitely many pairs $(\gamma_0, \gamma_1)$ satisfying $\gamma_0\oplus \gamma_1=\xi'$, there exists a subset $B\subset [0, \xi'')$ with $\sup B=\xi''$ and ordinals $\gamma_0, \gamma_1$ so that $\gamma_0\oplus \gamma_1=\xi'$ and for each $\zeta\in B$, $\gamma_{0, \zeta}=\gamma_0$ and $\gamma_{1, \zeta}=\gamma_1$.  Moreover, since $\zeta+1=\delta_{0,\zeta}\oplus \delta_{1, \zeta}$ for each $\zeta\in B$, we deduce $\sup_{\zeta\in B} \delta_{0, \zeta}\oplus \delta_{1, \zeta}=\xi''$.  But by our remarks above, this means there exists $j\in \{0,1\}$ so that $\sup_{\zeta\in B} \delta_{j, \zeta}=\xi''$.  Let $\xi_j= \gamma_j+\xi''$ and $\xi_{1-j}= \gamma_{1-j}$.  We finish by letting $A= \{\xi'+\zeta: \zeta\in B\}$.

\end{proof}

\section{Coloring lemmas}

In this section, we prove several coloring lemmas for trees.  The first few can be thought of as coloring segments of trees and searching for subsets of our trees which retain tree structure and have monochromatically colored segments.  We have different versions of the same principle because we will apply these arguments to multiple kinds of structures in our Banach spaces (local trees, asymptotic trees, and sequences).  The principle behind the first collection of lemmas is that one of our structures of order $\zeta\xi$ can be thought of as a structure of order $\xi$ built from structures of order $\zeta$. For this reason, the first  lemmas give product estimates.  The remaining lemmas are again two versions of the same principle, but instead of coloring segments of trees, we have maximal members of trees being colored by their predecessors.  This principle results in sum estimates.

\subsection{Product estimates}

Suppose $0<\zeta\in \ord$ is fixed.  If $\ttt$ is a $B$-tree and $f:C(\ttt)\to \rr$ is a function such that for any order preserving $i:\ttt_\zeta\to \ttt$, $$\inf \{f\circ i(c): c\in C(\ttt_\zeta)\}\leqslant 0,$$ then we say $f$ is $\zeta$ \emph{small}. We say $f$ is \emph{strictly} $\zeta$ \emph{small} if this infimum is a minimum.  We note that if $f:C(\ttt)\to \rr$ is $\zeta$ small (resp. strictly $\zeta$ small), $\www$ is a $B$-tree, and $i':\www\to \ttt$ is order preserving, then $f\circ i':C(\www)\to \rr$ is also $\zeta$ small (resp. strictly $\zeta$ small).  This is because if $i:\ttt_\zeta\to \www$ is order preserving, $i'\circ i:\ttt_\zeta\to \ttt$ is order preserving.  For $\varepsilon>0$, there exists $c\in C(\ttt_\zeta)$ with $f\circ i'(i(c))=f(i'\circ i(c))<\varepsilon$ (respectively, $f\circ i'(i(c))=f(i'\circ i(c))\leqslant 0)$.  

\begin{lemma} Fix $0<\zeta$.  Let $\ttt$ be a $B$-tree and $f:C(\ttt)\to \rr$ be $\zeta$ small.  Then if $o(\ttt)\geqslant \zeta\xi$, for any $(\varepsilon_n)\subset (0,1)$, there exists an order preserving $j:\ttt_\xi\to C(\ttt)$ so that for each $t\in \ttt_\xi$, $$f\circ j(t)\leqslant \varepsilon_{|t|}.$$  
Moreover, if $f$ is strictly $\zeta$ small, the conclusion holds with $\varepsilon_n=0$ for all $n\in \nn$.

\label{local coloring}

\end{lemma}

\begin{proof} First, we note that it is sufficient to consider $\ttt=\ttt_{\zeta\xi}$.  This is because we can fix an order preserving $i:\ttt_{\zeta\xi}\to \ttt$ and let $f'=f\circ i:C(\ttt_{\zeta\xi})\to \rr$.  Then if we have the result for $\ttt_{\zeta\xi}$, for $(\varepsilon_n)\subset (0,1)$, we can find $j':\ttt_\xi\to C(\ttt_{\zeta\xi})$ order preserving so that $f'\circ j'(t)\leqslant \varepsilon_{|t|}$.  Then with $j:\ttt_\xi\to C(\ttt)$ defined by $j(t)=i(j'(t))$, $f\circ j(t)= f\circ i\circ j'(t)= f'\circ j'(t)\leqslant \varepsilon_{|t|}$.  We also note that the last part of the lemma follows from a trivial modification of the first part.  

We prove the result on $\ttt_{\zeta\xi}$ by induction on $\xi$.  The $\xi=0$ case is trivial.  Suppose we have the result for $\xi$ and suppose $f:\ttt_{\zeta(\xi+1)}\to \rr$ is $\zeta$ small.  Note that by Proposition \ref{derivation prop}, $o(d^{\zeta\xi}(\ttt_{\zeta(\xi+1)}))=\zeta$.  Fix an order preserving $i':\ttt_\zeta\to d^{\zeta\xi}(\ttt_{\zeta(\xi+1)})\subset \ttt_{\zeta(\xi+1)}$.  There exists $c_1\in C(\ttt_\zeta)$ so that $f(i'(c_1))\leqslant \varepsilon_1$. If $\xi=0$, we are done with $j((1))=i'(c_1)$. Assume $\xi>0$. Let $t_0=\max i'(c_1)\in d^{\zeta\xi}(\ttt_{\zeta(\xi+1)})$.  Again, by Proposition \ref{derivation prop}, $o(\ttt_{\zeta(\xi+1)}(t_0))>\zeta\xi$, whence the $B$-tree $\ttt_{\zeta(\xi+1)}(t_0)\setminus \{\varnothing\}$ has order at least $\zeta\xi$.  Fix $g:\ttt_{\zeta\xi}\to \ttt_{\zeta(\xi+1)}(t_0)\setminus \{\varnothing\}$ order preserving and define $i'':\ttt_{\zeta\xi}\to \ttt_{\zeta(\xi+1)}$ by $i''(t)=t_0\verb!^!g(t)$.  Note that for each $t\in \ttt_{\zeta\xi}$, $t_0\prec i''(t)$, which means that for $c\in C(\ttt_{\zeta\xi})$, $i'(c_1)< i''(c)$.  Note that $f''= f\circ i'':C(\ttt_{\zeta\xi})\to \rr$ is also $\zeta$ small.  Let $\varepsilon''_n= \varepsilon_{n+1}$, and note that the inductive hypothesis guarantees the existence of some $j'':\ttt_\xi\to C(\ttt_{\zeta\xi})$ so that $f''\circ j''(t)\leqslant \varepsilon''_n$.  Define $j:\ttt_{\xi+1}\to C(\ttt_{\zeta(\xi+1)})$ by $j((\xi+1))=i'(c_1)$ and for $t\in \ttt_\xi$, $j((\xi+1)\verb!^!t)= i''\circ j''(t)$. Since $i'(c_1)<i''(c)$ for all $c\in C(\ttt_{\zeta\xi})$, $j$ is order preserving.   Moreover $j((\xi+1))\leqslant \varepsilon_1=\varepsilon_{|(\xi+1)|}$ and for $t\in \ttt_\xi$, $$f\circ j((\xi+1)\verb!^!t)= f\circ i''\circ j''(t)=f''\circ j''(t)\leqslant \varepsilon''_{|t|}=\varepsilon_{|(\xi+1)\verb!^!t|}.$$

Suppose we have the result for every $\eta<\xi$, $\xi$ a limit ordinal. Assume $f:C(\ttt_{\zeta\xi})\to \rr$ is $\zeta$ small.  Note that for each $\eta<\xi$, $\eta+1<\xi$, which means $\zeta(\eta+1)<\zeta\xi$.  For each $\eta<\xi$, fix an order preserving $i_\eta:\ttt_{\zeta(\eta+1)}\to \ttt_{\zeta\xi}$ and let $f_\eta=f\circ i_\eta:C(\ttt_{\zeta(\eta+1)})\to \rr$.  Then $f_\eta$ is $\zeta$ small, which means that for any $(\varepsilon_n)\subset (0,1)$, there exists $j_\eta:\ttt_{\eta+1}\to C(\ttt_{\zeta(\eta+1)})$ so that for each $t\in \ttt_{\eta+1}$, $f_\eta\circ j_\eta(t)\leqslant \varepsilon_{|t|}$.  Define $j:\ttt_\xi\to C(\ttt_{\zeta\xi})$ by letting $j|_{\ttt_{\eta+1}}=i_\eta\circ j_\eta$.

\end{proof}

We also state the following strengthening of the previous result which follows from a simple modification of the preceding proof.  In what follows, we have a similar scenario as above, except we do not have a single function $f$, but rather a collection of functions. The difference lies in the fact that we want the order preserving map to choose $j(t)$ to take a small value on functions which are determined by $j(s)$, where $s\prec t$.

\begin{lemma} Fix $0<\zeta\in \emph{\ord}$, $\xi\in \emph{\ord}$.  Suppose that for each $c\in C(\ttt_{\zeta\xi})\cup \{\varnothing\}$, $f_c:C(\ttt_{\zeta\xi})\to \rr$ is $\zeta$ small and for each $c<c'$, $c,c'\in C(\ttt_{\zeta\xi})\cup \{\varnothing\}$, and $c_1\in C(\ttt_{\zeta\xi})$, $f_c(c_1)\leqslant f_{c'}(c_1)$.  Then for any $(\varepsilon_n)\subset (0,1)$, there exists an order preserving $j:\mt_\xi\to C(\ttt_{\zeta\xi})\cup \{\varnothing\}$ so that $j(\varnothing)=\varnothing$ and so that for each $s\prec t\in \ttt_\xi$, $$f_{j(s)}(j(t))\leqslant \varepsilon_{|t|}.$$  
Moreover, if each $f_c$ is strictly $\zeta$ small, the result holds with $\varepsilon_n=0$ for all $n\in \nn$.

\label{strong coloring}

\end{lemma}

The following is an inessential modification of a result from \cite{Causey}.  The previous lemmas were the local version of this principle, while what follows is the asymptotic version.

\begin{lemma} Suppose $\varnothing \neq \fff, \gggg$ are regular families.  Suppose that for each $c\in  C(\fff)\cup \{\varnothing\}$, $f_c:\fff[\gggg]\to \rr$ is such that \begin{enumerate}[(i)]\item for $c<c'\in C(\fff)\cup \{\varnothing\}$ and any $c_1\in C(\fff[\gggg])$, $f_c(c_1)\leqslant f_{c'}(c_1)$, \item for any $E\in \fff$ and any embedding $i:\widehat{\gggg}\to \fff[\gggg]$, $$\inf\{f_E\circ i(c): c\in C(\gggg)\}=0.$$  \end{enumerate} Then for any sequence $(\varepsilon_n)\subset (0,1)$, there exists an embedding $j:\fff\to C(\fff[\gggg])\cup\{\varnothing\}$ so that $j(\varnothing)=\varnothing$ and for each $ E\prec F\in \fff$, $f_{ j(E)}\circ j(F)\leqslant \varepsilon_{|F|}$.  

If the infimum above is a minimum, then the conclusion holds with $\varepsilon_n=0$ for all $n\in \nn$.  

\label{asymptotic coloring}
\end{lemma}

We also have a mixed version of these results which combines local and asymptotic structures.   

\begin{lemma} Suppose $0<\zeta, \xi<\omega_1$.  Suppose $f_n:C(\ttt_{\zeta\xi})\to \rr$ is a sequence of functions so that \begin{enumerate}[(i)] \item for each $c\in C(\ttt_{\zeta\xi})$ and each $n\in \nn$, $f_n(c)\leqslant f_{n+1}(c)$, \item for each order preserving $i:\ttt_\zeta\to \ttt_{\zeta \xi}$ and each $n\in \nn$, $$\inf\{f_n\circ i(c): c\in C(\ttt_\zeta)\}=0.$$ \end{enumerate} Then for all $\varepsilon_n\downarrow 0$, there exists a regular family $\fff$ with $\iota(\fff)=\xi$ and an order preserving $j:\widehat{\fff}\to C(\ttt_{\zeta\xi})$ so that for all $E\in \widehat{\fff}$, $f_{\max E}(j(E)) \leqslant \varepsilon_{\max E}.$  

\label{mixed coloring}

\end{lemma} \begin{proof} First, we note that if there exists some regular family $\fff$ and an order preserving $j:\widehat{\fff}\to C(\ttt_{\zeta\xi})$ satisfying the conclusions, then for every regular family $\gggg$ with $\iota(\gggg)=\xi$, there exists an order preserving $j':\widehat{\gggg}\to C(\ttt_{\zeta\xi})$ satisfying the conclusions.  This is because if $\iota(\gggg)=\xi=\iota(\fff)$, there exists $M\in \infin$ so that $\gggg(M)\subset \fff$.  Then we can define $j':\widehat{\gggg}\to C(\ttt_{\zeta\xi})$ by $j'(E)= j(M(E))$.  Then for $E\in \widehat{\gggg}$, $$f_{\max E}(j'(E))= f_{\max E}(j(M(E)))\leqslant f_{\max M(E)} (j(M(E)))\leqslant \varepsilon_{\max M(E)}\leqslant \varepsilon_{\max E}.$$  

Of course, we prove the result again by induction on $\xi$ with $\zeta$ held fixed.  If $\xi= 1$, letting $i:\ttt_\zeta\to \ttt_\zeta$ be the identity, the hypotheses guarantees we can find $c_1, c_2, \ldots$ so that $c_n\in C(\ttt_{\zeta})$ and $f_n(c_n)\leqslant \varepsilon_n$.  We define $j:\aaa_1\to C(\ttt_{\zeta\xi})$ by $j((n))= c_n$.  

Next, suppose the result holds for $\xi$.  We can find an order preserving $i:\ttt_\zeta\to d^{\zeta\xi}(\ttt_{\zeta(\xi+1)})$ and $c_1, c_2, \ldots$, $c_n\in C(\ttt_\zeta)$ so that $f_n\circ i(c_n)\leqslant \varepsilon_n$.   Letting $t_n=\max  c_n$, as in the proof of Lemma \ref{local coloring}, we find for each $n$ an order preserving $i_n:\ttt_{\zeta\xi}\to \ttt_{\zeta(\xi+1)}(t_n)\setminus \{\varnothing\}$ and use these maps to define functions $f^n_k$, $k\in \nn$, on $\ttt_{\zeta\xi}$ which satisfy the hypotheses.  Then by the inductive hypothesis, for each $n\in \nn$, we can find a regular family $\fff_n$ with $\iota(\fff_n)= \xi$ and an order preserving $j_n:\widehat{\fff}_n\to C(\ttt_{\zeta\xi})\to C(\ttt_{\zeta(\xi+1)})$, where the second function takes the segment $c\in C(\ttt_{\zeta\xi})$ to $\{t_n\verb!^!i_n(t):t\in c\}$.  But by our first remark of the proof, we can assume that $\fff_n= \fff_1$ for all $n\in \nn$.  We let $\fff= (\aaa_1, \fff_1)$ and define $j:\widehat{\fff}\to C(\ttt_{\zeta(\xi+1)})$ by letting $j((n))=i(c_n)$ and $j(n\verb!^!E)= \{t_n\verb!^!t: t\in j_n(E)\}$.  One easily checks that this function satisfies the requirements.  Since $\iota(\fff)=\iota(\fff_1)+\iota(\aaa_1)=\xi+1$, this finishes the successor case.  

Assume the result holds for each ordinal less than $\xi$, where $\xi$ is a countable limit ordinal.  Choose $\xi_n\uparrow \xi$ arbitrary.  For each $n\in \nn$, we can define an order preserving $i_n:\ttt_{\zeta\xi_n}\to \ttt_{\zeta\xi}$ and use these maps to define functions $f^n_k:C(\ttt_{\zeta\xi_n})\to \rr$, $k\in \nn$ which also satisfy the hypotheses.  Then for each $n\in \nn$, we can find a regular family $\fff_n$ with $\iota(\fff_n)=\xi_n$ and an order preserving $j_n:\widehat{\fff}_n\to C(\ttt_{\zeta\xi_n})\to C(\ttt_{\zeta\xi})$ satisfying the requirements.  We can recursively choose $M_1\supset M_2\supset\ldots$, $M_n\in \infin$ so that with $\gggg_n=\{E\in \fin: M_n(E)\in \fff_n\}$, $\gggg_1\subset \gggg_2\subset \ldots$.  This is because if $\fff$ is any regular family and $M\in \infin$,  $\gggg=\{E\in \fin: M(E)\in \fff\}$ is regular with $\iota(\fff)=\iota(\gggg)$ \cite{Causey}.  We let $M_1=\nn$ so $\gggg_1= \fff_1$.  Since $\iota(\fff_1)<\iota(\fff_2)$, we can find $M_2\in [M_1]$ so that $\fff_1(M_2)\subset \fff_2$.  We then let $\gggg_2= \{E\in \fin: M_2(E)\in \fff_2\}$, so $\gggg_1\subset \gggg_2$.  Next, since $ \iota(\gggg_2)< \iota(\fff_3)$, we can find $M_3\in \infin$ so that $\gggg_2(M_3)\subset \fff_3$.  If $\gggg_3=\{E\in \fin: M_3(E)\in \fff_3\}$, $\gggg_2\subset \gggg_3$, and so on.  Thus by renaming, we can assume without loss of generality that $\fff_1\subset \fff_2\subset \ldots$.  We let $\fff=\{E\in \fin: \exists n\leqslant E\in \fff_n\}=\{E\in \fin: E\in \fff_{\min E}\}$.  Note that $\iota(\fff)=\sup \xi_n=\xi$.  We define $j:\widehat{\fff}\to C(\ttt_{\zeta\xi})$ by letting $j(E)= j_{\min E}(E)$.  This function is easily seen to satisfy the requirements.

\end{proof}

Finally, we have the simplest version of these results, which we will used on structures determined sequences rather than trees.  

\begin{lemma} Fix $\xi<\omega_1$.  If $f:\widehat{\sss}_{\omega^\xi}\to \{0,1\}$ is any function, then either there exists $M\in \infin$ so that for all $E\in \widehat{\sss}_{\omega^\xi}$, $f(M(E))=1$, or there exists a sequence $E_1<E_2<\ldots$ so that $f(E_i)=0$ for all $i\in \nn$ and for all $E\in \sss_{\omega^\xi}$, $\cup_{i\in E} E_i\in \sss_{\omega^\xi}$.

\label{easy coloring}
\end{lemma}

\begin{proof} Let $\zeta_n=\xi_n+1\uparrow \omega^\xi$ be the sequence used to define $\sss_{\omega^\xi}$ (we replace $\sss_{\zeta_n}$ with $\aaa_n$ in the case that $\xi=0$).  We consider two cases.  In the first case, for each $n\in \nn$ and for each $N\in \infin$, there exists $N'\in [N]$ so that for each $n\leqslant E\in \sss_{\zeta_n}$, $f(N'(E))=1$.  In this case, let $N_0=\nn$ and choose $N_1, N_2, \ldots$ so that $N_{n+1}\in [N_n]$ and so that for each $n\leqslant E\in \sss_{\zeta_n}$, $f(N_n(E))=1$.  Write $N_n=(m^n_i)$ and let $m_n= m^n_n$, $M=(m_n)$.  Then $M$ is easily seen to satisfy the conclusions.  This is because for $n\leqslant E\in \sss_{\zeta_n}$, $M(E)= N_n(F)$ for some spread $F$ of $E$.  Thus $n\leqslant F\in \sss_{\zeta_n}$ and $f(M(E))=f(N_n(F))= 1$.  

In the second case, there exist $n\in \nn$ and $N\in \infin$ so that for each $N'\in [N]$, there exists $n\leqslant E\in \sss_{\zeta_n}$ with $f(N'(E))=0$.  First, choose $L\in [N]$ so that $\sss_{\omega^\xi}[\sss_{\zeta_n}](L)\subset \sss_{\zeta_n+\omega^\xi}= \sss_{\omega^\xi}$.  Note that for any $s\in \nn$, there exists $s\leqslant F\in \sss_{\zeta_n}$ so that $f(L(F))=0$.  To see this, let $L=(l_k)$ and let $L'=(l_k')=(l_{s+k})\in [N]$.  By hypothesis, there exists $n\leqslant E\in \sss_{\zeta_n}$ so that $f(L'(E))= 0$.  But $$L'(E)= (l'_k: k\in E)= (l_{k+s}: k\in E)= (l_k: k\in F)=L(F),$$ where $F=(k+s: k\in E)$.  Note that $s\leqslant F\in \sss_{\zeta_n}$, since $F$ is a spread of $E$.  This means we can choose $F_1<F_2<\ldots$, $F_i\in \sss_{\zeta_n}$ so that $f(L(F_i))=0$.  Let $E_i=L(F_i)$.  For $E\in \sss_{\omega^\xi}$, $(\min F_i)_{i\in E}$ is a spread of $E$, so that $(\min F_i)_{i\in E}\in \sss_{\omega^\xi}$.  This means $$\cup_{i\in E} E_i = L\bigl(\cup_{i\in E}F_i\bigr)\in \sss_{\omega^\xi}[\sss_{\zeta_n}](L)\subset \sss_{\omega^\xi}.$$

\end{proof}

\subsection{Sum estimates}

Given a well-founded $B$-tree $\ttt$ and $t\in \ttt$, we let $\eee_t(\ttt)=\{s\in MAX(\ttt): t\preceq s\}$.  We say $(\ccc^0_t, \ccc^1_t)_{t\in\ttt}$ is a \emph{coloring} of $\ttt$ provided that for each $t\in \ttt$, $\eee_t(\ttt)=\ccc^0_t\cup \ccc^1_t$.  For $j\in \{0,1\}$, we say a coloring is \emph{monochromatically} $j$ provided that for each $t\in MAX(\ttt)$, $$t\in \bigcap_{i=1}^{|t|} \ccc^j_{t|_i}.$$  For well-founded $B$-trees $\ttt, \www$, we say a pair $(i,e)$, $i:\ttt\to\www$ and $e:MAX(\ttt)\to MAX(\www)$, is an \emph{extended order preserving map} if $i$ is order preserving and for each $t\in MAX(\ttt)$, $i(t)\preceq e(t)$.  We say an extended order preserving map $(i,e)$ is an \emph{extended embedding} if $i$ is an embedding.  If $(\ccc^0_w, \ccc^1_w)_{w\in \www}$ is a coloring of $\www$ and $(i,e)$ is an extended order preserving map, then $$\ddd^j_t=\{s\in MAX(\ttt): e(s)\in \ccc^j_{i(j)}\}$$  defines a coloring $(\ddd^0_t, \ddd^1_t)_{t\in \ttt}$, which we call the \emph{induced coloring}.  Of course, this coloring depends on $i$, $e$, and $(\ccc^0_w, \ccc^1_w)$, but we will omit reference to the coloring and extended order preserving map inducing the coloring when no confusion will arise.  It is clear that if $(\ccc^0_w, \ccc^1_w)$ is monochromatically $j$, then any coloring induced by this coloring is also monochromatically $j$. 

We also note that if $i:\ttt\to \www$ is any order preserving map (resp. embedding), and if $\www$ is well-founded, there exists $e:MAX(\ttt)\to MAX(\www)$ so that $(i,e)$ is an extended order preserving map (resp. extended embedding).  

\begin{lemma} Suppose $0<\xi\in \emph{\ord}$.  Suppose $(\ccc^0, \ccc^1)\subset MAX(\ttt_\xi)$ is such that $\ccc^0\cup \ccc^1=MAX(\ttt_\xi)$.  Then there exists an extended order preserving map $(i,e)$ of $\ttt_\xi$ into $\ttt_\xi$ and $j\in \{0,1\}$ so that $e(\ttt_\xi)\subset \ccc^j$.    

\label{successor step}

\end{lemma}

This lemma should be compared with the result found in \cite{PR} which states that if $\fff$ is regular, and if we color the maximal members of $\fff$ with two colors, there exists $M\in \infin$ so that $\fff\cap [M]^{<\omega}$ is monochromatic. In fact, they prove a stronger result where the set of maximal members of a regular family can be replaced by any family $\fff\subset \fin$ so that no member of $\fff$ is a subset of another member of $\fff$.  Such families are called \emph{thin}.

\begin{lemma} If $\ttt$ is any well-founded $B$ tree on any set, and if $(\ccc_t^0, \ccc_t^1)_{t\in \ttt}$ is any coloring of $\ttt$, then for $j\in \{0,1\}$, there exist an ordinal $\xi_j$ and an extended order preserving map $(i_j, e_j)$ of $\ttt_{\xi_j}$ into $\ttt$ so that the induced coloring on $\ttt_{\xi_j}$ is monochromatically $j$.  Moreover, $\xi_0, \xi_1$ can be chosen so that $\xi_0\oplus \xi_1=o(\ttt)$.  

\label{coloring sums}

\end{lemma}

Here, it should be understood that if $\xi_j=0$, we take $i_j$ and $e_j$ to be the empty maps, and that the empty map induces a monochromatically $j$ coloring on the empty set for both $j=0$ and $1$.

The corresponding result for regular families was shown in \cite{Causey}.  

\begin{lemma} If $\fff$ is a regular family, and if $(\ccc^0_E, \ccc^1_E)_{E\in \fff}$ is any coloring of $\fff$, then for $j\in \{0,1\}$, there exist an ordinal $\xi_j<\omega_1$, a regular family $\fff_j$ with $\iota(\fff_j)=\xi_j$, and an extended embedding $(i_j, e_j)$ of $\fff_j$ into $\fff$ so that the induced coloring on $\fff_j$ is monochromatically $j$.  Moreover, $\xi_0, \xi_1$ can be chosen so that $\xi_0\oplus \xi_1=\iota(\fff)$.

\end{lemma}

\begin{proof}[Proof of Lemma \ref{successor step}] For $\xi=1$, the result is trivial.  Suppose we have the result for $\xi$.  Suppose $\ccc^0\cup \ccc^1=MAX(\ttt_{\xi+1})$.  Define $\ddd^j=\{t\in \ttt_\xi: (\xi+1)\verb!^!t\in \ccc^j\}$.  Then $MAX(\ttt_\xi)= \ddd^0\cup \ddd^1$.  Take an extended embedding $(i',e')$ of $\ttt_\xi$ into $\ttt_\xi$ and $j\in \{0,1\}$ so that $e'(MAX(\ttt_\xi))\subset \ddd^j$.  Define $i((\xi+1))=(\xi+1)$, $i((\xi+1)\verb!^!t)= (\xi+1)\verb!^!i'(t)$ for $t\in \ttt_\xi$, and  $e((\xi+1)\verb!^!t)=(\xi+1)\verb!^!e'(t)$ for $t\in MAX(\ttt_\xi)$.  

Assume we have the result for every $\zeta<\xi$, $\xi$ a limit ordinal. For $\zeta<\xi$ and $j\in \{0,1\}$, let $\ccc^j_\zeta=\{t\in MAX(\ttt_{\zeta+1}): t\in \ccc^j\}$.  Find an extended order preserving map $(i_\zeta, e_\zeta)$ of $\ttt_{\zeta+1}$ into itself and $j_\zeta\in \{0,1\}$ so that $e_\zeta(MAX(\ttt_{\zeta+1}))\subset \ccc^{j_\zeta}_\zeta$.  For $j\in \{0,1\}$, let $A_j=\{\zeta<\xi: j_\zeta=j\}$.  Since $[0,\xi)=A_0\cup A_1$, there is $j\in\{0,1\}$ so that $\sup A_j= \xi$.  Thus we can choose $\phi:[0,\xi)\to A_j$ so that $ \zeta<\phi(\zeta)$ for all $\zeta<\xi$.  Then for each $\zeta<\xi$, we can find an extended order preserving map $(g_\zeta, h_\zeta)$ of $\ttt_{\zeta+1}$ into $\ttt_{\phi(\zeta)+1}$.  Define $i$ on $\ttt_{\zeta+1}$ and $e$ on $MAX(\ttt_{\zeta+1})$ by letting $$i|_{\ttt_{\zeta+1}}=  i_{\phi(\zeta)}\circ g_\zeta, \hspace{5mm} e|_{MAX(\ttt_{\zeta+1})} = e_{\phi(\zeta)}\circ h_\zeta.$$

\end{proof}

\begin{proof}[Proof of Lemma \ref{coloring sums}] If $\xi=1$, the assertion is trivial.  

Assume we have the result for an ordinal $\xi$ and suppose first that $MAX(\ttt_{\xi+1})\subset \mathcal{C}^0_{(\xi+1)}$.  We then apply the inductive hypothesis to the coloring $(\ccc^0_{(\xi+1)\verb!^!t}, \ccc^1_{(\xi+1)\verb!^!t})_{t\in \ttt_\xi}$ to deduce the existence of $\xi_j$, $(i'_j, e'_j):\ttt_{\xi_j}\to \ttt_\xi\to \ttt_{\xi+1}$, where the map from $\ttt_\xi$ to $\ttt_{\xi+1}$ is given by $t\mapsto (\xi+1)\verb!^!t$, so that the induced coloring on $\ttt_{\xi_j}$ is monochromatically $j$ and so that $\xi_0\oplus \xi_1= \xi$.  We define $(i_1, e_1):\ttt_{\xi_1}\to \ttt_{\xi+1}$ by simply letting $i_1=i_1'$ and $e_1=e_1'$.  We define $(i_0, e_0):\ttt_{\xi_0+1}\to \ttt_{\xi+1}$ by $$i_0((\xi_0+1))=(\xi+1),$$ $$i_0((\xi_0+1)\verb!^!t)= i_0'(t),$$ $$ e_0((\xi_0+1)\verb!^!t)=e_0'(t).$$  It is easy to see that these define extended order preserving maps, and that the coloring induced by $(i_1, e_1)$ is monochromatically $1$.  To see that the coloring on $\ttt_{\xi_0+1}$ induced by $(i_0, e_0)$ is monochromatically $0$, take $s\in MAX(\ttt_{\xi_0+1})$ and write $s= (\xi_0+1)\verb!^!t$.  Then by the properties of $(i_0', e_0')$, $$e_0(s)=e_0((\xi_0+1)\verb!^!t)=e_0'(t)\in \bigcap_{k=1}^{|t|} \ccc^0_{i_0'(t|_k)}=\bigcap_{k=2}^{|s|} \ccc^0_{i_0(s|_k)}.$$  But $e_0(s)\in MAX(\ttt_{\xi+1})\subset \ccc^0_{(\xi+1)}= \ccc^0_{i_0(s|_1)}$, so that $e_0(s)\in \bigcap_{k=1}^{|s|} \ccc^0_{i_0(s|_k)}$.  Since $(\xi_0+1)\oplus \xi_1=\xi_0\oplus \xi_1+1= \xi+1$, this finishes the proof in this special case.  How to complete the proof in the special case that $MAX(\ttt_{\xi+1})\subset \ccc^1_{(\xi+1)}$ is similar.  

For the general successor case, we simply reduce to one of these special cases by Lemma \ref{successor step}. With $\ccc^0= \ccc^0_{(\xi+1)}$ and $\ccc^1=\ccc_{(\xi+1)}^1$, we can find an extended order preserving map $(i', e'):\ttt_{\xi+1}\to \ttt_{\xi+1}$ so that $e'(MAX(\ttt_{\xi+1}))\subset \ccc^j$ for either $j=0$ or $j=1$.  We let $(\ddd^0_t, \ddd^1_t)_{t\in \ttt_{\xi+1}}$ be the coloring induced by $(i',e')$ and $(\ccc^0_t, \ccc^1_t)_{t\in \ttt_{\xi+1}}$.  Then the coloring $(\ddd^0_t, \ddd^1_t)_{t\in \ttt_\xi}$ is one of the special cases above.  We find $\xi_0, \xi_1$, and $(i''_j, e''_j):\ttt_{\xi_j}\to \ttt_{\xi+1}$ to satisfy the requirements with respect to $(\ddd^0_t, \ddd^1_t)_{t\in \ttt_{\xi+1}}$, and then define $i_j= i'\circ i''_j$, $e_j=e'\circ e''_j$.

Suppose we have the result for every ordinal less than the limit ordinal $\xi$.  If $(\ccc^0_t, \ccc^1_t)_{t\in \ttt_\xi}$ is a coloring, then for each $\zeta<\xi$, $(\ccc^0_t, \ccc^1_t)_{t\in \ttt_{\zeta+1}}$ is a coloring of $\ttt_{\zeta+1}$.  Then we can find $\xi_{j, \zeta}$ and extended order preserving maps $(i_{j, \zeta}, e_{j, \zeta}):\ttt_{\xi_{j, \zeta}}\to \ttt_{\zeta+1}\to \ttt_\xi$ so that the coloring induced on $\ttt_{\xi_{j, \zeta}}$ is monochromatically $j$, and so that $\xi_{0, \zeta}\oplus \xi_{1, \zeta}=\zeta+1$.  Then by Proposition \ref{splitting technique}, there exist a set $A\subset [0, \xi)$, ordinals $\xi_0, \xi_1$ so that $\xi_0\oplus \xi_1=\xi$, and $j\in \{0,1\}$ so that $\xi_j$ is a limit ordinal and $$\sup_{\zeta\in A} \xi_{j, \zeta}=\xi_j, \hspace{5mm} \min_{\zeta\in A}\xi_{1-j, \zeta}\geqslant \xi_{1-j}.$$  Without loss of generality, assume $j=0$.  

Fix any $\phi:[0, \xi_0)\to A$ so that $\xi_{0,\phi(\zeta)}>\zeta+1$.  Then for each $\zeta<\xi_0$, we can define an extended order preserving map $(i'_\zeta, e'_\zeta):\ttt_{\zeta+1}\to \ttt_{\xi_{0,\phi(\zeta)}}$.  Then $(i_{0, \phi(\zeta)}\circ i'_\zeta, e_{0, \phi(\zeta)}\circ e'_\zeta):\ttt_{\zeta+1}\to \ttt_\xi$ defines an extended order preserving map which is monochromatically $0$.  We define $i_0$ and $e_0$ on $\ttt_{\xi_0}$ by letting the restriction to $\ttt_{\zeta+1}$ be these compositions, which defines an extended order presering $(i_0, e_0):\ttt_{\xi_0}\to \ttt_\xi$ so that the induced coloring is monochromatically $0$.  

Choose any $\zeta\in A$ and choose an extended order preserving map $(i', e'):\ttt_{\xi_1}\to \ttt_{\xi_{1,\zeta}}$.  Defining $(i_1, e_1)= ( i_{1, \zeta}\circ i', e_{1, \zeta}\circ e'):\ttt_{\xi_1}\to \ttt_\xi$ gives an extended order preserving map from $\ttt_{\xi_1}\to \ttt_\xi$ such that the induced coloring is monochromatically $1$.  Since $\xi_0\oplus \xi_1=\xi$, this finishes the proof.

\end{proof}

Of course, we can now apply these results to colorings with any finite number of colors rather than simply two colors.  Moreover, if we have any finite coloring $(\ccc^0_t, \ldots, \ccc^n_t)_{t\in \ttt_{\omega^\xi}}$, we obtain $\xi_0, \ldots, \xi_n$ so that $\xi_0\oplus\ldots \oplus \xi_n= \omega^\xi$ and extended order preserving maps inducing monochromatic colorings.  But as we have already discussed, if $\xi_0\oplus \ldots \oplus \xi_n= \omega^\xi$, there exists $0\leqslant i\leqslant n$ so that $\xi_i=\omega^\xi$, thus in this case we obtain a monochromatic structure of the same size.  The same holds with colorings of regular families $\fff$ with $\iota(\fff)=\omega^\xi$ for some $\xi$, and in particular for the Schreier families.  

We wish to discuss a particular kind of coloring, which is that in which every member $t$ of $\ttt_\xi$ colors all of the members of $\eee_t(\ttt_\xi)$ with the same color.  This is simply the case of a function $f:\ttt_\xi\to \{0, \ldots, n\}$.  In this case, $\ccc^i_t= \varnothing$ if $i\neq f(t)$ and $\ccc^{f(t)}_t= \eee_t(\ttt_\xi)$.  In this case, the function $e$ plays no real part in the result.  That is, if we have $(i_j, e_j):\ttt_{\xi_j}\to \ttt$ inducing a monochromatically $j$ coloring on $\ttt_{\xi_j}$, and if $e_j':MAX(\ttt_{\xi_j})\to MAX(\ttt)$ is any function so that $(i_j, e_j')$ is also an extended order preserving map, $(i_j, e_j')$ also induced a monochromatically $j$ coloring.

\begin{corollary} If $S$ is any finite set, $\xi\in \emph{\ord}$ and if $f:\ttt_{\omega^\xi}\to S$ is any function, then there exists an order preserving $i:\ttt_{\omega^\xi}\to \ttt_{\omega^\xi}$ so that $f\circ i$ is constant.  The same holds if $\xi<\omega_1$ and if we replace $\ttt_{\omega^\xi}$ with $\sss_\xi$ and order preserving with embedding.  

\label{sum corollary}

\end{corollary}

\begin{corollary} If $0<a<b$ and if $f:\ttt_{\omega^\xi}\to [a,b]$ is any function, then for any $\delta>0$, there exists $\theta\in [a,b]$ and an order preserving $i:\ttt_{\omega^\xi}\to \ttt_{\omega^\xi}$ so that for all $t\in \ttt_{\omega^\xi}$, $$\theta\leqslant f\circ i(t)\leqslant (1+\delta)\theta.$$  

\label{approximate monochromatic}

\end{corollary}

\begin{proof} Choose $n\in \nn$ so that $(b/a)^{1/n}<1+\delta$.  For $t\in \ttt_{\omega^\xi}$, let $g(t)=j$ if $$a(1+\delta)^{j-1} \leqslant f(t)< a(1+\delta)^j.$$  This defines a finite coloring of $\ttt_{\omega^\xi}$.  If $i:\ttt_{\omega^\xi}\to \ttt_{\omega^\xi}$ is order preserving so that $g\circ i\equiv j$, then taking $\theta= a(1+\delta)^{j-1}$ gives the result.

\end{proof}

\section{Structures in Banach spaces}

\subsection{Coordinate systems}

We let $\textbf{Ban}$ denote the class of Banach spaces, and $\textbf{SB}$ the class of separable Banach spaces.  If $X$ is a Banach space, we let $S_X$, $B_X$ denote the unit sphere and unit ball of $X$, respectively.  If $A\subset X$, we let $[A]$ denote the closed span of $A$.  If $X$ is a Banach space, a sequence of finite dimensional subspaces $E=(E_n)$ of $X$ is called a \emph{finite dimensional decomposition} (or FDD) for $X$ provided that for each $x\in X$, there exists a unique sequence $(x_n)\subset X$ so that $x_n\in E_n$ for all $n\in \nn$ and so that $x=\sum x_n$.  In this case, the map defined by $x=\sum x_n\mapsto x_m$ is a well-defined, bounded, linear operator, called the $m^{th}$ \emph{canonical projection} and denoted $P^E_m$.  If $A\in \fin$, we let $P^E_A= \sum_{n\in A}P^E_n$.  We note that the principle of uniform boundedness implies that the \emph{projection constant of} $E$ \emph{in} $X$, defined by $\sup_{m\leqslant n}\|P^E_{[m,n]}\|$, is finite.  If the projection constant of $E$ in $X$ is $1$, we say $E$ is \emph{bimonotone} in $X$.  It is known that if $E$ is an FDD for $X$, one can equivalently renorm $X$ to make $E$ bimonotone in $X$ with respect to the new norm.  

If $x\in X$ and $E$ is an FDD for $X$, we define the \emph{support} of $x$, denoted $\supp_E(x)$, to be $\{n\in \nn: P^E_nx\neq 0\}$.  We define the \emph{range} of $x$ to be the smallest interval in $\nn$ containing $\supp_E(x)$, and denote the range of $x$ by $\ran_E(x)$.  We let $c_{00}(E)$ denote all vectors in $X$ having finite support. We say a sequence $(x_n)\subset c_{00}(E)$ of non-zero vectors in $X$ is a \emph{block sequence with respect to} $E$ provided $\supp_E(x_n)<\supp_E(x_{n+1})$ for all $n\in \nn$.  We let $\Sigma(E,X)$ denote the finite block sequences in $X$ with respect to $E$.  

We note that for each $n\in \nn$, $E_n^*$ can be embedded into $X^*$ by $(P_n^E)^*$, but this embedding need not be isometric unless $E$ is bimonotone.  We will consider $E_n^*$ as a subspace of $X^*$ and let $E^*=(E_n^*)$.  We let $[E^*]=[E_n^*:n\in \nn]$.  We note that $E^*$ is always an FDD for $[E^*]$ with projection constant in $[E^*]$ not exceeding the projection constant of $E$ in $X$.  We say $E$ is \emph{shrinking} if $[E^*]=X^*$.  We say $E$ is \emph{boundedly complete} if $E^*$ is a shrinking FDD for $[E^*]$, in which case $X=[E^*]^*$.  These facts can be found in \cite{FHHMZ}.

If $(e_i)$, $(f_i)$ are sequences of the same length in (possibly different) Banach spaces, we say $(f_i)$ $C$-dominates $(e_i)$, or that $(e_i)$ is $C$-dominated by $(f_i)$, if for all $(a_i)\in c_{00}$, $$\Bigl\|\sum a_ie_i\Bigr\| \leqslant C\Bigl\|\sum a_i f_i\Bigr\|.$$  We say that two sequences $(e_i), (f_i)$ are $C$-equivalent if there exist $a,b>0$ so that $ab\leqslant C$ and $(e_i)$ $a$-dominates and is $b$-dominated by $(f_i)$.  We say the sequence $(e_i)$ is $b$-\emph{basic} if for all $1\leqslant m\leqslant n$ and $(a_i)_{i=1}^n$, $$\bigl\|\sum_{i=1}^m a_ie_i\bigr\|\leqslant b\bigl\|\sum_{i=1}^n a_ie_i\bigr\|.$$  We say a sequence is \emph{basic} if it is $b$-basic for some $b\geqslant 1$, and the smallest $b$ for which $(e_i)$ is $b$-basic is called the \emph{basis constant} of $(e_i)$.

Let $X$ be a Banach space and $K\geqslant 1$.  For $1\leqslant p\leqslant \infty$, we let $$T_p(X, K)=\Bigl\{(x_i)_{i=1}^n\in B_X^{<\omega}: \forall (a_i)_{i=1}^n \in S_{\ell_p^n}, K^{-1}\leqslant \bigl\|\sum_{i=1}^n a_ix_i\bigr\|\leqslant 1\Bigr\}.$$  We note that this is a closed, $p$-absolutely convex tree on $X$ which is ill-founded if and only if $X$ admits a sequence which is $K$-equivalent to the unit vector basis of $\ell_p$ (resp. $c_0$ if $p=\infty$).  If $E$ is an FDD for $X$, we let $T_p(X, E, K)=T_p(X, K)\cap \Sigma(E, X)$. We note that this tree is not closed, but it is $p$-absolutely convex.  We also define $$W(X,K)=\Bigl\{(x_i)_{i=1}^n\in X^{<\omega}: \|x_i\|\geqslant 1, \forall (a_i)_{i=1}^n\in S_{\ell_\infty^n}, \bigl\|\sum_{i=1}^n a_ix_i\bigr\|\leqslant K\Bigr\},$$  $W(X, E, K)=W(X,K)\cap \Sigma(E, X)$.  

In the case that $E$ has an FDD, we define a second type of derivation on trees $T$ by $$d_E(T)=\Bigl\{(x_i)_{i=1}^n\in T: \exists (y_i), i\leqslant \supp_E(y_i), (x_1, \ldots, x_n, y_i)\in T \text{\ \ }\forall i\in \nn\Bigr\}.$$  We let $$d_E^0(T)=T,$$ $$d_E^{\xi+1}(T)=d_E(d^\xi_E(T)),$$ and if $\xi$ is a limit ordinal, we let $$d^\xi_E(T)=\bigcap_{\zeta<\xi} d^\zeta_E(T).$$  We let $o_E(T)=\min \{\xi\in \ord: d_E^\xi(T)=\varnothing\}$.  We think of this derivation as being an asymptotic derivation, while the usual derivation is a local one.  

We make the following definitions: \begin{enumerate}[(i)]\item $I_p(X,K)=o(T_p(X,K))$, \item $I_p(X,E,K)=o(T_p(X,E,K))$, \item $I^a_p(X,E,K)=o_E(T_p(X,E,K))$, \item $J(X,K)=o(W(X,K))$, \item $J(X,E,K)=o(W(X,E,K))$, \item $J^a(X, E, K)=o_E(W(X, E, K))$. \end{enumerate}

We let $I_p(X)=\sup_{K\geqslant 1}I_p(X,K)$, and define $I_p(X,E)$, $I^a_p(X,E)$, $J(X)$, $J(X,E)$, and $J^a(X,E)$ similarly.  We note that $I_p(\cdot)$ is originally due to Bourgain \cite{Bo}, and is called the Bourgain $\ell_p$ index of $X$.  The block indices using a basis were defined in \cite{JO}, and the asymptotic block derivative $d_E$ was define in \cite{OSZ}. We remark that $I_p(X)>\omega$ if and only if $\ell_p$ (resp. $c_0$) is finitely representable in $X$, and $I_p(X, E)>\omega$ if and only if $\ell_p$ (resp. $c_0$) is block finitely representable in $E$. Also, $I^a_p(X,E)>\omega$ if and only if for each $n\in \nn$, $\ell_p^n$ is in the $n^{th}$ asymptotic structure of $X$ determined by the filter of tail subspaces $[E_n:n\geqslant m]$, $m\in \nn$, of $X$.  For definitions regarding asymptotic structures, we refer the reader to \cite{MMT}.  

We remark that there are other coordinate systems of interest, some of which need not be ordered or countable.  For example, Markushevich finite dimensional decompositions, unconditional FDDs, or skipped block decompositions.  One can define skipped block trees or trees consisting of vectors with finite, disjoint support and formulate many of the results here for such coordinate systems.  However, since the difference betwee $I_p$ and the asymptotic index $I^a_p$, which has as an analogue the pointwise null index in the case of an unordered coordinate system, can never be very different, we do not explicitly state the results for each possible coordinate system.

We remark that by Proposition 5 of \cite{OSZ}, for any $\xi<\omega_1$, and any Banach space $X$ with FDD $E$, $I^a_p(X,E,K)>\omega^\xi$ if and only if there exists $(x_F)_{F\in \widehat{\sss}_\xi}\subset B_X$ so that \begin{enumerate}[(i)]\item for each $F\in \widehat{\sss}_\xi$, $(x_{F|_i})_{i=1}^{|F|}\in T_p(X,E,K)$, and \item for $F\in \sss_\xi'$ and $i>F$, $i\leqslant \supp_E(x_{F\verb!^!i})$. \end{enumerate}

More generally, if $\fff$ is a regular family and $(x_F)_{F\in \widehat{\fff}}$ is such that $(x_{F|_i})_{i=1}^{|F|}\in T_p(X,E,K)$ for each $F\in \widehat{\fff}$ and if $m_i\leqslant \supp_E(x_{F\verb!^!m_i})$ for each $F\in \fff'$, where $(m_i)_{i\in \nn}=(m\in \nn: F\verb!^!m\in \nn)$, then $I^a_p(X,E,K)>\iota(\fff)$.  Moreover, there exists a collection called the \emph{fine Schreier families} \cite{OSZ}, so that the $\iota$ index of the $\xi^{th}$ family is $\xi$ and if  $I^a_p(X, E, K)>\xi$, we can find a collection $(x_F)_{F\in \widehat{\fff}}\subset X$ satisfying (i) and (ii) above where $\fff$ is the $\xi^{th}$ fine Schreier family.  Since for any regular $\fff, \gggg$ with $\iota(\fff)=\iota(\gggg)$, there exists $M\in \infin$ so that $\fff(M)\subset \gggg$ and $\gggg(M)\subset \fff$, the exact regular family used as the index set is not important, only its Cantor-Bendixson index.  For this reason, we do not discuss the fine Schreier families.  

Note that by standard pruning arguments, such as those detailed in \cite{AJO}, if $I^a_p(X, E, K)>\xi$, we can find a regular family $\fff$ with $\iota(\fff)=\xi$ and a collection $(x_F)_{F\in \widehat{\fff}}\subset X$ so that for each $F\in \widehat{\fff}$, $(x_{F|_i})_{i=1}^{|F|}\in T_p(X,E,K)$, $x_F\in c_{00}(E)$, and so that for $F\in \fff'$, $(x_{F\verb!^!i})_{F<i}$ is a block sequence with respect to $E$. We will call such a collection an \emph{asymptotic block tree}. Note that all of the remarks here apply as well to the index $J^a$.

We also define higher order $\ell_p$ and $c_0$ spreading models.  For a regular family $\fff\subset \fin$ and $1\leqslant p<\infty$, we say $(x_n)$ is a $K$-$\ell_p^\fff$ spreading model provided that for each $E\in \fff$, $(x_n)_{n\in E}\in T_p(X,K)$. Note that the usual definition of a $K$-$\ell_p^\fff$ spreading model does not require that the sequence $(x_i)_{i\in E}$ $K$-dominates and is $1$-dominated by the $\ell_p^{|E|}$ basis for each $E\in \fff$.  Typically one requires the existence of constants $a,b>0$ with $ab\leqslant K$ so that for each $E\in \fff$, $(x_i)_{i\in E}$ $a$-dominates and is $b$-dominated by the $\ell_p^{|E|}$ basis.  Our definition is simply the usual definition with a normalization. The notion of $K$-$c_0^\fff$ spreading model is defined similarly.  If $\fff=\sss_\xi$, we write $\ell_p^\xi$ (resp. $c_0^\xi$) in place of $\ell_p^{\sss_\xi}$ (resp. $c_0^{\sss_\xi})$.  Note that, with our convention that $\sss_{\omega_1}=\fin$, $(x_n)$ is a $K$-$\ell_p^{\omega_1}$ spreading model if and only if it is $1$-dominated by and $K$-dominates the $\ell_p$ basis.

We also define a notion related to $c_0^\fff$ spreading models which will be used later.  We say $(x_n)\subset X$ is a $K$-$c_0^\fff$ \emph{special sequence} provided that for each $E\in \fff$, $(x_n)_{n\in E}\in W(X, K)$.  As with spreading models, we write $K$-$c_0^\xi$ special sequence in place of $K$-$c_0^{\sss_\xi}$ special sequence.  As we dicuss below, if $\fff$ contains all singletons and $\iota(\fff)\geqslant \omega$, any $K$-$c_0^\fff$ special sequence $(x_n)$ must be seminormalized and weakly null.  Therefore some subsequence of $(x_n)$ is $3/2$-basic, which means that this subsequence $3$-dominates the $c_0$ basis.  Therefore with these assumptions on $\fff$, any $K$-$c_0^\xi$ special sequence $(x_n)$ admits a subsequence $(y_n)$ so that $(K^{-1}y_n)$ is a $3K$-$c_0^\fff$ spreading model.

We define the following indices associated with containment of higher order spreading models.  For $1\leqslant p< \infty$ and $K\geqslant 1$, we let $$\ssm_p(X,K)=\min \{\xi\leqslant \omega_1: X\text{\ admits no\ }K\text{-}\ell_p^\xi \text{\ spreading model}\},$$  and $\ssm_p(X)=\sup_{K\geqslant 1}\ssm_p(X,K)$. For $p=\infty$, we replace $\ell_p$ spreading models with $c_0$ spreading models. Similarly, we let $\ssj(X,K)$ be the minimum ordinal $\xi$ not exceeding $\omega_1$ such that $X$ fails to contain a $K$-$c_0^\xi$ special sequence, and let $\ssj(X)=\sup_{K\geqslant 1}\ssj(X, K)$.  Note that by our conventions, $\ssm_p(X,K)=\infty$ if and only if $X$ admits a sequence $K$-equivalent to the $\ell_p$ (resp. $c_0$) basis.  

We make a few simple observations concerning higher order spreading models.  Note that if $\fff, \gggg$ are regular families and if $(x_n)\subset X$ is a $K$-$\ell_p^{\fff[\gggg]}$ spreading model, then any $p$-absolutely convex block $(y_n)$ of $(x_n)$ with $y_n=\sum_{j\in E_n} a_jx_j$, $E_n\in \gggg$ for all $n\in \nn$, must be a $K$-$\ell_p^\fff$ spreading model.  

If $\iota(\fff)\geqslant \omega$, then $\fff$ contains sets of arbitrarily large cardinality. Since $\fff$ is hereditary, we can fix $E_1<E_2<\ldots$, $E_n\in \fff$ with $|E_n|=n$.   If $(x_n)\subset X$ is a $K$-$\ell_p^\fff$ spreading model for $1<p\leqslant \infty$, and if $M\in \infin$, then $$\Bigl\|\frac{1}{n}\sum_{i\in E_n} x_{m_i}\Bigr\|\leqslant n^{1/p}/n\to 0,$$ whence $(x_n)$ must be a weakly null sequence. Since $(x_n)$ must have a seminormalized subsequence (and must be seminormalized itself if $\fff$ contains all singletons), some subsequence of $(x_n)$ is a $(1+\varepsilon)$-basic $K$-$\ell_p^\fff$ spreading model. Moreover, if $X$ has an FDD and admits a $K$-$\ell_p^\fff$ spreading model for some $K\geqslant 1$, $1<p\leqslant \infty$, and $\fff$ with $\iota(\fff)\geqslant \omega$, then $X$ admits a $(K+\varepsilon)$-$\ell_p^\fff$ spreading model which is a block sequence with respect to the FDD.  

Suppose $(x_n)\subset X$ is a $K$-$\ell_1^\fff$ spreading model.  By Rosenthal's $\ell_1$ theorem \cite{Rosenthal}, either some subsequence of $(x_n)$ is equivalent to the $\ell_1$ basis, or some subsequence is weakly Cauchy.  Assume $(x_n)$ itself is weakly Cauchy.  If $\iota(\fff)$ is a limit ordinal, $\iota(\fff[\aaa_2])=2\iota(\fff)=\iota(\fff)$, and there must exist $M\in \infin$ so that $\fff[\aaa_2](M)\subset \fff$.  Then with $M=(m_n)$, $(x_{m_n})$ is a $K$-$\ell_1^{\fff[\aaa_2]}$ spreading model and therefore $2y_n=x_{m_{2n}}-x_{m_{2n-1}}$ is such that $(y_n)$ is a weakly null $K$-$\ell_1^\fff$ spreading model.  Thus for any $1\leqslant \xi< \omega_1$, if $X$ admits a $K$-$\ell_1^\xi$ spreading model, then either every $K$-$\ell_1^\xi$ spreading model in $X$ admits a subsequence equivalent to the $\ell_1$ basis, or $X$ admits a weakly null $K$-$\ell_1^\xi$ spreading model.  

We last recall the definition of an \emph{asymptotic} $\ell_p^\fff$ or $c_0^\fff$ basis.   Given a basis $E$ for a Banach space $X$, we say $E$ is $K$-asymptotic $\ell_p^\fff$ (resp. $c_0^\fff$ when $p=\infty$) provided that each $\fff$-admissible, normalized block sequence $(x_i)_{i=1}^n$ with respect to $E$, and for each $(a_i)_{i=1}^n\in S_{\ell_p^n}$, $$K^{-1}\leqslant \Bigl\|\sum_{i=1}^n a_ix_i\Bigr\|\leqslant K.$$  We say a block sequence $(x_i)_{i=1}^n$ is $\fff$-admissible with respect to $E$ provided $(\supp_E(x_i))_{i=1}^E$ is $\fff$-admissible.  Note that if $E$ is a $K$-asymptotic $\ell_1^\fff$ basis, then every normalized block sequence with respect to $E$ is a $K$-$\ell_1^\fff$ spreading model.  Note that if $E$ is a bimonotone, $K$-asymptotic $c_0^\fff$ basis, then every normalized block sequence is a $K$-$c_0^\fff$ special sequence.

It is clear that for a Banach space $X$ with FDD $E$ and $1\leqslant p\leqslant \infty$, each of the properties below is implied by the successive properties, with the replacement of $\ell_p$ by $c_0$ in the case that $p=\infty$.  

\begin{enumerate}[(i)]\item $\ell_p$ is finitely representable in $X$. \item $\ell_p$ is finitely representable in $X$ on disjoint blocks of $E$.  \item $\ell_p$ is finitely block representable in $E$.  \item $I^a_p(X, E)>\omega$. \item $X$ admits $(1+\varepsilon)$-$\ell_p^{\aaa_n}$ spreading models for all $\varepsilon>0$ and all $n\in \nn$.  \end{enumerate} 

We recall some simple examples to show that no property on the list above implies any of the properties after it, so that the different indices are indeed distinct.  

\begin{enumerate}[(i)]\item Every Banach space is finitely representable in $c_0$, but the basis of $\ell_p$ is not finitely representable on disjoint blocks of the $c_0$ basis.  

\item The $c_0$ basis is finitely representable on disjoint blocks of the basis of Schlumprecht space $S$ \cite{Denka}, which means $\ell_1$ is finitely representable on disjoint blocks of the basis of $S^*$.  Therefore $\ell_p$ is finitely representable on disjoint blocks of the basis of the $p$-convexification of $S^*$.  But only $c_0$ is finitely block representable on this basis.

\item The space $X=\bigl(\sum \ell_p^n\bigr)_{\ell_q}$, where $\infty \neq q\neq p$, has a natural basis $E$ in which $\ell_p$ is block finitely representable, but it is clear that $I^a_p(X, E)=\omega$.  This is because in each asymptotic block tree, some branch will be closely equivalent to the basis of $\ell_q^n$ for some $n\in \nn$.  

\item For $1<q< \infty$ and $1\leqslant p<q$, we define for each $n\in \nn$ a norm on the finitely supported, scalar-valued functions defined on $\widehat{\aaa}_n$.  We let $$\|f\|_n = \sup\Bigl\{\Bigl(\sum_{c\in I}\bigl(\sum_{F\in c}|f(F)|^p\bigr)^{q/p}\Bigr)^{1/p}: I \text{\ is a collection of pairwise disjoint segments in\ }\widehat{\aaa}_n\Bigr\}.$$  We let $X_n$ denote the completion of this set with respect to this norm and note that with $e^n_F= 1_F$ for $F\in \widehat{\aaa}_n$, $(e^n_F)_{F\in \widehat{\aaa}_n}$ is a normalized, $1$-unconditional basis for $X_n$. Moreover, each $X_n$ is isomorphic to $\ell_q$, and the natural basis is equivalent to the $\ell_q$ basis.  Then we let $X=\bigl(\oplus X_n\bigr)_{\ell_q}$.  Then each basis $(e^n_F)_{F\in \widehat{\aaa_n}}$ naturally witnesses the fact that $I^a_p(X, E, 1)>n$.  But as was shown in \cite{Trees and branches}, every normalized, weakly null sequence in this space has a subsequence $(1+\varepsilon)$-equivalent to the unit vector basis of $\ell_q$.  Since $X$ contains no copy of $\ell_1$, if $(x_i)$ is a $K$-$\ell_p^{\aaa_{2n}}$ spreading model, we can pass to a weakly Cauchy subsequence and then to the $p$-absolutely convex block $(2^{-1/p}(x_{2i}-x_{2i-1}))$ to obtain a weakly null $K$-$\ell_p^{\aaa_n}$ spreading model.  By passing to a further subsequence, we can obtain a sequence closely equivalent to the $\ell_q$ basis, which implies a lower estimate on $K$ as a function of $n$ which tends to infinity as $n$ does.  Thus $X$ cannot contain uniform $\ell_p^{\aaa_n}$ spreading models for all $n\in \nn$.  Moreover, if $p=1$, the dual of this space satisfies $I^a_\infty(X^*, E^*,1)>\omega$, while $X^*$ does not admit uniform $c_0^{\aaa_n}$ spreading models.

\end{enumerate}

Next we collect some important facts about these indices to be used throughout.  For $p=1$, a number of the facts below can be found in \cite{JO}. Because our goal is precise quantification, we include stronger quantifications than those established in \cite{JO}.  It is clear that if $X$ is finite dimensional, $I_p(X)=1+\dim X$, and otherwise $I_p(X)\geqslant \omega$. Moreover, if $X$ has an FDD $E$, $I_p(X,E),I_p^a(X, E)\geqslant \omega$.  We collect the following facts concerning these indices for the infinite dimensional case.   

\begin{proposition} Let $X$ be an infinite dimensional Banach space.  Let $1\leqslant p\leqslant \infty$, $1\leqslant K$, and let $\lambda$ be a limit ordinal.  \begin{enumerate}[(i)] \item If $F\subset X^*$ is finite and $Y=\cap_{x^*\in F}\ker(x^*)=F_\perp$, $I_p(X,K)>\lambda$ if and only if $I_p(Y,K)>\lambda$ \item If $E$ is an FDD for $X$ and if $Y=[E_n:n\geqslant m]$ and $F=(E_n)_{n\geqslant m}$, $I_p(X, E, K)>\lambda$ if and only if $I_p(Y,F,K)>\lambda$. \item If $Y=[E_n:n\geqslant m]$ and $F=(E_n)_{n\geqslant m}$, $I_p^a(X,E,K)=I^a_p(Y,F,K)$.  \item Either $I_p(X)=\infty$ or there exists $\xi$ so that $I_p(X)=\omega^\xi$.  \item If $E$ is an FDD for $X$, either $I_p(X)=\infty$ or there exist $\alpha\leqslant \beta\leqslant \gamma\leqslant 1+\alpha$ so that $I_p^a(X, E)=\omega^\alpha$, $I_p(X,E)=\omega^\beta$, and $I_p(X)=\omega^\gamma$.  \item For any $K\geqslant 1$, $I_p(X, K)<I_p(X)$ and, if $E$ is an FDD for $X$, $I_p(X, E, K)<I_p(X,E)$ and $I^a_p(X,E,K)<I_p^a(X,E)$.   \end{enumerate}

\label{index facts}

\end{proposition}

\begin{proof}(i) Clearly $I_p(X, K)\geqslant I_p(Y, K)$.  If $I_p(X, K)>\lambda$, fix $(x_t)_{t\in \ttt_\lambda}$ so that $(x_{t|_i})_{i=1}^{|t|}\in T_p(X,K)$ for all $t\in \ttt_\lambda$ and fix $m>|F|$.  Define a function $f:C(\ttt_\lambda)=C(\ttt_{m\lambda})\to \rr$ by $$f(c)=\min \{\sum_{x^*\in F}|x^*(x)|: x\in \text{co}_p(x_t:t\in c)\}.$$  Then by simply comparing dimensions, for every order preserving $i:\ttt_m\to \ttt_{m\lambda}$, $\min \{f\circ i(c):c\in C(\ttt_m)\}=0$.  Thus by \ref{local coloring} we can find an order preserving $j:\ttt_\lambda\to C(\ttt_\lambda)$ so that $f\circ j=0$.  For each $t\in \ttt_\lambda$, fix $u_t\in \text{co}_p(x_s:s\in j(t))$ with $f(u_t)=0$, which means $u_t\in Y$.  Then since $j$ is order preserving, $(u_t)_{t\in \ttt_\lambda}$ is such that $(u_{t|_i})_{i=1}^{|t|}\in T_p(Y, K)$, since this sequence lies in $Y$ and is a $p$-absolutely convex block of a member of $T_p(X,K)$.  Thus $(u_t)_{t\in \ttt_\lambda}$ witnesses the fact that $I_p(Y,K)>\lambda$.  

(ii) Suppose that $I_p(X,E,K)>\lambda$ and fix $(x_t)_{t\in \ttt_\lambda}$ so that $(x_{t|_i})_{i=1}^{|t|}\in T_p(X,E,K)$ for each $t\in \ttt_\lambda$.  Let $$f(c)= \min\{\|P^E_{[1,m)}x\|: x\in \text{co}_p(x_t:t\in c)\}.$$  We finish as in (i).   

(iii) If $I^a_p(X, E, K)>\lambda$, then there exists a regular family $\fff$ with $\iota(\fff)=\lambda$ and a collection $(x_G)_{G\in \widehat{\fff}}$ so that for each $G\in \widehat{\fff}$, $(x_{G|_i})_{i=1}^{|G|}\in T_p(X, E, K)$ and $\max G\leqslant \supp_E(x_G)$.  Then with $\gggg=\fff\cap [m, \infty)^{<\omega}$, $(x_G)_{G\in \widehat{\gggg}}$ witnesses the fact that $I^a_p(Y, F, K)>\lambda$.

(iv) Assume $I_p(X)<\infty$, which means $I_p(X)$ is an infinite ordinal.  It is sufficient to show that if $\lambda<I_p(X)$, $\lambda 2<I_p(X)$. Assume $\lambda<I_p(X,K)$.  Choose any sequence $(x_i)_{i=1}^m\in T_p(X, K)$.  Fix $\varepsilon>0$ and choose $F\subset X^*$ finite which is $(1+\varepsilon)$-norming for $[x_i:1\leqslant i\leqslant m]$. Note that in the space $[x_i:1\leqslant i\leqslant m]\oplus F_\perp$, $[x_i:1\leqslant i\leqslant m]$ is $(1+\varepsilon)$-complemented, and so $F_\perp$ is $(2+\varepsilon)$-complemented.  If $(u_i)_{i=1}^n\in T_p(F_\perp, K)$, then for any scalars $(a_i)_{i=1}^m$ and $(b_i)_{i=1}^n$, \begin{align*} \Bigl\|\sum_{i=1}^m a_i x_i +\sum_{i=1}^n b_iu_i\Bigr\| & \leqslant \Bigl\|\sum_{i=1}^m a_ix_i\Bigr\|+\Bigl\|\sum_{i=1}^n b_iu_i\Bigr\| \leqslant \Bigl(\sum_{i=1}^m |a_i|^p\Bigr)^{1/p}+\Bigl(\sum_{i=1}^n |b_i|^p\Bigr)^{1/p} \\ & \leqslant 2^{1/q}\Bigl(\sum_{i=1}^m |a_i|^p+\sum_{i=1}^n |b_i|^p\Bigr)^{1/p}, \end{align*} where $q$ is the conjugate exponent to $p$.  Moreover, \begin{align*} \Bigl\|\sum_{i=1}^m a_i x_i +\sum_{i=1}^n b_iu_i\Bigr\| & \geqslant (2+\varepsilon)^{-1}\max\Bigl\{\Bigl\|\sum_{i=1}^m a_ix_i\Bigr\|, \Bigl\|\sum_{i=1}^n b_iu_i\Bigr\|\Bigr\} \\ & \geqslant (2+\varepsilon)^{-1}/K \max\Bigl\{\Bigl(\sum_{i=1}^m |a_i|^p\Bigr)^{1/p}, \Bigl(\sum_{i=1}^n|b_i|^p\Bigr)^{1/p}\Bigr\} \\ & \geqslant 2^{-1/p}(2+\varepsilon)^{-1}/K.\end{align*} Therefore $2^{-1/q}(x_1, \ldots, x_m, u_1, \ldots, u_n)\in T_p(X,2(2+\varepsilon)K)$.  But since this holds for any $(u_1, \ldots, u_n)\in T_p(F_\perp, K)$, and since $I_p(F_\perp, K)>\lambda$, $$\Bigl\{2^{-1/q}(x_1, \ldots, x_m): (x_1, \ldots, x_m)\in T_p(X, K)\Bigr\}\subset d^\lambda(T_p(X, 2(2+\varepsilon)K)).$$  Since the tree on the left also has order exceeding $\lambda$, we deduce that for any $\delta>0$, $I_p(X, 4K+\delta)>\lambda 2$.  
It is easy to see how to modify this proof to see that if $I_p(X, K)>\lambda$, then for any $n\in \nn$ and $\varepsilon>0$, $I_p(X, 2nK+\varepsilon)>\lambda n$.  We simply select $(x^1_i)_{i=1}^{l_1}\in T_p(X, K)$, $F^1\subset X^*$ finite and $(1+\varepsilon)$-norming for $[x_i^1:1\leqslant i\leqslant i_l]$.  Then we choose $(x^2_i)_{i=1}^{l_2}\in T_p(F^1_\perp, K)$ and $F^2\subset X^*$ finite and $(1+\varepsilon)$-norming for $[x_i^j: 1\leqslant j\leqslant 2, 1\leqslant i\leqslant l_j]$.  Choose $(x^3_i)_{i=1}^{l_3}\in T_p(F^2_\perp, K)$, and so on.  Then $$\Bigl\|\sum_{j=1}^n \sum_{i=1}^{l_j} a_{ij}x^j_i\Bigr\|\leqslant n^{1/q}\Bigl(\sum_{j=1}^n\sum_{i=1}^{l_j}|a_{ij}|^p\Bigr)^{1/p}$$ and \begin{align*} \Bigl\|\sum_{j=1}^n\sum_{i=1}^{l_j} a_{ij} x^j_i\Bigr\| & \geqslant K^{-1}(1+\varepsilon)^{-1}(2+\varepsilon)^{-1}\max_{1\leqslant j\leqslant n}\Bigl\{\Bigl(\sum_{i=1}^{l_j}|a_{ij}|^p\Bigr)^{1/p}\Bigr\} \\ & \geqslant n^{-1/p}K^{-1}(1+\varepsilon)^{-1}(2+\varepsilon)^{-1}\Bigl(\sum_{j=1}^n \sum_{i=1}^{l_j} |a_{ij}|^p\Bigr)^{1/p}.\end{align*}

(v) Again, assume $I_p(X)<\infty$.  The proof of the existence of $\alpha$ and $\beta$ is essentially the same as the previous case, except we assume that $(x_i)_{i=1}^m$ is a block sequence so that $\supp_E(x_m)\subset [1,k)$ for some $k$, and then concatenate this with members of $T_p([E_n:n\geqslant k], K)$.  We note that in the case that $E$ is bimonotone, if $(x^1_i)_{i=1}^{l_i}$, \ldots, $(x^n_i)_{i=1}^{l_n}$ are block sequences so that the concatenation is also a block sequence, then $$\Bigl\|\sum_{j=1}^n \sum_{i=1}^{l_j} a_{ij}x^j_i\Bigr\|\geqslant \max_{1\leqslant j\leqslant n}\Bigl\{\Bigl\|\sum_{i=1}^{l_j} a_{ij}x^j_i\Bigr\|\Bigr\},$$ so that in this case we have that if $\lambda<I_p(X, E, K)$, then $\lambda n<I_p(X, E, Kn)$.  We will prove later that this lower estimate on the growth rate of $I_p(X,E,\cdot)$ is sharp.  

Note that at this point, we have established that $I_p(X), I_p(X,E)$, and $I^a_p(X, E)$ are limit ordinals, which gives (vi).  We will use this to prove that $\gamma\leqslant 1+\alpha$.  If $\gamma>1+\alpha$, then $\omega^\gamma> \omega^{1+\alpha}=\omega I^a_p(X,E)>\omega I^a_p(X, E, K)$ for all $K$.  But this means there exists some $K\geqslant 1$ so that $I_p(X, K)>\omega I^a_p(X,E,K)$.  Let $\xi= I^a_p(X,E,K)$ and let $(x_t)_{t\in \ttt_{\omega\xi}}$ be so that $(x_{t|_i})_{i=1}^{|t|}\in T_p(X, K)$ for each $t\in \ttt_{\omega\xi}$. By replacing $K$ with any strictly larger value, we can assume that each $x_t$ has finite support.  For each $n\in \nn$, define $f_n:C(\ttt_{\omega\xi})\to \rr$ by $$f_n(c)=\Bigl\{\sum_{i=1}^n \|P^E_{[1,i]}x\|: x\in \text{co}_p(x_t:t\in c)\Bigr\}.$$  Again, simply by comparing dimensions, for any $n\in \nn$ and any order preserving $i:\ttt_\omega\to \ttt_{\omega\xi}$, $$\min \{f_n\circ i(c):c\in C(\ttt_\omega)\}=0.$$  Note that $f_1\leqslant f_2\leqslant f_3\leqslant \ldots$ pointwise.  Thus by Lemma \ref{mixed coloring} we can find a regular family $\fff$ with $\iota(\fff)=\xi$ and an order preserving $j:\widehat{\fff}\to C(\ttt_{\omega\xi})$ so that for each $G\in \widehat{\fff}$, $$f_{\max G}(j(G))=0.$$ Choose for each $G\in \widehat{\fff}$ some $u_G\in \text{co}_p(x_t:t\in j(G))$ so that $\sum_{i=1}^n \|P_{[1,i]}u_G\|=0$, which means $\max G \leqslant \supp_E(x_G)$.  Then $(u_G)_{G\in \widehat{\fff}}$ witnesses the fact that $I^a_p(X, K)>\xi$, a contradiction.

\end{proof}

Next, we define for each ordinal a class of Banach spaces admitting (i) a crude $\ell_p$ structure of a certain size, and (ii) almost isometric $\ell_p$ structures of a certain size.  These classes will be our primary objects of study.    For $1\leqslant p\leqslant \infty$ and $\xi\in \ord$, we let $$P_1^\vee(\xi, p)=\bigcup_{K>1}\{X\in \textbf{Ban}: I_p(X,K)>\omega^\xi\},$$ $$ P_1^\wedge(\xi, p) = \bigcap_{K>1} \{X\in \textbf{Ban}: I_p(X, K)>\omega^\xi\},$$ $$P_2^\vee(\xi, p)=\bigcup_{K>1}\{(E,X):E\text{\ is an FDD for\ }X\in \textbf{SB}, I_p(X,E,K)>\omega^\xi\},$$ $$ P_2^\wedge(\xi, p) = \bigcap_{K>1} \{(E,X):E \text{\ is an FDD for\ }X\in \textbf{SB}, I_p(X,E, K)>\omega^\xi\}.$$  $$P_3^\vee(\xi, p)=\bigcup_{K>1}\{(E,X):E\text{\ is an FDD for\ }X\in \textbf{SB}, I_p^a(X,E,K)>\omega^\xi\},$$ $$ P_3^\wedge(\xi, p) = \bigcap_{K>1} \{(E,X):E \text{\ is an FDD for\ }X\in \textbf{SB}, I_p^a(X,E, K)>\omega^\xi\},$$ $$P^\vee_4(\xi,p)=\bigcup_{K>1}\{X\in \textbf{Ban}: \ssm_p(X, K)>\xi\},$$ $$P^\wedge_4(\xi, p)=\bigcap_{K>1}\{X\in \textbf{Ban}: \ssm_p(X, K)>\xi\}.$$

Of course, it is well known that $P^\vee_j(1,p)=P^\wedge_j(1,p)$ for all $1\leqslant p\leqslant \infty$.  This is a result of James for $p=1 $ and $p=\infty$ \cite{James}, and a consequence of Krivine's theorem \cite{K} for $1<p<\infty$.  In sections 6 and 7, we discuss when $P^\vee_i(\xi, p)=P^\wedge_i(\xi, p)$ for $i=1,2,3,4$, $1\leqslant p\leqslant \infty$, and $\xi\in \ord$.  In particular, for $i=2,3,4$, we completely solve this question for all ordinals when $1<p<\infty$, and for all countable ordinals when $p=1$ or $\infty$.  Typically the study of these indices has been restricted to the case of separable spaces and countable indices.  We first give a simple construction that shows that for any $\xi\in \ord$ and for any $1\leqslant p\leqslant \infty$, there exists $X$ with $I_p(X)>\xi$ and not containing a copy of $\ell_p$.

\begin{proposition} For any $\xi\in \emph{\ord}$ and any $1\leqslant p\leqslant \infty$, there exists a Banach space $U_{\xi,p}$ with $\omega^\xi<I_p(U_{\xi,p}, 1)$ containing no copy of $\ell_p$ (resp. $c_0$).  More precisely, we can choose $U_{\xi,1}$ to be reflexive with $1$-unconditional (not necessarily countable) basis $E_\xi$ so that $\omega^\xi< I_1(U_{\xi,1}, E_\xi, 1)$, $U_{\xi, \infty}=U_{\xi, 1}^*$, and $U_{\xi, p}$ is the $p$-convexification of $U_{\xi, 1}$.  

\label{high examples}

\end{proposition}

Here, since the basis $E_\xi$ is unordered, we interpret $T_p(U_{\xi,1}, E_\xi, 1)$ as being the collection of finite, disjointly supported sequences isometrically equivalent to the $\ell_p^n$ basis, and $I_p(U_{\xi, 1}, E_\xi, 1)$ is the order of this tree.

\begin{proof} First, note that for any $\xi, \zeta\in \ord$, any Banach spaces $X,Y$, any $K\geqslant 1$, and $1\leqslant p\leqslant \infty$, if $\xi<I_p(X,K)$ and $\zeta<I_p(Y, K)$, $\zeta+\xi<I_p(X\oplus_p Y, K)$.  This is because for any $s\in T_p(X, K)$, $T_p(Y, K)\subset T_p(X\oplus_p Y, K)(s)$, so that $s\in d^\zeta(T_p(X\oplus_p Y, K))$.  This means $T_p(X, K)\subset d^\zeta(T_p(X\oplus_p Y, K))$, and $\varnothing \neq d^\xi(T_p(X, K))\subset d^{\zeta+\xi}(T_p(X\oplus_p Y, K)) $. Similarly, if $E$ is an unconditional basis for $X$ and $F$ is an unconditional basis for $Y$, $I_p(X\oplus_p Y, E\cup F, K)>\zeta+\xi$ if $\xi<I_p(X, E, K)$ and $\zeta<I_p(Y, F, K)$.

It is clear that if $E_\xi$ is a $1$-unconditional basis for a reflexive $U_{\xi,1}$ and $\omega^\xi<I_1(U_{\xi,1}, E_\xi, 1)$, then the $p$-convexification $U^{(p)}_{\xi,1}$ with the same basis $E_\xi$ will contain no copy of $\ell_p$ and satisfy $\omega^\xi<I_p(U_{\xi,1}^{(p)}, E_\xi, 1)$.  Thus it is sufficient to treat the $p=1$ case and its dual. 

First, we take $U_{0,1}= \mathbb{R}$ in the real case, $U_{0,1}=\mathbb{C}$ in the complex case. Therefore $U_{0,1}^*=U_{0,1}$ and $I_1(U_{0,1}, E_0,1)=I_\infty(U_{0,1}, E_0,1)=2$.      

If $U_{\xi,1}$ has been defined, we let $V_1= U_{\xi,1}$ and $V_{n+1}= U_{\xi,1}\oplus_1 V_n$.  Then $V_n$ is reflexive with $1$-unconditional basis, say $F_n$, and $V_n^*=\bigl(\oplus U^*_{\xi,1}\bigr)_{\ell_\infty^n}$.  Moreover, for each $n\in \nn$, since $I_1(U_{\xi,1}, E_\xi  ,1)>\omega^\xi$, $I_1(V_n, F_n, 1)>\omega^\xi n$.  Similarly, if $F_n^*$ is the dual basis to $F_n$, $I_\infty(V_n^*, F_n^*, 1)>\omega^\xi n$.  We take $U_{\xi+1,1}=\bigl(\oplus V_n\bigr)_{\ell_2}$.  

Last, suppose $\xi$ is a limit ordinal, and that for each $\zeta<\xi$, $U_{\zeta,1}$ has been defined to have the announced properties.  We take $U_{\xi,1} = \bigl(\oplus U_{\zeta,1}\bigr)_{\ell_2([1, \xi))}$.

\end{proof}

It is not hard to see that these spaces to not admit $\ell_1^1$ spreading models.  We will later see that if $W$ is any non-zero Banach space not admitting an $\ell_1^{\omega^\zeta}$ spreading model, if we repeat the above construction with $W$ in place of $U_{0,1}$, we build up a collection $W_{\xi, 1}$ so that $I_1(W_{\xi,1}, 1)>\omega^\xi$ admitting no $\ell_1^{\omega^\zeta}$ spreading model.  Of course, these spaces contain $\ell_2$ for $\xi>0$.  In \cite{AMP}, remarkable examples were given of Banach spaces $W_\xi$, $\xi<\omega_1$, so that $I_1(X)>\xi$ for every infinite dimensional subspace $X$ of $W_\xi$, while $W_\xi$ admits no $\ell_1^1$ spreading model.

\section{Dualization}

\subsection{Bourgain-Delbaen constructions}

It was asked in \cite{CS} if whenever a Banach space $X$ contains a copy of $\ell_p$, $1<p<\infty$, must $X^*$ contain a copy of $\ell_q$, where $q$ is the conjugate exponent to $p$? It was shown in \cite{BD} that there exists an $\mathcal{L}_\infty$ Banach space $Z$ with dual isomorphic to $\ell_1$ so that every infinite dimensional subspace of $Z$ admits a further infinite dimensional reflexive subspace.  Haydon \cite{H} actually showed that for any $1<p<\infty$, there exists an $\mathcal{L}_\infty$ Banach space $Z_p$ so that $Z_p^*$ is isomorphic to $\ell_1$ and so that $\ell_p$ embeds into $Z_p$.  Later, Freeman, Odell, and Schlumprecht \cite{FOS} showed that if $X$ is any Banach space having separable dual, there exists a $\mathcal{L}_\infty$ Banach space $Z_X$ containing an isomorphic copy of $X$ so that $Z_X^*$ is isomorphic to $\ell_1$.  Thus we see that not only does the question from \cite{CS} have a negative answer, but that the failure is quantitatively essentially as strong as possible.  By this, we mean that for $1<p<\infty$, $\ssm_p(Z_p)=\infty$ while $\ssm_q(Z_p^*)=\ssm_q(\ell_1)=1$ by the Schur property.  Moreover, if $1<p<2$, $I_p(Z_p)=\infty$ while $I_q(Z_p^*)=\omega$, the smallest possible value.  However, since $\ell_q$ is finitely representable in $\ell_1$ for $1\leqslant q\leqslant 2$, if $2\leqslant p<\infty$, $I_p(Z_p)=\infty$ while $I_q(Z_p^*)=\omega^2$.  Moreover, for $\xi<\omega_1$, with $T$ being the Tsirelson space $T(\theta, \xi)$, to be defined later, $I_1(Z_T)\geqslant I_1(T)= \omega^{\xi\omega}$, while $I_\infty(Z_T^*)=\omega$.  Thus we can find separable Banach spaces with arbitrarily large, countable $I_1$ index, while the duals of these spaces have the smallest possible $I_\infty$ index.  Thus the presence of large $\ell_p$ structures in a given space does not imply the presence of a large $\ell_q$ structure in the dual.

\subsection{Dualization for $p=\infty$}

The obvious exception to the above argument is $p=\infty$.  It is tempting to believe that the $\ell_1$ structures in the dual of a Banach space must be at least as large as the $c_0$ indices in the space itself.  For the indices involving an FDD, the result is clearly true.  This is because if $E$ is an FDD for $X$ with projection constant $b$ in $X$, then for any $x\in c_{00}(E)$, we can choose $x^*\in bB_{X^*}\cap c_{00}(E^*)$ with $\ran_{E^*}(x^*)\subset\ran_E(x)$ so that $\|x\|=x^*(x)$.  In this way, we can construct block trees in $B_{[E^*]}$ the branches of which are normed by branches of $T_\infty(X, E, K)$ to verify that the order of $T_1([E^*], E^*, bK)$ must be at least as large as the order of $T_\infty(X, E, K)$.  This gives the general idea for dualization of each index, which is to witness membership of sequences in $T_1(X^*, K)$ by norming them by an $\infty$-absolutely convex combination of a member of $T_\infty(X,K)$ acting biorthogonally or almost biorthogonally on it.  However, unlike the block case, it is not obvious how to choose the functionals.  Of course, if $(x, x_1, \ldots, x_n)\in T_\infty(X, K)$, we can take Hahn-Banach extensions $(x^*, x_1^*, \ldots, x_n^*)$ of the biorthogonal functionals of this sequence.  But say $(x, y_1, \ldots, y_m)\in T_\infty(X, K)$.  We can again obtain a sequence of Hahn-Banach extensions $(y^*, y_1^*, \ldots, y_m^*)$ of the biorthogonals to this sequence, but the functionals $x^*$ and $y^*$ need not be the same.  Thus we must work harder to choose the vectors acting biorthogonally on branches of $T_\infty(X, K)$ in a way that is consistent for vectors occurring in multiple members of $T_\infty(X, K)$.  The argument also works for preduals of $X$ with a slight modification using Helly's theorem \cite{H}.  Recall that Helly's theorem implies that if $X$ is a Banach space, $F\subset X^*$ is finite, and $x^{**}\in X^{**}$, then for any $\varepsilon>0$ there exists $x\in X$ with $\|x\|<\|x^{**}\|+\varepsilon$ so that for each $x^*\in F$, $x^{**}(x^*)= x^*(x)$.

\begin{proposition} Let $X$ be a Banach space.  Let $F\subset X^*$ be finite, $T=d^\alpha(T_\infty(X, K)(t_0))$ for some $\alpha\in \emph{\ord}$ and $t_0\in T_\infty(X,K)$.  For $0<\xi\in \emph{\ord}$ and $n\in \nn$, if $o(T)>\omega^\xi n$, there exist a $B$-tree $\ttt$ on $[1, \omega^\xi n)$ with $o(\ttt)=\omega^\xi n$ and functions $f:\ttt\to T$, $\phi:\ttt\to X$, and $\phi^*:\ttt\to KB_{X^*}$ so that for each $t\in \ttt$, \begin{enumerate}[(i)]\item $f(t|_1)\verb!^!\ldots \verb!^!f(t)\in T$, \item $\phi(t)\in \emph{\text{co}}_\infty(f(t))$, \item for $s\in \ttt$ comparable to $t$, $\phi^*(s)(\phi(t))=\delta_{st}$, \item for $x\in t_0$, $\phi^*(t)(x)=0$, \item for $x^*\in F$, $x^*(\phi(t))=0$. 

\end{enumerate}

Moreover, if $X$ is a dual space, say $X=Y^*$, and if $C>K$, the conclusion holds with $\phi^*$ mapping $\ttt$ into $CB_Y$.

\label{dualize}

\end{proposition}

\begin{theorem} Fix $\xi\in \emph{\ord}$ and a Banach space $X$.   \begin{enumerate}[(i)]\item  If $X\in P^\vee_1(\xi, \infty)$, then the dual of $X$ and any predual of $X$ lies in $P^\vee_1(\xi, 1)$.  \item Fix $i\in \{2,3\}$.  Suppose $E$ is an FDD for $X$ such that $(E,X)\in P^\vee_i(\xi, \infty)$. Then $(E^*, [E^*])\in P^\vee_i(\xi, 1)$.  If $Y$ is a Banach space with FDD $F$ so that $F^*=E$ and $[F^*]=X$, then $(F, Y)\in P^\vee_i(\xi, 1)$.  \item If $X\in P^\vee_4(\xi, \infty)$, then the dual of $X$ and any predual of $X$ lies in $P^\vee_4(\xi, 1)$.  \end{enumerate}

\end{theorem}

\begin{proof} If $X^*$ contains $\ell_1$ (respectively, if $X=Y^*$ and if $Y$ contains $\ell_1)$, the statement is obvious.  If $c_0$ embeds into $X$, then $X^*$ has $\ell_1$ as a quotient, and therefore as a subspace.  Similarly, if $c_0$ embeds into $X=Y^*$, then $\ell_1$ embeds into $Y$.  Thus we can assume $\ell_1$ does not embed into $X^*$ (resp. $Y$), and consequently that $c_0$ does not embed into $X$.  

(i) This follows immediately from Proposition \ref{dualize} by taking $T=T_\infty(X,K)$ for some $K$ such that $o(T_\infty(X,K))>\omega^\xi$.  We obtain a tree $(\phi^*(t))_{t\in \ttt}$ with $o(\ttt)=\omega^\xi$ and $(\phi(t))_{t\in \ttt}$ so that for each $t\in \ttt$, $(\phi(t|_i))_{i=1}^{|t|}\in T_\infty(X, K)$.  Then $(K^{-1}\phi^*(t|_i))_{i=1}^{|t|}\in T_1(X, K)$, since any linear combination of this sequence can be appropriately normed by an $\infty$-absolutely convex combination of $(\phi(t|_i))_{i=1}^{|t|}$.  Thus $X^*\in P^\vee_1(\xi, 1)$.  If $X=Y^*$, showing $Y\in P_1^\vee(\xi, 1)$ is similar.   

(ii) and (iii) are trivial.

(iv) If $X$ contains a $c_0^\xi$ spreading model, it contains a seminormalized basic $c_0^\xi$ spreading model $(x_n)$.  The biorthogonal functionals to this basis form an $\ell_1^\xi$ spreading model in $[x_n]^*$.  Taking Hahn-Banach extensions of these biorthogonals and scaling appropriately to make the sequence lie in $B_{X^*}$ gives an $\ell_1^\xi$ spreading model in $X^*$.  

In the case that $X=Y^*$, let $(x_n^*)\subset X^*$ be a bounded sequence acting biorthogonally on the $K$-$c_0^\xi$ spreading model $(x_n)\subset X$.  Let $C=\sup_n \|x_n^*\|$.  By Helly's theorem, we can choose for each $n\in \nn$ some $y_n\in (C+\varepsilon)B_Y$ so that for each $1\leqslant m\leqslant n$, $x_m(y_n)= x_n^*(x_m)=\delta_{mn}$.  Fix $\varepsilon\in (0,1)$ and choose $(\varepsilon_n)\subset (0,1)$ so that $\sum \varepsilon_n< \varepsilon$.  Since $(x_n)$ is weakly null, we can pass to a subsequence of $(x_n)$ and the corresponding subsequence of $(y_n)$ and assume that for each $1\leqslant n<m$, $|x_m(y_n)|<\varepsilon_m$.  Then for $E\in \sss_\xi$, \begin{align*} \Bigl\|\sum_{n\in E} a_n y_n\Bigr\| & \geqslant \Bigl(\sum_{m\in E} \sgn(a_m)x_m\Bigr)\Bigl(\sum_{n\in E}a_n y_n\Bigr) \\ & \geqslant \sum_{n\in E} |a_n|x_n(y_n)- \sum_{n\in E}\sum_{n>m\in E}|a_n||x_m(y_n)| - \sum_{n\in E}\sum_{n<m\in E}|a_n||x_m(y_n)|\\ & = \sum_{n\in E}|a_n| - \sum_{n\in E}\sum_{n<m\in E}|a_n||x_m(y_n)| \\ & \geqslant \sum_{n\in E}|a_n|- \sum_{n\in E}|a_n|\sum_m \varepsilon_m > (1-\varepsilon)\sum_{n\in E}|a_n|. \end{align*}
Thus $((C+\varepsilon)^{-1}y_n)$ is an $\ell_1^\xi$ spreading model in $Y$.

\end{proof}

We note that we can record the following quantitative versions of the previous result.  For any $X\in \textbf{Ban}$ with FDD $E$ in the appropriate cases, $\xi\in \ord$, and $K\geqslant 1$, \begin{enumerate}[(i)]\item $I_\infty(X, K)>\omega^\xi$ implies $I_1(X^*, K)>\omega^\xi$, \item $\ssm_\infty(X, K)>\xi$ implies $\ssm_1(X^*, 2K+\varepsilon)>\xi$, \item if $E$ has projection constant $b$ in $X$, $I_\infty(X,E,K)>\xi$ implies $I_1([E^*], E^*, bK)>\xi$, \item $I^a_\infty(X,E,K)>\xi$ implies $I^a_1([E^*], E^*, bK)>\xi$. \end{enumerate} Moreover, the analogous statements for preduals hold for (ii), (iii), and (iv), and the analogous statement for preduals holds for (i) if we replace $K$ by $K+\varepsilon$.  We remark that the $2K+\varepsilon$ estimate in (ii) comes from the fact that we can guarantee the $c_0^\xi$ spreading model $(x_n)$ is basic with basis constant $(1+\varepsilon)$, and $K^{-1}\leqslant \|x_n\|$.  In this case, the norms of the biorthogonal functionals are bounded by $(2+2\varepsilon)K$ and satisfy a $1$-$\ell_1$ lower estimate.

\begin{proof}[Proof of Proposition \ref{dualize}] Suppose $o(T)>\omega$.  Fix $m>|F|$.  For each $k\in \nn$, we can choose $(x_{ik})_{i=1}^{mk}\in T$.  Since $m>|F|$, for each $1\leqslant j\leqslant k$, we can choose $$y_{jk}\in \text{co}_\infty(x_{ik}: (j-1)m< i \leqslant jm) \cap \bigcap_{f\in F}\ker (f).$$  Write $t_0=(u_1, \ldots, u_l)$, and note that $(u_1, \ldots, u_l, y_{1k}, \ldots, y_{kk})\in T_\infty(X, K)$.  This means if $f_{1k}, \ldots, f_{kk}$ are defined on $[u_i, y_{jk}: 1\leqslant i\leqslant l, 1\leqslant j\leqslant k]$ by $f_{jk}(y_{ik})=\delta_{ij}$ and $f_{jk}(u_i)=0$ for each $i$, $\|f_{jk}\|\leqslant K$.  Let $x_{jk}^*\in KB_{X^*}$ be a Hahn-Banach extension of $f_{jk}$.  Recall that $$\ttt_\omega=\{(k, k-1, \ldots, j): k\in \nn, 1\leqslant j\leqslant k\}.$$  define $$f((k, k-1, \ldots, j))= (x_{ik}: (j-1)m< i\leqslant jm),$$  $$\phi((k, \ldots, j))=y_{jk},$$ and $$\phi^*((k, \ldots, j))=x^*_{jk}.$$  This gives the $\xi=1$, $n=1$ case. In the case that $X=Y^*$, the only modification we need is to replace $x_{jk}^*$ with some member $v_{jk}\in (K+\varepsilon)B_Y$ so that $v_{jk}$ and $x_{jk}^*$ agree on the vectors $u_1, \ldots, u_l$ and $y_{1k}, \ldots, y_{kk}$, which we can do by Helly's theorem.

Next, suppose we have the result for trees of order exceeding $\omega^\xi k$, $1\leqslant k\leqslant n$.  Suppose $T=d^\alpha(T_\infty(X,K)(t_0))$ is such that $o(T)>\omega^\xi(n+1)$, and let $T_0=d^{\omega^\xi }(T)= d^{\alpha+\omega^\xi}(T_\infty(X, K)(t_0))$.  Note that $o(T_0)>\omega^\xi n$.  Let $\ttt$ be a $B$-tree on $[1, \omega^\xi n)$ with order $\omega^\xi n$ and let $f:\ttt\to T_0$, $\phi:\ttt\to X$, $\phi^*:\ttt\to KB_{X^*}$ satisfy the conclusions.  For each $t\in MAX(\ttt)$, note that $$s_t:=f(t|_1)\verb!^!\ldots\verb!^! f(t)\in T_0 = d^{\omega^\xi}(T)=d^{\alpha+\omega^\xi}(T_\infty(X, K)(t_0)).$$  This means that $T_t:= d^\alpha(T_\infty(X,K)(t_0\verb!^!s_t))$ has $o(T_t)>\omega^\xi$.  Apply the inductive hypothesis to obtain a $B$-tree $\ttt_t$ on $[\omega^\xi n +1, \omega^\xi(n+1))$ with $o(\ttt_t)=\omega^\xi$ and maps $f_t:\ttt_t \to T_t$, $\phi_t:\ttt_t\to X$, and $\phi^*_t:\ttt_t\to KB_{X^*}$ to satisfy the conclusions with $F$ replaced by $F\cup \{\phi^*(t|_i): 1\leqslant i\leqslant |t|\}$ and $t_0$ replaced by $t_0\verb!^!s_t$.  Note that since for a fixed $\zeta$, $\beta\mapsto \zeta+\beta$ is order preserving, we can assume the tree $\ttt_t$ is on $[\omega^\xi n+1, \omega^\xi (n+1))$ rather than $[1, \omega^\xi)$.  We now let $$\overline{\ttt}=\ttt \cup \{t\verb!^!s: t\in MAX(\ttt), s\in \ttt_t\}.$$  Note that $\overline{\ttt}$ is a $B$-tree on $[1, \omega^\xi(n+1))$ with $o(\overline{\ttt})=\omega^\xi(n+1)$.  We extend $f:\ttt\to T$, $\phi:\ttt\to X$, $\phi^*:\ttt\to KB_{X^*}$ to functions on $\overline{\ttt}$ by letting $$f(t\verb!^!s)=f_t(s), \hspace{5mm} \phi(t\verb!^!s)=\phi_t(s), \hspace{5mm} \phi^*(t\verb!^!s)=\phi^*_t(s).$$  One easily checks that the conclusions are satisfied with these definitions.  

Next, for each $n\in \nn$, suppose the result holds for $\omega^\xi n$.  Then if $o(T)>\omega^{\xi+1}$, $o(T)>\omega^\xi n$ for all $n\in \nn$.  Let $s_0=0$ and let $s_n= s_{n-1}+n$.  Then for each $n\in \nn$, there exists a $B$-tree $\ttt_{(n)}$ on $[1, \omega^\xi n)$ of order $\omega^\xi_n$ and maps $f_n, \phi_n, \phi^*_n$ satisfying the conclusions.  Since $[1, \omega^\xi n)$ is order isomorphic to $[\omega^\xi s_{n-1}+1, \omega^\xi s_n)$, we can assume $\ttt_{(n)}$ is a tree on $[\omega^\xi s_{n-1}+1, \omega^\xi s_n)$.  Then we let $\ttt$ be the totally incomparable union of the $\ttt_{(n)}$.  This is a $B$-tree on $[1, \omega^{\xi+1})$ with order $\omega^{\xi+1}$.  We define $f, \phi, \phi^*$ on $\ttt$ by letting the restriction of each to $\ttt_{(n)}$ be equal to $f_n, \phi_n, \phi^*_n$, respectively.  

Last, suppose the result holds for each $\zeta<\xi$, $\xi$ a limit ordinal.  Then for each $\zeta<\xi$, there exists a $B$-tree $\ttt_{(\zeta)}$ on $[1, \omega^{\zeta+1})$ having order $\omega^{\zeta+1}$ and maps $f_\zeta, \phi_\zeta, \phi^*_\zeta$ satisfying the conclusions.  Since $[1, \omega^{\zeta+1})$ is order isomorphic to $[\omega^\zeta+1, \omega^{\zeta+1})$, we can assume $\ttt_{(\zeta)}$ is actually a tree on $[\omega^\zeta +1, \omega^{\zeta+1} )$.  We let $\ttt$ be the totally incomparable union of $\ttt_{(\zeta)}$, $\zeta<\xi$, and let $f, \phi, \phi^*$ be defined by letting the restriction to $\ttt_{(\zeta)}$ be $f_\zeta, \phi_\zeta, \phi^*_\zeta$, respectively.

\end{proof}

\section{Constant reduction for $\ell_1$ and $c_0$}

It is clear that among normalized Schauder bases, the bases of $\ell_1$ and $c_0$ play special roles.  Every normalized Schauder basis is trivially $1$-dominated by the $\ell_1$ basis.  Therefore if we wish to verify that a sequence $(x_i)_{i=1}^n\subset B_X$ is closely equivalent to the $\ell_1^n$ basis, we need only verify that this sequence has tight lower $\ell_1$ estimates.  A famous argument of James \cite{James} went as follows: If $(x_i)_{i=1}^{n^2}\subset B_X$ $C^2$-dominates the $\ell_1^{n^2}$ basis, then there exists a block $(y_i)_{i=1}^n\subset B_X$ of $(x_i)_{i=1}^{n^2}$ which $C$-dominates the $\ell_1^n$ basis.  Either there exists $1\leqslant j\leqslant n$ so that for each $x\in \text{co}_1(x_i: (j-1)n<i\leqslant jn)$, $\|x\|\geqslant 1/C$, in which case we take $y_i=x_{(j-1)n+i}$.  Otherwise we can take $z_j\in \text{co}_1(x_i: (j-1)n<i\leqslant jn)$ with $\|z_j\|<1/C$.  But $(z_j)_{j=1}^n$ also $C^2$-dominates the $\ell_1^n$ basis, since it is a $1$-absolutely convex block of a sequence which does, and so $(y_j)_{j=1}^n=(Cz_j)_{j=1}^n\subset B_X$ $C$-dominates the $\ell_1^n$ basis.  Iterating this argument implies that if $\ell_1$ is crudely finitely representable in $X$, then $\ell_1$ is finitely representable in $X$.  To restate in our language, $P_1^\vee(1,1)=P_1^\wedge(1,1)$.  

In \cite{JO}, a transfinite version of the dichotomy above was formulated to prove an extension of this result.  Much of this section is contained either implicitly or explicitly in \cite{JO}, while much of it is also new.  Many results here, while closely related to those of \cite{JO}, are more precise quantifications, extensions to uncountable ordinals and non-separable spaces, or have simplified proofs using our coloring lemmas.

We note that whereas the $\ell_1$ upper estimates come automatically for sequences in $B_X$, $c_0$ lower estimates do not follow automatically for sequences $(x_i)$ with $\|x_i\|\geqslant 1$.  For this, we need another argument of James, also from \cite{James}.  If $\varepsilon\in [0,1)$ and if $(x_i)_{i=1}^n$ is a sequence of vectors with $\|x_i\|\geqslant 1$ which is $(1+\varepsilon)$-dominated by the $\ell_\infty^n$ basis, then $(x_i)_{i=1}^n$ $(1-\varepsilon)^{-1}$-dominates the $\ell_\infty^n$ basis.  To restate in our language, if $(x_i)_{i=1}^n\in W(X, 1+\varepsilon)$ for $\varepsilon\in [0,1)$, $((1+\varepsilon)^{-1} x_i)_{i=1}^n\in T_\infty(X, (1+\varepsilon)(1-\varepsilon)^{-1})$.  To see this, assume $(a_i)_{i=1}^n\in S_{\ell_\infty^n}$ and $a_j=1$ for some $1\leqslant j\leqslant n$. Let $w=\sum_{i=1}^n a_ix_i$ and $w'=a_jx_j-\sum_{j\neq i=1}^n a_i x_i$, then $$2\leqslant 2\|a_jx_j\| = \|w+w'\|\leqslant \|w\|+\|w'\|\leqslant \|w\|+1+\varepsilon.$$  This is the reason we content ourselves to study the trees $W(\cdot, \cdot)$ in place of $T_\infty(\cdot, \cdot)$.  

We say an ordinal index $I$ depending on $K\geqslant 1$ is \emph{subadditive} provided if for all $C, K\geqslant 1$, $I(CK)\leqslant I(C)+I(K)$, and \emph{submultiplicative} if $I(CK)\leqslant I(C)I(K)$.  Of course, we automatically deduce that in this case we can reverse the order of addition or multiplication.  

\subsection{Submultiplicativity and consequences}

\begin{theorem}

If $X$ is any Banach space (with FDD $E$ in the appropriate cases), each of the following indices is submultiplicative: \begin{enumerate}[(i)]\item $I_1(X, \cdot)$, \item $I_1(X, E, \cdot)$, \item $I_1^a(X, E, \cdot)$, \item $J(X, \cdot)$, \item $J(X, E, \cdot)$, \item $J^a(X, E, \cdot)$. \end{enumerate}

Moreover, $ \ssm_1(X, \cdot)$ and $\ssj(X, \cdot)$ are subadditive. 

\label{submultiplicative} 

\end{theorem}

\begin{proof} (i) If $I_1(X, C)=\infty$ or $I_1(X, K)=\infty$, there is nothing to prove. Let $\zeta=I_1(X,C)$ and $\xi=I_1(X,K)$.  Assume $I_1(X, CK)>\zeta\xi$.  Then we can choose $(x_t)_{t\in \ttt_{\zeta\xi}}$ so that $(x_{t|_i})_{i=1}^{|t|}\in T_1(X, CK)$ for all $t\in \ttt_{\zeta\xi}$.  Define $f:C(\ttt_{\zeta\xi})\to \{0,1\}$ by $f(c)= 1$ provided $$\min \{\|x\|: x\in \text{co}_1(x_t:t\in c)\}\geqslant 1/C,$$ and $f(c)=0$ otherwise.  If $i:\ttt_\zeta\to \ttt_{\zeta\xi}$ is order preserving so that $f\circ i(c)=1$ for all $c\in C(\ttt_\zeta)$, then $(x_{i(t)})_{t\in \ttt_\zeta}$ witnesses the fact that $I_1(X, C)>\zeta$, a contradiction.  Therefore by Lemma \ref{local coloring}, there exists an order preserving $j:\ttt_\xi\to C(\ttt_{\zeta\xi})$ so that $f\circ j\equiv 0$.  For each $t\in \ttt_\xi$, we choose $u_t\in \text{co}_1(x_s: s\in j(t))$ with $\|u_t\|< 1/C$.  Then since $j$ is order preserving, $(u_{t|_i})_{i=1}^{|t|}$ is a $1$-absolutely convex block of a member of $T_1(X, CK)$, and is therefore also a member of $T_1(X, CK)$.  Then $(Cu_t)_{t\in \ttt_\xi}$ witnesses the fact that $I_1(X, K)>\xi$, another contradiction.  

(ii) The proof is the same, except each $(x_{t|_i})_{i=1}^{|t|}$, and therefore each $(u_{t|_i})_{i=1}^{|t|}$, is a block sequence with respect to some FDD.  

(iii) The proof is essentially the same, with the $B$-trees $\ttt_\xi, \ttt_\zeta, \ttt_{\zeta\xi}$ replaced by regular families $\fff, \gggg, \fff[\gggg]$ with $\iota(\fff)=\xi$ and $\iota(\gggg)=\zeta$.  We remark that in this case, in either of the alternatives above, the sequences of immediate successors do have minima of supports tending to infinity, as required.  This is because Lemma \ref{asymptotic coloring} involves embeddings rather than order preserving maps.  

(iv)-(vi) are essentially the same.  We consider the function $f(c)=1$ if $$\max \{\|x\|: x\in \text{co}_\infty(x_t:t\in c)\}\leqslant C,$$ and $0$ otherwise.

The remark about $\ssm_1$ and $\ssj$ is even easier.  If $(x_n)$ is a $CK$-$\ell_1^{\xi+\zeta}$ spreading model, we can pass to a subsequence which is a $CK$-$\ell_1^{\sss_\zeta[\sss_\xi]}$ spreading model.  Either there exists $N\in \nn$ so that for all $N\leqslant E\in \sss_\xi$, $(x_n)_{n\in E}\in T_1(X, C)$, in which case $(x_n)_{n>N}$ is a $C$-$\sss_\xi$ spreading model, or there exist $E_1<E_2<\ldots$, $E_i\in \sss_\xi$, and scalars $(a_i)$ so that for all $n\in \nn$, $\sum_{i\in E_n}|a_i|=1$ and $\|\sum_{i\in E_n} a_ix_i\|<1/C$.  Then let $y_n=\sum_{i\in E_n}a_ix_i$, so $(Cy_n)$ is a $K$-$\ell_1^\zeta$ spreading model.  The proof for $\ssj$ is similar.

\end{proof}

This yields the following: If $\xi$ is any ordinal, and if $\zeta<\omega^{\omega^\xi}$, then $\zeta^n<\omega^{\omega^\xi}$ for every $n\in \nn$.  If $I_1(X, K)>\omega^{\omega^\xi}$ for some $K\geqslant 1$, and if $C>1$, we can fix $n\in \nn$ so that $K^{1/n}< C$.  Then if $I_1(X, C)=\zeta<\omega^{\omega^\xi}$, $$\omega^{\omega^\xi}< I_1(X,K) \leqslant I_1(X, C^n) \leqslant I_1(X, C)^n = \zeta^n< \omega^{\omega^\xi},$$ a contradiction.  So $I_1(X, C)\geqslant \omega^{\omega^\xi}$ and, since $I_1(X,C)$ is a successor, $I_1(X,C)>\omega^{\omega^\xi}$.  Thus if $X\in P^\vee_1(\omega^\xi, 1)$, $X\in P^\wedge_1(\omega^\xi, 1)$.  We can perform similar arguments for $P^\vee_2(\omega^\xi, 1)$ and $P^\vee_3(\omega^\xi, 1)$.  

Similarly, if $I_\infty(X, K)>\omega^{\omega^\xi}$, $J(X, K)>\omega^{\omega^\xi}$.  As in the previous paragraph, we deduce that for any $\varepsilon\in (0,1)$, $J(X, 1+\varepsilon)>\omega^{\omega^\xi}$, which means $I_\infty(X, (1+\varepsilon)(1-\varepsilon)^{-1})>\omega^{\omega^\xi}$.  Thus we deduce $P^\vee_i(\omega^\xi, \infty)=P^\wedge_i(\omega^\xi, \infty)$ for $i=1,2,3$.  

If $\xi$ is a limit ordinal, building a tree of order $\xi$ is simply a matter of building a tree of order $\zeta$ for every $\zeta<\xi$ and taking a totally incomparable union.  But for spreading models, this is not true.  One must perform the blocking argument in a way which results in a sequence, not a tree.  It was shown in \cite{BCM} that this can be done to obtain the spreading model analogue of the results above. We sketch a proof, since it is an easy application of Lemma \ref{easy coloring}.

\begin{theorem} For any $0\leqslant \xi\leqslant \omega_1$, $P^\vee_4(\omega^\xi, 1)=P^\wedge_4(\omega^\xi, 1)$ and $P^\vee_4(\omega^\xi, \infty)=P^\wedge_4(\omega^\xi, \infty)$.  

\label{BCM theorem 1}

\end{theorem}

\begin{proof} For $\xi=\omega_1$, the result is simply a restatement of the result of James.  Suppose $(x_n)$ is a $K$-$\ell_1^{\omega^\xi}$ spreading model.  It is sufficient to show that some block of this sequence is a $K^{1/2}$-$\ell_1^{\omega^\xi}$ spreading model.  We define $f:\widehat{\sss}_{\omega^\xi}\to \{0,1\}$ by letting $f(E)=1$ if $$\min \{\|x\|: x\in \text{co}_1(x_n:n\in E)\}\geqslant K^{-1/2},$$ and $f(E)=0$ otherwise.  Either there exists $M\in \infin$ so that for each $E\in \widehat{\sss}_{\omega^\xi}$, $f(M(E))=1$, in which case $(x_{m_n})$ is a $K^{1/2}$-$\ell_1^{\omega^\xi}$ spreading model, or there exist $E_1<E_2<\ldots$ so that $f(E_i)=0$ and so that for each $E\in \sss_{\omega^\xi}$, $\cup_{i\in E}E_i\in \sss_{\omega^\xi}$.  For each $n\in \nn$, choose $y_n\in \text{co}_1(x_i:i\in E_n)$ with $\|y_n\|<K^{-1/2}$.  Then $(K^{1/2}y_n)$ is a $K^{1/2}$-$\ell_1^{\omega^\xi}$ spreading model.  

For the $c_0$ case, we mimic the argument above with $c_0^{\omega^\xi}$ special sequences in place of $\ell_1^{\omega^\xi}$ spreading models.  As before, a $(1+\varepsilon)$-$c_0^{\omega^\xi}$ special sequence yields a $(1+\varepsilon)(1-\varepsilon)^{-1}$-$c_0^{\omega^\xi}$ spreading model.

\end{proof}

\begin{corollary} For $\xi\in \emph{\ord}$, $p\in \{1, \infty\}$, and $i\in \{1,2,3,4\}$, $P^\vee_i(\omega^\xi, p)=P^\wedge_i(\omega^\xi, p)$.  

\label{reduction}

\end{corollary}

We can now define for $p\in \{1, \infty\}$ and $i\in \{1,2,3,4\}$ $\phi_{p,i}:\ord\to \ord$ by $$\phi_{p,i}(\xi)=\min \{\zeta\in \ord: P^\vee_i(\zeta, p)\subset P^\wedge_i(\xi, p)\}.$$  Corollary \ref{reduction} implies that this is well-defined, and $$\phi_{p,i}(\xi)\leqslant \min \{\omega^\zeta: \omega^\zeta\geqslant \xi\}.$$  We next aim to discuss the sharpness of this estimate for certain values of $i, p$.  

\subsection{Schreier and Tsirelson spaces, the repeated averages hierarchy}

For convenience in this section, we treat members of $\fin$ as projections on $\ell_\infty$ by letting $Ex$ be the pointwise product of $x$ and the sequence $(1_E(n))_n$.  
If $\fff\subset \fin$ is any regular family containing all singletons, the formula $$\|x\|_\fff=\sup_{i\in E}\|Ex\|_{\ell_1}$$ defines a norm making the canonical $c_{00}$ basis normalized and $1$-unconditional.  The completion of $c_{00}$ with respect to this norm is denoted $X_\fff$.  If $\fff=\sss_\xi$, we write $X_\xi$ in place of $X_{\sss_\xi}$ and $\|\cdot\|_\xi$ in place of $\|\cdot\|_{\sss_\xi}$.  This space is called the \emph{Schreier space of order} $\xi$.  This space was considered for $\xi=1$ by Schreier \cite{Schreier}, for finite values of $\xi$ in \cite{AO}, and for $\xi<\omega_1$ in \cite{AA}.  

Additionally, we have Banach spaces of the type of Tsirelson.  Fix $\theta\in (0,1)$. Define the sequence $(|\cdot|_n)_{n\geqslant 0}$ on $c_{00}$ by letting $|\cdot|_0=\|\cdot\|_{c_0}$ and $$|x|_{n+1}= |x|_n \vee \sup \Bigl\{ \theta \sum_{i=1}^s |E_i x|_n: (E_i)_{i=1}^s \text{\ is\ }\fff\text{-admissible}\Bigr\}.$$  Since $|\cdot|_n\leqslant \|\cdot\|_{\ell_1}$, we note that $\lim_n |x|_n$ defines a norm on $c_{00}$ making the canonical $c_{00}$ basis normalized and $1$-unconditional.  We let $T(\theta, \fff)$ denote the completion of $c_{00}$ under this norm.  We note that the norm $\|\cdot\|_{T(\theta, \fff)}$ satisfies the implicit formula $$\|x\|_{T(\theta, \fff)} = \|x\|_{c_0}\vee \sup \Bigl\{\theta\sum_{i=1}^s \|E_i x\|_{T(\theta, \fff)}: (E_i)_{i=1}^s \text{\ is\ }\fff\text{-admissible}\Bigr\}.$$    For $\theta=1/2$ and $\fff=\sss_1$, this is the Figiel-Johnson Tsirelson space \cite{FJ}, which is the dual of Tsirelson's original space \cite{Tsirelson}. As usual, we write $T(\theta, \xi)$ in place of $T(\theta, \sss_\xi)$.    

We next recall the repeated averages hierarchy, introduced in \cite{AMT}. This will be the final tool we will need in order to discuss the sharpness of the results above.  Given $L\in \infin$ and $0\leqslant \xi<\omega_1$, we will define a convex blocking $(x_n^{L, \xi})$ of the canonical $c_{00}$ basis so that \begin{enumerate}[(i)]\item $\supp(x_n^{L, \xi})\in \sss_\xi$, and \item $\cup_n \supp(x_n^{L, \xi})=L$. \end{enumerate} These sequences will allow us precise quantification of the complexity of blockings we will require to measure $I_1(X_\xi, K)$ and $I_1(T(\theta, \omega^\xi), K)$.

We let $L=(l_n)$ and $x_n^{L, 0}=e_{l_n}$.  If $(x_n^{L, \xi})$ has been defined, let $s_0=0$ and $L_0=L$.  Recursively choose $s_n, p_n, L_n$ so that for all $n\in \nn$, \begin{enumerate}[(i)]\item $p_n=\min L_{n-1}$, \item $s_n= p_n+s_{n-1}$, \item $L_n= L_{n-1}\setminus \cup_{i=s_{n-1}+1}^{s_n} \supp(x_i^{L, \xi})$. \end{enumerate} Then let $x_n^{L, \xi+1}=p_n^{-1}\sum_{i=s_{n-1}+1}^{s_n} x_i^{L, \xi}$.  

Suppose that $\xi<\omega_1$ is a limit ordinal and that for every $M\in \infin$ and every $\zeta<\xi$, $x_n^{M, \zeta}$ has been defined.  Let $\xi_n\uparrow \xi$ be such that $\sss_\xi=\{E:\exists n\leqslant E\in \sss_{\xi_n+1}\}$.  Let $L_0=L$ and recursively define $p_n, L_n$ so that for each $n\in \nn$, \begin{enumerate}[(i)]\item $p_n=\min L_{n-1}$, \item $L_n=L_{n-1}\setminus \supp(x_1^{L_{n-1}, \xi_{p_n}+1})$.     \end{enumerate} Let $x_n^{L, \xi}=x_1^{L_{n-1}, \xi_{p_n}+1}$.

For $(m_n)=M\in \infin$, we let $T_M:\ell_\infty\to \ell_\infty$ be the operator defined by $Te_n= e_{m_n}$.  We observe that for $x\in c_{00}$ and $\xi<\omega_1$, $\|T_Mx\|_\xi$ may be much larger than $\|x\|_\xi$, and that $T_M$ does not necessarily map $X_\xi$ into $X_\xi$.  Given $M\in \infin$ and $\varepsilon>0$, we say $L\in \infin$ is $(M, \varepsilon)$ \emph{fast growing} provided $m_{l_n}/l_{n+1} < \varepsilon/(1+2\varepsilon)$ for all $n\in \nn$.  Note that any infinite subset of an $(M, \varepsilon)$ fast growing set is also $(M, \varepsilon)$ fast growing, and any set $N\in \infin$ admits $L\in [N]$ which is $(M, \varepsilon)$ fast growing.

\begin{lemma} If $L$ is $(M, \varepsilon)$ fast growing, then for any $k\in \nn$, $$\Bigl\|T_M\sum_{n=1}^k x_n^{L, \xi}\Bigr\|_\xi \leqslant 1+\varepsilon.$$  

\label{repeated averages}

\end{lemma}

\begin{proof} For $M, \varepsilon$ fixed, we will prove by induction on $\xi$ that for any $(M, \varepsilon)$ fast growing $L\in \infin$ and any $E\in \sss_\xi$, $$\Bigl\|ET_M\sum_{n=1}^\infty x_n^{L, \xi}\Bigr\|_{\ell_1} \leqslant 1+\varepsilon.$$ For $\xi=0$, since $\|\cdot\|_0=\|\cdot\|_{c_0}$, this case is trivial.  

Write $x^{L, \xi+1}_n= p_n^{-1}\sum_{i=s_{n-1}+1}^{s_n} x_i^{L, \xi}$, where $p_n= \min \supp(x_n^{L, \xi+1})$.  Choose $E\in \sss_{\xi+1}$ and write $E=\cup_{i=1}^s E_i$ with $(E_i)_{i=1}^s$ $\sss_1$-admissible.  If $ET_Mx_n^{L, \xi+1}=0$ for all $n\in \nn$, there is nothing to prove.  So assume not, and let $r=\min (n: ET_Mx_n^{L, \xi+1}\neq 0)$.  Fix $n\in \nn$ and choose $l_{k_0}<\ldots < l_{k_n}$ so that $m_{l_{k_0}}\in E\cap T_Mx_r^{L, \xi+1}$ and $l_{k_i}= p_{r+i}$ for $1\leqslant i\leqslant n$.  Note that $s\leqslant \min E\leqslant m_{l_{k_0}}$, so $$\frac{s}{p_{r+n}}\leqslant \frac{m_{l_{k_0}}}{l_{k_n}} \leqslant \prod_{i=1}^n \frac{m_{l_{k_{i-1}}}}{l_{k_i}}\leqslant (\varepsilon/(1+2\varepsilon))^n.$$  Then \begin{align*} \bigl\|ET_M\sum x_n^{L, \xi+1}\bigr\|_{\ell_1} & \leqslant 1 +\sum_{n=r+1}^\infty p_n^{-1}\sum_{j=1}^s \bigl\|E_jT_M \sum_{j=s_{n-1}+1}^{s_n} x_j^{L, \xi}\bigr\|_{\ell_1} \\ & \leqslant 1+ \sum_{n=r+1}^\infty p_n^{-1}\sum_{j=1}^s (1+\varepsilon) = 1+(1+\varepsilon)\sum_{n=1}^\infty \frac{s}{p_{r+n}} \\ & \leqslant 1+ (1+\varepsilon)\sum_{n=1}^\infty (\varepsilon/(1+2\varepsilon))^n = 1+\varepsilon. \end{align*}

Suppose the result holds for every $(M, \varepsilon)$ fast growing set and every ordinal strictly less than the limit ordinal $\xi$.  Recall that for some sequence $\xi_n\uparrow \xi$, $\sss_\xi=\{E: \exists n\leqslant E\in \sss_{\xi_n+1}\}$, and that for each $n\in \nn$, $\sss_{\xi_n+1}\subset \sss_{\xi_{n+1}}$.  Let $x_n^{L, \xi}=x_1^{L_n, \xi_{p_n}+1}$ for some $L_n\in [L]$ and $p_n=\min \supp(x_n^{L, \xi})$.  Note that each $L_n$ is $(M,\varepsilon)$ fast growing.  Fix $E\in \sss_\xi$ and let $r=\min (n: ET_Mx^{L, \xi}_n\neq 0)$.  As in the previous case, we deduce that $1/p_{r+n}< (\varepsilon/(1+2\varepsilon))^n$ for each $n\in \nn$.  Moreover, if $m=m_{l_{k_0}}\in E\cap \supp(T_Mx_r^{L, \xi})$ and if $l_{k_n}=p_{r+n}$, we deduce that $E\in \sss_{\xi_m+1}$ and that $m< p_{r+n}$.  The second assertion can be seen by noting that $m/l_{k_n}< \varepsilon/(1+2\varepsilon)<1$.  This means that $E\in \sss_{\xi_m+1}\subset \sss_{\xi_{p_n}}$ for each $n>r$.  Then recalling that $x_1^{L_n, \xi_{p_n}+1}= p_n^{-1}\sum_{j=1}^{p_n} x_j^{L_n, \xi_{p_n}}$, we deduce \begin{align*} \bigl\|ET_M \sum x_n^{L, \xi}\bigr\|_{\ell_1} & \leqslant 1+ \sum_{n=r+1}^\infty \|ET_Mx^{L_n, \xi_{p_n+1}}_1\|_{\ell_1} \\ & = 1+ \sum_{n=r+1}^\infty p_n^{-1}\bigl\|ET_M  \sum_{j=1}^{p_n}x_j^{L_n, \xi_{p_n}}\bigr\|_{\ell_1} \\ & \leqslant 1+ (1+\varepsilon)\sum_{n=1}^\infty 1/p_{n+r} \leqslant 1+\varepsilon. \end{align*}

\end{proof}

We remark that even if $M=\nn$, the fast growing condition is required to obtain the $(1+\varepsilon)$-estimate.  To see this, fix $n\in \nn$ and an interval $I$ with $|I|=n$ and minimum $m>n+1$.  Let $J=[m+n, 2(m+n))$.  Then with $x= (n+1)^{-1}1_{(n+1)\verb!^!I}$, $y=|J|^{-1}1_J$, and $E= [m, 2m)\in \sss_1$, $$\|x+y\|_1 \geqslant \|E(x+y)\|_{\ell_1} = \frac{n}{n+1}+\frac{m-n}{m+n},$$ which can be made arbitrarily close to $2$ by auspicious choices of $m$ and $n$.

\subsection{Schreier spaces}

Assume $0<\xi<\omega_1$.  It is clear that if $(x_i)$ is a bimonotone basic sequence which is a $K$-$\ell_1^\xi$ spreading model, then $(x_i)$ is also a $Kn$-$\ell_1^{(\sss_\xi)_n}$ spreading model.  Since any space containing a $K$-$\ell_1^\xi$, contains a sequence equivalent to $\ell_1$ or a $K$-$\ell_1^\xi$ spreading model which is weakly null, we can in either case find a $K$-$\ell_1^\xi$ spreading model which is basic with basis constant almost $1$, and therefore has projection constant not more than $2+\varepsilon$.  Then any space containing a $K$-$\ell_1^\xi$ spreading model contains a $(2+\varepsilon)nK$-$\ell_1^{(\sss_\xi)_n}$ spreading model. Similarly, if $E$ is bimonotone, the proof of Proposition \ref{index facts} implies that if $I_1(X, E, K)>\omega^\xi$, $I_1(X, E, nK)>\omega^\xi n$.  We will use the Schreier spaces to demonstrate that this lower estimate on the growth rate is optimal.  We state this as follows.  

\begin{proposition} Let $0<\xi<\omega_1$ and let $E$ denote the basis of the Schreier space $X_\xi$.  Then $E$ is a $1$-$\ell_1^\xi$ spreading model, but if $1\leqslant K<n\in \nn$, $I_1(X_\xi, E, K)<\omega^\xi n$.  In particular, $X_\xi$ contains a $K$-$\ell_1^{(\sss_\xi)_n}$ spreading model if and only if $K\geqslant n$.  

\label{Schreier sharp}

\end{proposition}

\begin{proof}

It is clear that if $(x_i)$ is a normalized block sequence in $X_\xi$, and if $m_i=\max \supp(x_i)$, then $(x_i)$ is $1$-dominated by $(e_{m_i})$. Moreover, if $m_i\leqslant k_i$ for each $i\in \nn$, $(e_{m_i})$ is $1$-dominated by $(e_{k_i})$.  For $K\geqslant 1$, we let $$\gggg_K=\{(m_i)_{i=1}^n\in \fin: (e_{m_i})_{i=1}^n\in T_1(X_\xi, E, K)\}.$$  By the properties mentioned above, $\gggg_K$ is spreading, and clearly it is hereditary.  Moreover, we claim that for each $\zeta$, $$\{(\max \supp(x_i))_{i=1}^n: (x_i)_{i=1}^n\in d^\zeta(T_1(X_\xi, E, K))\}= d^\zeta(\gggg_K). $$  By definition of $\gggg_K$, $\gggg_K\subset \{(\max \supp(x_i))_{i=1}^n: (x_i)_{i=1}^n\in T_1(X_\xi, E, K)\}$.  But since $(x_i)_{i=1}^n\in T_1(X_\xi, E, K)$ is $1$-dominated by $(e_{\max \supp(x_i)})_{i=1}^n$, we have $\gggg_K\supset\{(\max \supp(x_i))_{i=1}^n: (x_i)_{i=1}^n\in T_1(X_\xi, E, K)\}$. This establishes the $\zeta=0$ case.  The successor and limit ordinal cases of the claim are trivial.  If $\gggg_K$ is not compact, then there must exist some $M\in \infin$ so that $[M]^{<\omega}\subset \gggg_K$.  If $\gggg_K$ is compact, $\gggg_K$ is regular.  In this case, if $\iota(\gggg_K)\geqslant \omega^\xi n$, there would exist some $M\in \infin$ so that $(\sss_\xi)_n(M)\subset \gggg_K$. In either case, there exists $M\in \infin$ so that $(\sss_\xi)_n(M)\subset \gggg_K$.   Fix $\varepsilon>0$ so that $(1+\varepsilon)K< n$ and choose $L\in \infin$ so that $L$ is $(M, \varepsilon)$ fast growing.  Then $E=\cup_{i=1}^n \supp(x_i^{L, \xi})\in (\sss_\xi)_n$, and $\supp(T_M\sum_{i=1}^n x_i^{L, \xi})=M(E)\in \gggg_K$.  But this means $\|T_M\sum_{i=1}^n x_i^{L, \xi}\|_\xi\geqslant n/K>1+\varepsilon$, while Lemma \ref{repeated averages} tells us that $\|T_M\sum_{i=1}^n x_i^{L, \xi}\|_\xi\leqslant 1+\varepsilon$.  This proves that $\gggg_K$ is compact with $\iota(\gggg_K)<\omega^\xi n$.  Then as a consequence of our claim above, $o(T_1(X_\xi, E, K))\leqslant \iota(\gggg_K)+1<\omega^\xi n$.

\end{proof}

Note that the canonical basis $E^*$ of $X_\xi^*$ is a $1$-$c_0^\xi$ spreading model.  But since the basis of $X_\xi$ is bimonotone, $I_\infty(X_\xi^*, E^*, K)<\omega^\xi n$ for $K<n$ by the quantitative versions of Theorem \ref{dualize}. Thus we see the sharpness of this lower estimate on growth rate for the $p=\infty$ structures as well.   Finally, the analogue of this statement for $1<p<\infty$ and the $p$-convexification of $X_\xi$ gives optimality of the lower bound on the growth rate of each of the indices for $1<p<\infty$.

\subsection{Tsirelson space} In this section, we aim to prove that for $i=2,3,4$ and $p=1,\infty$, $\phi_{p,i}(\xi)=\min \{\omega^\zeta: \omega^\zeta\geqslant \xi\}$ for countable $\xi$, yielding that Corollary \ref{reduction} is sharp in some cases.  We note that our methods rely on dealing with block sequences in a basis, so they do not fully elucidate the values $\phi_{1,1}(2)$ or $\phi_{\infty,1}(2)$.  We will prove the following.  

\begin{theorem} For $\xi<\omega_1$ and $\omega^\xi<\zeta<\omega^{\xi+1}$, $T(\theta, \omega^\xi)\in P_4^\vee(\zeta, 1)$, while $(E, T(\theta, \omega^\xi))\notin P^\wedge_2(\omega^\xi+1, 1)$.  Here, $E$ denotes the canonical basis of $T(\theta, \omega^\xi)$.  

\label{tsirelson}

\end{theorem}

We first recall that if $E$ is an FDD for $X$ and if $X\in P_1^\wedge(1+\zeta, 1)$, then $(E,X)\in P_2^\wedge(\zeta, 1)$ and, if $\zeta\geqslant \omega$, if $X\in P_1^\wedge(\zeta, 1)$, then $(E,X)\in P_2^\wedge(\zeta, 1)$. Thus the argument will show that $T(\theta, 1)\notin P_1^\wedge(3, 1)$ and $T(\theta, \omega^\xi)\notin P_1^\wedge(\omega^\xi+1, 1)$.  This settles the values of $\phi_{1,i}(\xi)$ for  all countable values of $\xi$ if $i=2,3,4$, and for all values of $\xi$ except $2$ when $i=1$.  

By dualization, this argument will settle the values of $\phi_{\infty,i}(\xi)$ for all countable $\xi$ if $i=2,3,4$ and for all countable values of $\xi$ except $2$ when $i=1$. If $T(\theta, \omega^\xi)^*\in P_i^\wedge(\omega^\xi+1, \infty)$ for $i=2,3,4$, it would imply that $T(\theta, \omega^\xi)\in P_i^\wedge(\omega^\xi+1, 1)$.  

Note that the basis $(e_n)$ of $T(\theta, \omega^\xi)$ is $\theta^{-1}$-asymptotic $\ell_1^{\omega^\xi}$, and therefore $\theta^{-k}$-asymptotic $\ell_1^{[\sss_\xi]^k}$ for each $k\in\nn$.  If $\zeta< \omega^{\xi+1}$, there exists $k\in \nn$ so that $\zeta<\omega^\xi k$.  There exists $M\in \infin$ so that $\sss_\zeta(M)\subset [\sss_\xi]^k$, so that $(e_{m_n})$ is a normalized $\theta^{-k}$-asymptotic $\ell_1^\zeta$ basis.  Therefore it is a $\theta^{-k}$-$\ell_1^\zeta$ spreading model, and $T(\theta, \omega^\xi)\in P_4^\vee(\zeta, 1)$.  A similar argument for the basis of $T(\theta, \omega^\xi)^*$, which is $\theta^{-1}$-asymptotic $c_0^{\omega^\xi}$, gives that $T(\theta, \omega^\xi)^*\in P_4^\vee(\zeta, \infty)$ for each $\zeta< \omega^{\xi+1}$.  

To complete the proof that $T(\theta, \omega^\xi)$, we will need to appeal to the following technical propositions.  The first consists of now standard arguments.    

\begin{proposition} \begin{enumerate}[(i)]\item Let $(x_i)_{i=1}^N$ be a block sequence in $B_{T(\theta, \fff)}$ and let $E_1<\ldots < E_n$ be any sets. Then $$\sum_{j=1}^n \bigl\|E_j\sum_{i=1}^N x_i\bigr\|_{T(\theta, \fff)} \leqslant N+2n.$$  \item Let $(x_i)$ be a block sequence in $B_{T(\theta, \fff)}$ and let $m_i=\max \supp(x_i)$.  Then for any $(a_i)\in c_{00}$, $$\Bigl\|\sum a_ix_i\Bigr\|_{T(\theta, \fff)} \leqslant \theta\sum |a_i| + \Bigl\|\sum a_i e_{m_i}\Bigr\|_\fff.$$  
\end{enumerate}

\label{technical 1}
\end{proposition}

\begin{proof} (i) By $1$-unconditionality, we can assume that for each $j$, $E_j$ is an interval and that $\supp(x_i)\subset \cup_{j=1}^n E_j$ for each $i$.  For each $1\leqslant j\leqslant n$, let $A_j=(i: E_jx_i\neq 0)$ and $B_j=(i: E_jx_i=x_i)$.  Then $\sum_{j=1}^n |B_j|\leqslant N$ and $|A_j|\leqslant 2+|B_j|$.  Then $$\sum_{j=1}^n \|E_j \sum_{i=1}^N x_i\|_{T(\theta, \omega^\xi)}\leqslant\sum_{j=1}^n \|\sum_{i\in A_j}x_i\|_{T(\theta, \omega^\xi)} \leqslant \sum_{j=1}^n |A_j|\leqslant 2n +\sum_{j=1}^n |B_j|\leqslant N+2n.$$

Of course, we could replace $T(\theta, \fff)$ with any space having a bimonotone basis if we required that the sets $E_j$ are intervals.

(ii) If $\|\sum a_ix_i\|_{T(\theta, \fff)}=\|\sum a_ix_i\|_{c_0}\leqslant \|\sum a_ie_{m_i}\|_\fff$, we are done.  If $\|\sum a_ix_i\|_{T(\theta, \fff)}> \|\sum a_ix_i\|_{c_0}$, there exists an $\fff$-admissible sequence $(E_j)_{j=1}^s$ so that $$\Bigl\|\sum a_ix_i\Bigr\|_{T(\theta, \fff)} = \theta\sum_{j=1}^s \bigl\|E_j \sum a_ix_i\bigr\|_{T(\theta, \fff)}.$$   

Let $$A=\{i\in \nn: E_jx_i\neq 0 \text{\ for exactly one value of\ }j\}$$ and $$B=\{i\in \nn: E_jx_i\neq 0 \text{\ for at least two values of\ }j\}.$$  For $i\in A$, let $j_i$ be the unique value of $j$ so that $E_jx_i\neq 0$.  For $i\in B$, let $j_i$ be the maximum value of $j$ so that $E_jx_i\neq 0$.  Then $i\mapsto j_i$ defines an injection from $B$ into $\{1, \ldots, s\}$.  Moreover, $\min E_{j_i}\leqslant m_i$ for $i\in B$.  Thus $(m_i)_{i\in B}$ is a spread of $(\min E_{j_i})_{i\in B}$.  We deduce that $(m_i)_{i\in B}\in \fff$.  Then \begin{align*} \Bigl\|\sum a_ix_i\Bigr\|_{T(\theta, \fff)} & = \theta \sum_{j=1}^s \bigl\|E_j \sum a_ix_i\bigr\|_{T(\theta, \fff)} \\ & = \theta\sum_{i\in A} |a_i|\|E_{j_i}x_i\|_{T(\theta, \fff)} + \sum_{i\in B} |a_i|\theta\sum_{j=1}^s \|E_j x_i\|_{T(\theta, \fff)} \\ & \leqslant \theta\sum |a_i| + \sum_{i\in B} |a_i| \leqslant \theta \sum |a_i|+ \bigl\|\sum a_ie_{m_i}\bigr\|_\fff.\end{align*}

\end{proof}

We are now ready to sketch the proof that $I_1(T(\theta,1), E, K)<\omega^2$ for any $K<\theta^{-1}$.  This proof uses rapidly increasing sequences of $\ell_1$ averages, which we will see again in the next section. It is based on E. Odell's proof that $T(1/2, 1)$ is $(2-\varepsilon)$-distortable.   If $I_1(T(\theta,1), E, K)>\omega^2$, for any $N\in \nn$,  we could recursively choose $t_n=(x_i^n)_{i=1}^{l_n}$ so that for each $1\leqslant n\leqslant N$, $t_1\verb!^!\ldots \verb!^!t_n\in d^{\omega(N-n)}(T_1(T(\theta, 1), E, K))$ and so that $\max \supp(x_{l_n}^n)/l_{n+1}< \varepsilon$ for each $1\leqslant n<N$.  Let $x_n=l_n^{-1}\sum_{i=1}^{l_n} x_i^n$ and $x=N^{-1}\sum_{n=1}^N x_n$.  Then $\|x\|_{T(\theta, 1)}\geqslant 1/K$, since it is a $1$-absolutely convex combination of a member of $T_1(T(\theta, 1), E, K)$.  Note that $\|x\|_{c_0}\leqslant 1/N$.  If $\|x\|_{T(\theta, 1)}>\|x\|_{c_0}$, we can fix $(E_j)_{j=1}^s$ $\sss_1$-admissible so that $$\|x\|_{T(\theta, 1)}= \theta\sum_{j=1}^s \|E_jx\|_{T(\theta, 1)}.$$  Let $r=\min (n: E_1x_n\neq 0)$.  Then $s\leqslant \min E_1\leqslant \max \supp(x_r)$.  Therefore \begin{align*} \|x\|_{T(\theta, 1)} & \leqslant \theta N^{-1}\sum_{n=r}^N l_n^{-1}\sum_{j=1}^s \bigl\|E_j\sum_{i=1}^{l_n} x_i^n\bigr\|_{T(\theta, 1)} \\ & \leqslant N^{-1}\|x_r\|_{T(\theta, 1)}+\theta \sum_{n=r+1}^N l_n^{-1}(l_n + 2s) \\ & \leqslant N^{-1} + \theta(1+2\varepsilon) (N-1)/N. \end{align*} Since we can do this for any $N\in \nn$ and any $\varepsilon>0$, this yields a contradiction when $K<\theta^{-1}$.  

This settles the $\xi=0$ case of Theorem \ref{tsirelson}.  For $0<\xi$, by the proof of Proposition \ref{index facts}, it is sufficient to compute the $I^a_1$ index of $T(\theta, \omega^\xi)$.  For this we will use the following, which is an easy consequence of Lemma \ref{repeated averages}

\begin{proposition} Suppose $m:\widehat{\sss}_{\omega_1}\to \nn$ is such that for $\varnothing \prec E\prec F$, $m(E)<m(F)$, and if $E<k<l$, $m(E\verb!^!k)<m(E\verb!^!l)$.  Then for any $\rho>0$ and $\xi<\omega_1$, there exist $F\in \sss_{\xi+1}$ and scalars $(a_n)_{n=1}^{|F|}$ with $\sum_{n=1}^{|F|}|a_n|=1$ and $\|\sum_{n=1}^{|F|} a_n e_{m(F|_n)}\|_\xi< \rho$.  

\label{tsirelson thing}

\end{proposition}

\begin{proof} Fix $\varepsilon>0$ and choose recursively $l_1<l_2< l_3< \ldots$ and $m_1< m_2< \ldots$ so that for each $n\in \nn$,  \begin{enumerate}[(i)]\item $l_1> (1+\varepsilon)\rho^{-1}$, \item $m((l_1, \ldots, l_n))\leqslant m_{l_n}$, \item $m_{l_n}< \varepsilon l_{n+1}/ (1+2\varepsilon)$. \end{enumerate}

Then with $M=(m_n)$ and $L=(l_n)$, $L$ is $(M, \varepsilon)$ fast growing.  Write $x^{L, \xi+1}_1=\sum_{i=1}^N a_i e_{l_i}$, and recall that $x^{L, \xi+1}_1=l_1^{-1}\sum_{i=1}^{l_1} x_i^{L, \xi}$.  Then $F=(l_i)_{i=1}^N=\supp(x^{L, \xi+1}_1)\in \sss_{\xi+1}$.  Moreover, since $$m(F|_n)= m((l_1, \ldots, l_n))\leqslant m_{l_n},$$  $$\Bigl\|\sum_{n=1}^N a_n e_{m(F|_n)}\Bigr\|_\xi \leqslant \Bigl\|\sum_{n=1}^N a_n e_{m_{l_n}}\Bigr\|_\xi= \|T_M x_1^{L, \xi+1}\|_\xi = l_1^{-1}\Bigl\|T_M\sum_{n=1}^{l_1} x_n^{L, \xi}\Bigr\|_\xi \leqslant \frac{1+\varepsilon}{l_1}<\rho.$$

\end{proof}

\begin{proof}[Proof of Theorem \ref{tsirelson}] We have already treated the $\xi=0$ case.  If $(E,T(\theta, \omega^\xi))\in P_2^\wedge(\omega^\xi+1, 1)$, then $(E,T(\theta, \omega^\xi))\in P^\wedge_3(\omega^\xi+1, 1)$.  This means that for $K< \theta^{-1}$, we can find an asymptotic block tree $(x_F)_{F\in \widehat{\sss}_{\omega^\xi+1}}$ so that for each $F\in \widehat{\sss}_{\omega^\xi+1}$, $(x_{F|_i})_{i=1}^{|F|}\in T_1(T(\theta, \omega^\xi), E, K)$.   Let $m(F)=\max \supp(x_F)$.  Extend $m:\widehat{\sss}_{\omega^\xi+1}\to \nn$ to any function $m:\widehat{\sss}_{\omega_1}\to \nn$ which still satisfies the hypotheses of Proposition \ref{tsirelson thing}.  Choose $\rho>0$ so that $\theta+ \rho < K^{-1}$ and $F\in \sss_{\omega^\xi+1}$, $(a_i)_{i=1}^{|F|}$ so that $\sum_{i=1}^{|F|}|a_i|=1$ and $\|\sum_{i=1}^{|F|} a_i e_{m(F|_i)}\|_{\omega^\xi}< \rho.$  Then $$\bigl\|\sum_{i=1}^N a_ix_{F|_i}\bigr\|_{T(\theta, \omega^\xi)} \leqslant \theta + \bigl\|\sum_{i=1}^{|F|} a_i e_{m(F|_i)}\bigr\|_{\omega^\xi} < \theta+\rho < K^{-1}.$$  This contradiction yields the result.

\end{proof}

Note that we have actually shown that $I_1(T(\theta, \omega^\xi), K)<\omega^{\xi+1}$ when $K<\theta^{-1}$.

\section{Constant reduction for reflexive $\ell_p$}

\subsection{Positive results from Krivine's theorem}

As was noted in \cite{L} and \cite{Rosenthal Krivine}, Krivine's theorem \cite{K} implies the following result.    

\begin{theorem} If $E$ is any Schauder basis, there exists $1\leqslant p\leqslant \infty$ so that for all $n\in \nn$ and $\varepsilon>0$, there exists a normalized block sequence $(x_i)_{i=1}^n$ in $E$ which is $(1+\varepsilon)$-equivalent to the unit vector basis of $\ell_p^n$.  

\label{Krivine}

\end{theorem}

It is clear that if the basis $E$ is equivalent to the $\ell_q$ (resp. $c_0$) unit vector basis, then the $p$ from Krivine's theorem can only be $q$ (resp. $\infty$).  These facts together are easily seen to imply the following quantitative version: For $K\geqslant 1$, $n\in \nn$, $\varepsilon>0$, and $1\leqslant p\leqslant \infty$, there exists $N=N(n, K, \varepsilon,p)$ so that if $(e_i)_{i=1}^N$ is $K$-equivalent to the unit vector basis of $\ell_p^N$, then there exists a normalized block $(x_i)_{i=1}^n$ of $(e_i)_{i=1}^N$ which is $(1+\varepsilon)$-equivalent to the $\ell_p^n$ basis.  Indeed, if it were not so, we could find $K\geqslant 1$, $n\in \nn$, and $\varepsilon>0$ so that for every $N\in \nn$, there exists $(e_i^N)_{i=1}^N$ $K$-equivalent to the unit vector basis of $\ell_p^N$ with no normalized block of length $n$ being $(1+\varepsilon)$-equivalent to the $\ell_p^n$ basis.  If $\uuu$ is a free ultrafilter on $\nn$, we define a norm on $c_{00}$ by $$\Bigl\|\sum a_i e_i\Bigr\|_\uuu = \underset{N\in \uuu}{\lim} \Bigl\|\sum a_i e_i^N\Bigr\|.$$  Then this norm is $K$-equivalent to the $\ell_p$ norm on $c_{00}$.  By Theorem \ref{Krivine}, there is a normalized block $(x_i)_{i=1}^n$ of the $c_{00}$ basis which is $(1+\varepsilon)^{1/2}$-equivalent to the $\ell_p^n$ basis.  But by standard arguments, this block is $(1+\varepsilon)^{1/2}$ equivalent to a block of $(e_i^N)_{i=1}^N$ for some $N\in \nn$.  This contradiction gives the quantitative version.  

This quantitative version of Krivine's theorem yields immediately that $P^\vee_1(1, p)=P^\wedge_1(1, p)$ and $P^\vee_2(1, p)=P^\wedge_2(1, p)$ for all $1\leqslant p\leqslant \infty$.  That is, crude finite representability of $\ell_p$ in a Banach space $X$ is equivalent to finite representability of $\ell_p$ in $X$, and crude block finite representability is equivalent to block finite representability.  This string of positive results continues with $P^\vee_3(1, p)=P^\wedge_3(1, p)$.  To see this, suppose $I_p^a(X, E, K)>\omega$.  Fix $\varepsilon>0$, $n\in \nn$, and let $N=N(n, K, \varepsilon,p)$ be as above.  Suppose $(x_F)_{F\in \widehat{\aaa}_N}$ is such that for each $k_1<\ldots <k_N$, $(x_{(k_1, \ldots, k_i)})_{i=1}^N$ is a block in $F$ which is $K$-equivalent to the $\ell_p^N$ basis, and that for each $F\in \aaa_N'$, $k\leqslant \supp_E(x_{F\verb!^!k})$.  Then define a norm on $(e_i)_{i=1}^N$ by $$\Bigl\|\sum_{i=1}^N a_i e_i\Bigr\| = \lim_{k_1\in \uuu}\ldots \lim_{k_N\in \uuu} \Bigl\|\sum_{i=1}^N a_i x_{(k_1, \ldots, k_i)}\Bigr\|.$$  By standard pruning arguments, we can pass to a pruned subtree and assume that every branch of the tree $(x_F)_{F\in \widehat{\aaa}_N}$ is $(1+\varepsilon)$-equivalent to $(e_i)_{i=1}^N$.  Moreover, $(e_i)_{i=1}^N$ is $K$-equivalent to the $\ell_p^N$ basis, so there exist $0=k_0<\ldots <k_n=N$ and scalars $(b_j)_{j=1}^N$ so that with $f_i=\sum_{j=k_{i-1}+1}^{k_i} b_je_j$, $(f_i)_{i=1}^n$ is $(1+\varepsilon)$-equivalent to the unit vector basis of $\ell_p^n$.  We can then define a tree $(y_F)_{F\in \widehat{\aaa}_n}$ the blocks of which are blocks of branches of $(x_F)_{F\in \widehat{\aaa}_N}$, which we deduce to be $(1+\varepsilon)^2$-equivalent to the $\ell_p^n$ basis. To see this, define $G_{(i)}=(i+1, \ldots, i+k_1)$ and, if for each $1\leqslant j\leqslant l<n$, $G_{(i_1, \ldots, i_j)}$ has been defined so that $(G_{(i_1, \ldots, i_j)})_{j=1}^l$ is successive and $|G_{(i_1, \ldots,i_j)}|=k_j-k_{j-1}$, we let $m=\max G_{(i_1, \ldots, i_l)}$ and for $i>i_l$, let $$G_{(i_1, \ldots, i_l, i)}= (m+i+1, \ldots, m+i+k_{l+1}-k_l).$$  Then for $F\in \widehat{\aaa}_n$, with $\cup_{i=1}^{|F|}G_{F|_i}=(s_i)_{i=1}^{k_{|F|}}$, we let $$y_F=\sum_{i=k_{|F|-1}+1}^{k_{|F|}} b_ix_{(s_1, \ldots, s_i)}. $$ Then by our construction, $(y_{F|_i})_{i=1}^n$ is $(1+\varepsilon)$-equivalent to $(f_i)_{i=1}^n$, and therefore $(1+\varepsilon)^2$-equivalent to the $\ell_p^n$ basis. With this, we deduce $P^\wedge_3(1, p)=P^\vee_3(1, p)$.

Modern proofs of Krivine's theorem use spreading models, and actually prove more than is stated in Theorem \ref{Krivine}. Let us say that a block $(y_i)_{i=1}^n$ of $(x_i)$ is a \emph{constant distribution block} if there exists a sequence $(a_j)_{j=1}^N$ of scalars and natural numbers $m_1<m_2<\ldots< m_{nN}$ so that for each $1\leqslant i\leqslant n$, $y_i=\sum_{j=1}^{N} a_j x_{m_{(i-1)N+j}}$.  We have a similar definition for infinite blocks $(y_i)$ of $(x_i)$.  We say that the Schauder basis $(e_i)$ is \emph{constant distribution block finitely representable in} $(x_i)$ if for every $n\in \nn$ and every $\varepsilon>0$, there exists a constant distribution block $(y_i)_{i=1}^n$ of $(x_i)$ which is $(1+\varepsilon)$-equivalent to $(e_i)_{i=1}^n$.  We note that constant distribution finite block representability is transitive.  Moreover, if $\uuu$ is a free ultrafilter on $\nn$, and if for each $N\in \nn$, $(y_i^N)$ is a normalized, constant distribution block of $(x_i)$, then the norm defined on $c_{00}$ by $$\Bigl\|\sum a_i e_i\Bigr\| = \lim_{N\in \uuu} \Bigl\|\sum a_i y_i^N\Bigr\|$$ defines a norm on $c_{00}$ which is constant distribution block finitely representable in $(x_i)$.  Moreover, it is clear that constant distribution blocks and spreading models of $(x_i)$ are themselves constant distribution block finitely representable in $(x_i)$.  Then modern proofs of Krivine's theorem for a normalized Schauder basis $(x_i)$ such that neither $\ell_1$ nor $c_0$ is block finitely representable in $(x_i)$ involve first passing to a weakly Cauchy subsequence of $(x_i)$ and then the seminormalized weakly null constant distribution block $(x_{2i}-x_{2i-1})$.  Then one passes to a spreading model and then uses the method above involving free ultrafilters several times to build further sequences, so that eventually one of these sequences is isometrically equivalent to the $\ell_p$ basis for some $1<p<\infty$. At each step, each constructed basis was constant distribution block finitely representable in the previous basis, so that by transitivity we see that $\ell_p$ is constant distribution block finitely representable in $X$.  

As before, we deduce the quantitative version:  For $\varepsilon>0$, $n\in \nn$, $K\geqslant 1$, and $1<p<\infty$, there exists $N=N(n, K, \varepsilon,p)$ so that if $(x_i)_{i=1}^N$ is any sequence $K$-equivalent to the unit vector basis of $\ell_p^N$, there exists a constant distibution block $(y_i)_{i=1}^n$ of $(x_i)_{i=1}^N$ which is $(1+\varepsilon)$-equivalent to $\ell_p^n$.  Again, the proof is by contradiction.  The only modification to the proof above we need is that when we construct the norm on $c_{00}$ using a free ultrafilter, any finite, constant distribution block in $(e_i)$ is $(1+\varepsilon)$-equivalent to a constant distribution block of $(x_i^N)_{i=1}^N$ for some $N$.  We see the following application.  Suppose that $X$ is a Banach space and $K\geqslant 1$ is such that for some $1<p<\infty$, $X$ admits a $K$-$\ell_p^{\aaa_N}$ spreading model.   That is, for each $N\in \nn$, there exists $(x_i)\subset X$ so that for every $E$ with $|E|\leqslant N$, $(x_i)_{i\in E}\in T_p(X, K)$.  Then for $n\in \nn$ and $\varepsilon>0$, we can choose $N=N(n, K, \varepsilon, p)$ and a $K$-$\ell_p^{\aaa_N}$ spreading model $(x_i)$ in $X$.  By passing to a subsequence, we can assume that for every sequence of scalars $(a_i)_{i=1}^N$ and every $m_1<\ldots <m_N$ and $l_1<\ldots <l_N$, $$\Bigl\|\sum_{i=1}^N a_ix_{m_i}\Bigr\|\leqslant (1+\varepsilon)\Bigl\|\sum_{i=1}^N a_i x_{l_i}\Bigr\|.$$  Then there exists a constant distribution block of $(x_i)_{i=1}^N$, say $y_i'=\sum_{j=1}^m b_j x_{k_j^i}$, which is $(1+\varepsilon)$-equivalent to the $\ell_p^n$ basis, where the sets $E_i=(k^i_j)_{j=1}^m$ are successive subsets of $(1, \ldots, N)$.  But then if $y_i=\sum_{j=1}^m b_j x_{im+j}$, $(y_i)$ is a constant distribution block of $(x_i)$ so that for any $E$ with $|E|\leqslant n$, $(y_i)$ is $(1+\varepsilon)^2$-equivalent to $(y'_i)_{i=1}^n$.  Therefore if we can find $\ell_p^{\aaa_N}$ spreading models with a uniform constant, we can find them with constants arbitrarily close to $1$.  It follows from the fact that these block copies of $\ell_p$ in Krivine's theorem are found spreading models that if $X$ is any infinite dimensional Banach space, either $X$ admits $c_0^{\aaa_N}$ spreading models uniformly (and therefore almost isometrically) or there exists $1\leqslant p<\infty$ so that $X$ admits $\ell_p^{\aaa_N}$ spreading models uniformly (and therefore almost isometrically).  

The previous paragraph emphasizes the difference between structures determined by trees and strcutures determined by sequences.  Odell and Schlumprecht \cite{OS1} demonstrated the existence of a Banach space $X_{OS}$ admitting no $c_0^1$ or $\ell_p^1$ spreading model.  But $X_{OS}$ admits uniform $\ell_1^{\aaa_n}$ spreading models for all $n\in \nn$.  Moreover, it is easy to see that $X_{OS}^*$ admits uniform $c_0^{\aaa_n}$ spreading models and the $p$-convexification of $X_{OS}$ admits uniform $\ell_p^{\aaa_n}$ spreading models for all $n\in \nn$.

\subsection{Distortability and consequences}

In \cite{JO}, the following observation was made:  For $1<p<\infty$, there does not exist a function $\phi_p:\ord\to \ord$ so that $P^\vee_1(\phi_p(\xi), p)\subset P^\wedge_1(\xi, p)$ for all $\xi\in \ord$.  This is because the spaces $\ell_p$, $1<p<\infty$, are distortable.  If $X$ is $K$-isomorphic to $\ell_p$, $I_p(X, K)=\infty$.  If such a $\phi_p$ existed, this would mean $I_p(X, 1+\varepsilon)=\infty$, and $X$ must admit a $(1+\varepsilon)$-isomorphic copy of $\ell_p$.  But this would imply that $\ell_p$ is not distortable.  This leads to the following definitions.  For $1<p<\infty$ and $i\in \{1,2,3,4\}$, let $M_i(p)$ denote the supremum of ordinals $\xi$ so that there exists a function $\phi_{p,i}:[0, \xi]\to \ord$ so that for each $1\leqslant \zeta\leqslant \xi$, $P^\vee_i(\phi_{p,i}(\zeta), p)\subset P^\vee_i(\zeta, p)$. The observation above from \cite{JO} shows that $M_i(p)$ is well-defined.  In \cite{JO}, the problem of finding explicit estimates for $M_1(p)$ was raised.  We will show the following.

\begin{theorem} Fix $1<p<\infty$.  \begin{enumerate}[(i)]\item $1\leqslant M_1(p)\leqslant 2$. \item $1=M_2(p)=M_3(p)$. \item $0\leqslant M_4(p)\leqslant 1$. \end{enumerate}

\label{neg constant}

\end{theorem}

We devote the remainder of this section to a quantification of the Odell-Schlumprecht distortion of $\ell_p$.  The argument is rather technical, so we give the outline first. Here, $S$ will denote Schlumprecht space, which we define below.  The equivalent norm on $\ell_p$ is defined by selecting a sequence $\varepsilon_k\downarrow 0$ and a sequence of sets $\iii_k\subset c_{00}\cap B_S$ and $\iii_k^*\subset c_{00}\cap B_{S^*}$ so that for each $k\in \nn$ and each $x\in \iii_k$, there exists $x^*\in \iii_k^*$ so that $x^*(x)>1-\varepsilon_k$ and so that for $k,l\in \nn$, $k\neq l$, $x\in \iii_k$ and $x^*\in \iii_l^*$, $|x^*(x)|< \varepsilon_{\min \{k,l\}}$.  Then one uses these sets to construct a set $\jjj_k\subset S_{\ell_1}$ which almost intersects any block subspace of $\ell_1$, and then moves this set via the Mazur map into two sets $\kkk_k\subset S_{\ell_p}$ and $\kkk^*_k\subset B_{\ell_q}$, where $q$ is the conjugate exponent to $p$.  These sets will retain the property that any block subspace of $\ell_p$ almost intersects $\kkk_k$ and so that if $x^*\in \kkk_k$ and $x\in \kkk_l$, with $k\neq l$, $|x^*(x)|$ will be small.

We let $\theta_k=1/\log_2(k+1)$, and we let $W$ be the minimal subset of $c_{00}$ so that for each $n\in \nn$ and each $E\in \fin$, $\pm e_n\in W$ and if $f_1<\ldots< f_n$, $f_i\in W$, then $\theta_n\sum_{i=1}^n E f_i\in W$.  If $f_1<\ldots <f_n$, $f_i\in W$ and $f=\theta_n\sum_{i=1}^n f_i$, we say $f$ has \emph{weight} $n$.  We can then define the Schlumprecht norm to be $$\|x\|_S=\sup_{f\in W}|f(x)|.$$  From this, we can deduce that if $x_1<\ldots<x_n$, $\|\sum_{i=1}^n x_i\|_S \geqslant \theta_n\sum_{i=1}^n \|x_i\|_S$.  From this, we deduce that $\ell_1$ is block finitely representable in any block subspace of $S$, since $1$ is the only possible value of $p$ to satisfy the conclusion of Krivine's theorem in any block subpsace of $S$.   Then for $C>1$ and $n\in \nn$, we say $x\in S_S$ is a $C$- $\ell_1^n+$ average provided there exist $x_1<\ldots <x_n$ so that $\|x_i\|\leqslant C$ for each $1\leqslant i\leqslant n$ and $x=n^{-1}\sum_{i=1}^n x_i$. By giving another quantitative version of Krivine's theorem, we deduce that for any $C>1$ and any $n\in \nn$, there exists $N=N(C,\varepsilon)$ so that for any block sequence of the basis of $S$ of length $N$, the span of this block sequence contains a $C$-$\ell_1^n+$ average.  For $\varepsilon>0$ and $n\in \nn$, we say a sequence $(x_i)_{i=1}^k$ is an RIS sequence of length $k$ with constant $(1+\varepsilon)$ if $x_i$ is a $(1+\varepsilon)$-$\ell_1^{k_i}+$ average and $k_i$ is much larger than $|\supp(x_j)|$ for $1\leqslant j<i\leqslant k$. Of course, precise quantifications must be made, for which we refer the reader to \cite{OS}. We say $x$ is a $(1+\varepsilon)$ RIS average of length $k$ if $x$ is the normalization of $\sum_{i=1}^k x_i$, where $(x_i)_{i=1}^k$ is an RIS sequence of length $k$ with constant $(1+\varepsilon)$.  The importance of such vectors is the following: If $n\in \nn$ and $f_1<\ldots <f_n$, $f_i\in W$, and if $x\approx k^{-1}\sum_{i=1}^k x_i$ is a $(1+\varepsilon)$-$\ell_1^k$ average, $$\theta_n\bigl(\sum_{j=1}^nf_j\bigr)\bigl(k^{-1}\sum_{i=1}^k x_i\bigr) \lesssim \theta_n(1+\varepsilon)(1+2n/k).$$  This value is small if $k$ is much larger than $n$.  Thus long $\ell_1+$ averages cannot take large values under functionals with small weights.  But if $n$ is much larger than $|\supp(x)|$, $$\theta_n\bigl(\sum_{j=1}^n f_j\bigr)\bigl(\sum_{i=1}^n x_i\bigr)\leqslant \theta_n|\supp(x)|,$$ which is small.  Thus we can deduce that functionals with weights which are too small or too large (depending on $k$ and $|\supp(x)|$, respectively) will give $x$ a small value.  Thus if we construct a sequence $(x_i)_{i=1}^k$ so that $k_i$ is much larger than $ | \supp(x_j)|$ for each $1\leqslant j<i\leqslant n$, then any functional in $W$ must have a weight which is either too large or too small for all but one of the members of the sequence $(x_i)_{i=1}^k$. Moreover, since $\ell_1$ is block finitely representable in every block subspace of $S$, for any $k\in \nn$ and any $\varepsilon>0$, we can find in any block subspace an RIS sequence of length $k$ with constant $(1+\varepsilon)$.  This process naturally suggests that we require averages of averages to witness the distortion with this method.  We will see below that we can find $(1+\varepsilon)$-$\ell_1^k+$ vectors in any sufficiently long normalized block in $\ell_p$, and then we must average a sufficiently long sequence of these averages to obtain RIS sequences.

For any $\varepsilon_k\downarrow 0$, we can fix $s_k\in \nn$, $s_k\uparrow \infty$, and $\delta_k>0$ so that if $\iii_k\subset S_S$ consists of all $(1+\delta_k)$ RIS averages of RIS sequences of legnth $s_k$, and if $\iii_k^*\subset B_{X^*}$ consists of vectors of the form $\theta_{s_k}\sum_{j=1}^{s_k}x^*_j$, where $(x_j)_{i=1}^{s_k}$ is a block sequence in $B_{X^*}$, then \begin{enumerate}[(i)]\item for each $k\in \nn$ and each $x\in \iii_k$, there exists $x^*\in \iii_k^*$ so that $x^*(x)\geqslant 1-\varepsilon_k$, \item for each $j\neq k$, each $x\in \iii_k$, and each $x^*\in \iii_j^*$, $|x^*(x)|\leqslant\varepsilon_{\min \{j,k\}}$, and \item for each $k\in \nn$ and any block subspace $X$ of $S$, $X\cap \iii_k\neq \varnothing$. 

\end{enumerate}

If we are interested in quantification, however, we can state the following quantified version of (iii) above: If $(x_t)_{t\in \ttt_{\omega^2}}\subset S_S$, then for any $k\in \nn$, there exists $t\in \ttt_{\omega^2}$ and $x\in \iii_k\cap \text{span}(x_{t|_i}:1\leqslant i\leqslant |t|)$.  We choose recursively $t_1, \ldots, t_k$ so that $t_1\verb!^!\ldots \verb!^!t_i\in d^{\omega(s_k-i)}(\ttt_{\omega^2})$ so that $|t_i|$ is sufficiently long to guarantee the existence of a $(1+\varepsilon)$-$\ell_1^{k_i}+$ average in $\text{span}(x_{t_i|j}:1\leqslant j\leqslant |t_i|)$, where $k_i$ depends on the choices of $(1+\varepsilon)$-$\ell_1^{k_j}+$ averages for $1\leqslant j<i$.  In what follows, if $x=(a_n)$ and $y=(b_n)$ are scalar sequences, we will let $xy=(a_nb_n)$, $|x|=(|x_n|)$, and $x^r=(\sgn(x_n) |x_n|^r)$ for $r>0$.  Recall the definition of the Mazur map, $M_p:S_{\ell_1}\to S_{\ell_p}$, $M_px= x^{1/p}$, and that this is a uniform homeomorphism \cite{R}.  We define the sets $\jjj_k\subset S_{\ell_1}$, $\kkk_k\subset S_{\ell_p}$, and $\kkk_k^*\subset S_{\ell_q}$ by $$\jjj_k=\Bigl\{\frac{u^*u}{\|u^*u\|_{\ell_1}}: u\in \iii_k, u^*\in \iii^*_k, \|u^*u\|_{\ell_1}\geqslant 1-\varepsilon_k\Bigr\},$$ $$\kkk_k= M_p(\jjj_k),$$ $$\kkk^*_k= M_q(\jjj_k).$$  We remark that since the basis of Schlumprecht space is $1$-unconditional, the sets $\kkk_k$ and $\kkk_k^*$ are $1$-unconditional.  By this, we mean that $a\in \kkk_k$ if and only if $|a|\in \kkk_k$.  

If $2\leqslant p$, let $\varepsilon_k'=\varepsilon_k^{2/p}/(1-\varepsilon_1)$, and if $1<p<2$, let $\varepsilon_k'=\varepsilon_k^{2/q}/(1-\varepsilon_1)$.  Note that if $x= M_p(s)\in \kkk_k=M_p(\jjj_k)$, then $x^*=M_q(s)= \kkk^*_k=M_q(\jjj_k)$, and $x^*(x)= 1$.  Thus for each $k\in \nn$ and $x\in \kkk_k$, there exists $x^*\in \kkk_k^*$ so that $x^*(x)=1$.  It was shown in \cite{OS} that if $k\neq j$ and if $x^*\in \kkk_k^*$, $x\in \kkk_j$, then $|x^*(x)|\leqslant \varepsilon_{\min \{j,k\}}'$.  Thus we have the $\ell_p$ analogues of (i) and (ii) above.  What remains is a quantified version of (iii).  The facts in the next lemma are contained in \cite{OS}, without the quantification of the supports of the blocking required.

\begin{lemma} \begin{enumerate}[(i)]\item For $n\in \nn$ and $\varepsilon>0$, there exists $N=N(\varepsilon, k)$ so that if $(y_i)_{i=1}^N\subset S_{\ell_1}$ is a block sequence, there exist $y\in \emph{\text{co}}_1(y_i:1\leqslant i\leqslant N)$, a $(1+\varepsilon)$ $\ell_1^n+$ average $u\in S_S$, and $u^*\in B_{S^*}$ so that $\|y - u^*u\|_{\ell_1}< \varepsilon$ and $\supp(u)\subset \supp(y)$. \item If $(y_t)_{t\in \ttt_{\omega^2}}\subset S_{\ell_1}$ is a block tree, $\varepsilon>0$, and $j\in \nn$, there exists $t\in \ttt_{\omega^2}$, $y\in \emph{\text{co}}_1(y_{t|_i}:1\leqslant i\leqslant |t|)$ and $y_0\in \jjj_j$ so that $\|y- y_0\|_{\ell_1}<\varepsilon$. \item  If $(x_F)_{t\in \ttt_{\omega^2}}\subset S_{\ell_1}$ is a block tree, $\varepsilon>0$, and $j\in \nn$, there exists $t\in \ttt_{\omega^2}$, $x\in \emph{\text{co}}_p(x_{t|_i}:1\leqslant i\leqslant |t|)$ and $x_0\in \jjj_j$ so that $\|x- x_0\|_{\ell_p}<\varepsilon$.

\end{enumerate}

\label{prime lemma}
\end{lemma}

\begin{proof}(i)  We note that the statement here is stronger than the statement given in \cite{OS}, but the statement here is implicit in the proof of Lemma 3.5 from \cite{OS}.  

(ii) As outlined above, we recursively choose $t_1<\ldots < t_{s_j}$, $t_1\verb!^!\ldots \verb!^!t_i\in d^{\omega(s_j-i)}(\ttt_{\omega^2})$, $y_i\in \text{co}_1(y_{t_i|_k}:1\leqslant k\leqslant |t_i|)$, $u_i\in S_S$, and $u_i^*\in B_{S^*}$ according to (i) so that the normalization $u$ of $\sum_{i=1}^{s_j} u_i$ is a member of $\iii_j$, $u^*=(1+\rho)^{-1}\theta_{s_j}\sum_{i=1}^{s_j} u_i^*\in \iii^*_j$, where $\rho>0$ is chosen small depending on $\varepsilon$, $\|u^*_iu_i- y_i\|_{\ell_1}< \varepsilon/2$, and $\|u^*u\|_{\ell_1}> 1-\varepsilon_j$.  Then $y_0=\frac{u^*u}{\|u^*u\|_{\ell_1}}\in \jjj_j$ and, with $y=s_j^{-1}\sum_{i=1}^{s_j} y_i$, $$\|y- y_0\|_{\ell_1}\leqslant \|y - u^*u\|_{\ell_1} + \|u^*u - y_0\|_{\ell_1} < 2(\varepsilon/2)= \varepsilon.$$

(iii) This follows from (ii) and the fact that the Mazur map $M_p:S_{\ell_1}\to S_{\ell_p}$ is a uniform homeomorphism which preserves supports and takes $1$-absolutely convex combinations of block sequences to $p$-absolutely convex combinations of block sequences.

\end{proof}

\begin{lemma} Fix $k\in \nn$ and let $X_k$ be the completion of $c_{00}$ under the norm $$\|x\|=\max \{\varepsilon_k' \|x\|_{\ell_p}, \sup_{x^*\in \kkk_k^*} x^*(x)\}.$$  Then if $E$ is the $c_{00}$ basis, $I_p(X_k,E, K)< \omega^2$ for any $1\leqslant K<(\varepsilon_k')^{-1}$, and $I_p(X_k, E, (\varepsilon_k')^{-1})=\infty$.

\label{andy daly}
\end{lemma}

\begin{proof} Note that $\varepsilon_k'\|x\|_{\ell_p}\leqslant \|x\|\leqslant \|x\|_{\ell_p}$, which gives the last statement.

Since $I_p(X, E, K)$ is a successor, it is sufficient to show that if $I_p(X, E, K)>\omega^2$, $K\geqslant (\varepsilon_k')^{-1}$. If $I_p(X, E, K)>\omega^2$, we can find a block tree $(g_t)_{t\in \ttt_{\omega^2}}$ so that for each $t\in \ttt_{\omega^2}$, $(g_{t|_i})_{i=1}^{|t|}\in T_p(X, E, K)$.   

Fix $\delta>0$.  We can apply Corollary \ref{approximate monochromatic} and assume there exists $\theta>0$ so that for every $t\in \ttt_{\omega^2}$, $\theta\leqslant \|g_t\|_{\ell_p} \leqslant \theta(1+\delta)$.  Let $x_t= g_t/\|g_t\|_{\ell_p}$.  Choose $t\in \ttt_{\omega^2}$, $x\in \text{co}_p(x_{t|_i}: 1\leqslant i\leqslant |t|)$, and $x_0\in \kkk_k$ so that $\|x- x_0\|_{\ell_p}<\delta$.  Then as we noted earlier, $\|x_0\|=1$ and therefore $\|x\|\geqslant 1- \|x-x_0\|\geqslant 1- \|x-x_0\|_{\ell_p}>1-\delta$.  Let $(a_i)_{i=1}^{|t|}\in S_{\ell_p^{|t|}}$ be such that $x= \sum_{i=1}^{|t|} a_i x_{t|_i}$, and let $g= \sum_{i=1}^{|t|} a_i g_{t|_i}$.  Then \begin{align*} 1 & \geqslant \|g\| = \Bigl\|\sum_{i=1}^{|t|} a_i g_{t|_i}\Bigr\| \\ & = \Bigl\|\sum_{i=1}^{|t|} a_i \|g_{t|_i}\|_{\ell_p} x_{t|_i}\Bigr\| \geqslant \theta \Bigl\|\sum_{i=1}^{|t|} a_i  x_{t|_i}\Bigr\| \\ & = \theta \|x\| \geqslant \theta (1- \delta), \end{align*} and $\theta\leqslant (1-\delta)^{-1}$.  Here we have used $1$-unconditionality of the $c_{00}$ basis with respect to $\|\cdot\|$.    

Next, choose any $j>k$.  Choose $s\in \ttt_{\omega^2}$, $x'\in \text{co}_p(x_{s|_i}: 1\leqslant i\leqslant |s|)$, and $x_0'\in \kkk_j$ so that $\|x'- x_0'\|_{\ell_p}< \delta$.  Then $\|x_0'\|\leqslant \varepsilon_k'$, and $\|x'\|\leqslant \varepsilon_k'+\delta$.  Choose $(b_i)_{i=1}^{|s|}\in S_{\ell_p^{|s|}}$ so that $x'=\sum_{i=1}^{|s|} b_i x_{s|_i}$ and let $g'=\sum_{i=1}^{|s|}b_i g_{s|_i}$.  Then \begin{align*} K^{-1} & \leqslant \|g'\| = \Bigl\|\sum_{i=1}^{|s|} b_i g_{s|_i}\Bigr\|   \leqslant \theta(1+\delta)\Bigl\|\sum_{i=1}^{|s|} b_i x_{s|_i}\Bigr\| \\ & \leqslant \theta(1+\delta) (\varepsilon'_k+\delta) \leqslant \frac{(1+\delta) (\varepsilon_k'+\delta)}{1-\delta}, \end{align*} where again we have used $1$-unconditionality.  Since $\delta>0$ was arbitrary, we deduce that $K^{-1}\leqslant \varepsilon_k'$.

\end{proof}

\begin{proof}[Proof of Theorem \ref{neg constant}] We have already shown that $M_1(p), M_2(p), M_3(p)\geqslant 1$.   We have just shown that $M_2(p)=1$, so that $M_4(p),M_3(p)\leqslant 1$.  We have shown in the proof of Proposition \ref{index facts} (v) that for any Banach space $X$ with FDD $E$ and any $K\geqslant 1$ and $\varepsilon>0$, $I_p(X,K)\leqslant \omega I_p^a(X,E, K+\varepsilon)$, from which we deduce $M_1(p)\leqslant 1+M_2(p)=2$.

\end{proof}

We note here that this is the stongest possible result about the negative results concerning existence of uniform structures given the existence of crude structures.  That is, for each $\xi<\omega^2$, there exists a constant $K$ such that if $X$ is any Banach space with $I_p(X)>\omega$, then $I_p(X, K)>\xi$.  It follows from the proof of Proposition \ref{index facts} (iv) where if $\xi<\omega n$, $n\in \nn$, then we can take $K=2n+\varepsilon$ for any $\varepsilon>0$.

\section{Subspace and quotient estimates}

As we have seen already, there is a strong dichotomy between the $\ell_p$ and $c_0$ cases and the $\ell_p$ cases for $1<p<\infty$ because we must only be concerned with one-sided estimates when $p=1$ or $\infty$.  We invesgiate here another instance when we deduce positive results due to only needing to verify a one-sided estimate: The case of passing to a quotient.  If $X$ is a Banach space and $Y$ is a closed subspace, any sequence $1$-dominated by the $\ell_p$ basis in $X$ is sent by the quotient map $Q:X\to X/Y$ to a sequence also $1$-dominated by the $\ell_p$ basis.  Thus we have the following dichotomy: Given an $\ell_p$ structure in $X$, the image of this $\ell_p$ structure under the quotient map must also be an $\ell_p$ structure, or $p$-absolutely convex blocks of the branches must have small quotient norm.  In the first case, we find an $\ell_p$ structure in $X/Y$, and in the second case, we find an $\ell_p$ structure which is only a small perturbation from being in $Y$.  Assuming the perturbation is small enough, we will find an $\ell_p$ structure in $Y$.  We note that again, $\ell_1$ plays a special role for two reasons: The first is that $\ell_1$ structures in $X/Y$ imply $\ell_1$ structures in $X$.  The second is that small, uniform perturbations of sequences behaving like the $\ell_1$ basis also behave like the $\ell_1$ basis, but this is not true for the $\ell_p$, $1<p<\infty$, or $c_0$ bases.  This is easily seen by considering $(\theta e_n+f_n)\subset \ell_1\oplus \ell_p$, where $\theta>0$ is small, $(e_n)$ is the $\ell_1$ basis, and $(f_n)$ is the $\ell_p$ basis.  

If $X$ is a Banach space containing a copy of $\ell_p$, and if $Y$ is any closed subspace, then either $Y$ or $X/Y$ contains a copy of $\ell_p$.  To see this, suppose $(x_n)\subset X$ is equivalent to the $\ell_p$ basis.  Let $Q:X\to X/Y$ denote the quotient map.  We consider two cases.  In the first case, there exist $N\in \nn$ and  $\varepsilon>0$ so that for all $(a_i)_{i=N+1}^\infty \in c_{00}\cap S_{\ell_p}$, $\|\sum_{i=N+1}^\infty a_iQx_i\|\geqslant \varepsilon$.  In this case, $(x_i)_{i>N}$, $(Qx_i)_{i>N}$, and the $\ell_p$ basis are all equivalent.  In the second case, for any $\varepsilon_n\downarrow 0$, we can find a $p$-absolutely convex block $(u_n)$ of $(x_n)$ so that for each $n\in \nn$, $\|Qu_n\|<\varepsilon_n$.  If we choose $(y_n)\subset Y$ so that $\|u_n-y_n\|<\varepsilon_n$, and if $\varepsilon_n$ is chosen tending to $0$ rapidly enough, depending on the constant of equivalence between $(x_n)$ and the $\ell_p$ basis, $(y_n)$ will be equivalent to the $\ell_p$ basis.  More precisely, in the second case, if $(x_n)$ is $K$-equivalent to the $\ell_p$ basis, then $(u_n)$ is $K$-equivalent to the $\ell_p$ basis. For any $C>0$, by choosing $\varepsilon_n\downarrow 0$ appropriately we can guarantee that the equivalence constant between $(y_n)$ and the $\ell_p$ basis is at most $C$.   

As a consequence of this, we deduce that the class of Banach spaces not containing $\ell_p$ is closed under taking direct sums.  To restate, if $I_p(Y), I_p(Z)<\infty$, then $I_p(Y\oplus Z)<\infty$.  Dodos has shown that this result is actually uniform on $\textbf{SB}$ \cite{D}.  That is, there exists a function $\psi_p:[0, \omega_1)^2\to [0, \omega_1)$ so that for $Y,Z\in \textbf{SB}$, $I_p(Y\oplus Z)\leqslant \psi_p(I_p(Y), I_p(Z))$.  This was shown using descriptive set theoretic techniques, and was actually shown for a larger class of spaces than the $\ell_p$ spaces.  The methods used do not provide explicit estimates of the function $\psi_p$, and in \cite{D}, Dodos asks for some explicit estimate.  We provide the first explicit estimate for this function below.  Note that we do not require separability, nor do we assume countability of any ordinals.

\begin{theorem} Fix $1\leqslant p\leqslant \infty$.  For any Banach space $X$, any closed subspace $Y$, and $1\leqslant K<C$, \begin{enumerate}[(i)]\item $I_p(X, K)\leqslant I_p(X/Y) I_p(Y, C)$. \item $I_1(X, K)\leqslant I_1(X/Y, 3K)I_1(Y, 3K)$. \item $\ssm_p(X, K)\leqslant \ssm_p(X/Y)+\ssm_p(Y, C)$. \item $\ssm_1(X, K)\leqslant \ssm_1(X/Y, 3K)+\ssm_1(Y, 3K)$. \item $\ssm_\infty(X)\leqslant \max\{\ssm_\infty(X/Y), \ssm_\infty(Y)\}$. \end{enumerate}

\label{three space}

\end{theorem}

\begin{corollary} For any ordinal $\xi$, $I_1(\cdot)<\omega^{\omega^\xi}$ and $I_1(\cdot)\leqslant \omega^{\omega^\xi}$ are three-space properties on the class \emph{\textbf{Ban}}.   \end{corollary}

This is because if $\alpha ,\beta<\omega^{\omega^\xi}$, $\alpha\beta<\omega^{\omega^\xi}$.  This observation gives the following corollary as well.  

\begin{corollary} Fix $1\leqslant p\leqslant \infty$.  If $X\in\emph{\textbf{Ban}}$ and $Y$ is a closed subpace of $X$ so that $I_p(X/Y), I_p(Y)<\omega^{\omega^\xi}$, then $I_p(X)< \omega^{\omega^\xi}$.  

\end{corollary}

In the case that $Y$ is complemented, we can reverse the order of multiplication.  In this case, we obtain the slightly stronger result that if $I_p(Y), I_p(Z)\leqslant \omega^{\omega^\xi}$ and at least one of these inequalities is strict, then $I_p(Y\oplus Z)\leqslant \omega^{\omega^\xi}$.  

Of course, $I_p(\cdot)<\omega^{\omega^\xi}$ is not a three-space property for $1<p$, since $\ell_p$ and $c_0$ are quotients of $\ell_1$.  Thus for $1<p$, $I_p(X/Y)$ cannot be controlled by $I_p(X)$.  

As we will see below, the $\ssm_1(X)= \sup_{K\geqslant 1} \ssm_1(X, K)$ may be attained. This fact has been hinted at by the existence of Banach spaces admitting uniform $\ell_1^{\aaa_n}$ spreading models but no $\ell_1^1$ spreading model.  Thus the estimate (iv) of Theorem \ref{three space} does not imply the following, which was proven in \cite{BCM}.  

\begin{theorem} For $\xi<\omega_1$, admitting no $\ell_1^{\omega^\xi}$ spreading model is a three space property on the class of Banach spaces.  That is, if $X\in \emph{\textbf{Ban}}$ and $Y\leqslant X$, $\ssm_1(Y), \ssm_1(X/Y)\leqslant \omega^\xi$ implies $\ssm_1(X)\leqslant \omega^\xi$.  

\label{BCM theorem 2}

\end{theorem}

This is a simple consequence of Lemma \ref{easy coloring}.  If $(x_n)$ is a $K$-$\ell_1^{\omega^\xi}$ spreading model in $X$, define the function $f$ on $\widehat{\sss}_{\omega^\xi}$ by $f(E)=0$ if $\min \{\|Qx\|:x\in \text{co}_1(x_n:n\in E)\}\geqslant 1/3K$ and $f(E)=1$ otherwise.  If there exists $M\in \infin$ so that $f|_{\widehat{\sss}_{\omega^\xi}(M)}\equiv 1$, then the image $(Qx_{m_n})$ of the subsequence $(x_{m_n})$ is a $3K$-$\ell_1^{\omega^\xi}$ spreading model.  Otherwise there exist $E_1<E_2<\ldots$ so that $f(E_i)=0$ and for each $E\in \sss_{\omega^\xi}$, $\cup_{n\in E}E_n\in \sss_{\omega^\xi}$.  For each $n\in \nn$, pick $u_n\in \text{co}_1(x_i:i\in E_n)$ so that $\|Qu_n\|< 1/3K$ and $y_n\in Y$ so that $\|u_n-y_n\|< 1/3K$.  Then $(u_n)$ is a $K$-$\ell_1^{\omega^\xi}$ spreading model.  This implies $(2^{-1}y_n)$ is a $3K$-$\ell_1^{\omega^\xi}$ spreading model.

\begin{proof}[Proof of Theorem \ref{three space}]

(i) If $I_p(X, K)=\infty$, then $X$ contains a $K$-isomorphic copy of $\ell_p$. We have already treated this case above, so assume $I_p(X, K)<\infty$.  If $I_p(X/Y)=\infty$ or $I_p(Y, C)=\infty$, there is nothing to prove.   Therefore we can assume $\zeta= I_p(X/Y)<\infty$ and $\xi= I_p(Y, C)<\infty$.  Assume $I_p(X, K)>\zeta\xi$ and choose $(x_t)_{t\in \ttt_{\zeta\xi}}$ so that for each $t\in \ttt_{\zeta\xi}$, $(x_{t|_i})_{i=1}^{|t|}\in T_p(X, K)$.  Define $f:C(\ttt_{\zeta\xi})\to \rr$ by $$f(c)=\inf \{\|Qx\|: x\in \text{co}_p(x_t: t\in c)\}.$$  By Lemma \ref{local coloring}, either there exists an order preserving $i:\ttt_\zeta\to \ttt_{\zeta\xi}$ so that $$\inf \{f\circ i(c): c\in C(\ttt_\zeta)\}=\varepsilon>0,$$ or there exists an order preserving $j:\ttt_\xi\to C(\ttt_{\zeta\xi})$ so that for each $t\in \ttt_\xi$, $f\circ j(t)<\varepsilon_{|t|}$, where $(\varepsilon_n)\subset (0,1)$ is to be determined.  

In the first case, $(Qx_{i(t)})_{t\in \ttt_\zeta}$ witnesses the fact that $I_p(X/Y)>\zeta$, a contradiction. In the second case, we can choose for each $t\in \ttt_\xi$ some $u_t\in \text{co}_p(x_s: s\in j(t))$ so that $\|Qu_t\|< \varepsilon_{|t|}$. Then since $j$ is order preserving, $(u_{t|_i})_{i=1}^{|t|}$ is a $p$-absolutely convex block of a member of $T_p(X, K)$, and is itself a member of $T_p(X, K)$.  For each $t\in \ttt_\xi$, we can choose $y_t\in Y$ so that $\|u_t-y_t\|<\varepsilon_{|t|}$.  Then for each $t\in \ttt_\xi$ and $(a_i)_{i=1}^{|t|}\in S_{\ell_p^{|t|}}$, \begin{align*} \Bigl\|\sum_{i=1}^{|t|} a_i y_{t|_i}\Bigr\| & \leqslant \Bigl\|\sum_{i=1}^{|t|} a_i u_{t|_i}\Bigr\| + \sum_{i=1}^{|t|}|a_i|\|u_{t|_i}- y_{t|_i}\| \\ & \leqslant 1+ \sum_{i=1}^{|t|} \varepsilon_{t|_i} \leqslant 1+\sum\varepsilon_n.\end{align*}  Similarly, $\|\sum_{i=1}^{|t|} a_i y_{t|_i}\|\geqslant 1/K- \sum \varepsilon_n$.  If $\varepsilon= \sum \varepsilon_n$ is chosen so that $(1+\varepsilon)(K^{-1}-\varepsilon)^{-1}<C$, we deduce that $((1+\varepsilon)^{-1}y_{t|_i})_{i=1}^{|t|}\in T_p(Y, C)$.  This means $I_p(Y,C)>\xi$, another contradiction.

(ii) If $X$ contains a $K$-isomorphic copy of $\ell_1$, then either $X/Y$ or $Y$ contains a $(1+\varepsilon)$-isomorphic copy of $\ell_1$ for all $\varepsilon>0$.   Assume $I_1(X, K)<\infty$, which implies that $\zeta= I_1(X/Y, 3K)$, $\xi=I_1(Y, 3K)<\infty$.  The proof is similar to the previous case.  If $I_1(X, K)>\zeta\xi$, choose $(x_t)_{t\in \ttt_{\zeta\xi}}$ so that $(x_{t|_i})_{i=1}^{|t|}\in T_1(X,K)$.  Define a function $f:C(\ttt_{\zeta\xi})\to \{0,1\}$ by letting $f(c)=1$ if $$\inf \{\|Qx\|: x\in \text{co}_1(x_t: t\in c)\}\geqslant 1/3K,$$ and $f(c)=0$ otherwise.  Either there exists an order preserving $i:\ttt_\zeta\to \ttt_{\zeta\xi}$ so that $f\circ i(c)=1$ for all $c\in C(\ttt_\zeta)$, in which case $I_1(X/K, 3K)>\zeta$, or there exists an order preserving $j:\ttt_\xi\to C(\ttt_{\zeta\xi})$ so that $f\circ j\equiv 0$.  The first case immediately yields a contradiction.  In the second case, we choose for each $t\in \ttt_\xi$ some $u_t\in \text{co}_1(x_s:s\in j(t))$ so that $\|Qu_t\|<1/3K$ and some $y_t\in Y$ so that $\|u_t-y_t\|<1/3K$.   Then for any $t\in \ttt_\xi$ and $(a_i)_{i=1}^{|t|}\in S_{\ell_1^{|t|}}$, $$\Bigl\|\sum_{i=1}^{|t|} a_i y_{t|_i}\Bigr\|\geqslant 1/K- \sum_{i=1}^{|t|}|a_i| \|u_{t|_i}- y_{t|_i}\| \geqslant 2/3K.$$  But $\|y_t\|\leqslant \|u_t\|+1/3K <2$, so $(2^{-1}y_t)_{t\in \ttt_\xi}\subset B_Y$ witnesses the fact that $I_1(Y, 3K)>\xi$, another contradiction.

(iii) If any of the three spaces contains $\ell_p$, we finish as in (i).  Assume each index does not exceed $\omega_1$.  In this case, if either $\ssm_p(X/Y)=\omega_1$ or $\ssm_p(Y,C)=\omega_1$, we have the result.  So assume $\zeta= \ssm_p(X/Y)<\omega_1$ and $\xi=\ssm_p(Y, C)<\omega_1$.  If $\ssm_p(X, K)>\zeta+\xi$, then there must exist some $(x_n)\subset X$ which is a $K$-$\ell_p^{\sss_\xi[\sss_\zeta]}$ (resp. $c_0^{\sss_\xi[\sss_\zeta]}$ if $p=\infty$) spreading model.  Suppose that for some $N\in \nn$ and some $\varepsilon>0$ it is true that for every $E\in \sss_\zeta\cap (N, \infty)^{<\omega}$ and every $(a_n)_{n\in E}\in S_{\ell_p^{|E|}}$, $\varepsilon \leqslant \|\sum_{n\in E} a_n Qx_n\|$.  Then $(Qx_n)_{n>N}$ is a $\varepsilon^{-1}$-$\ell_p^{\sss_\zeta\cap (N, \infty)^{<\omega}}$ spreading model, a contradiction.  Thus we can find $E_1<E_2<\ldots$, $E_i\in \sss_\zeta$, and $(a_j)_{j\in E_i}\in S_{\ell_p^{|E_i|}}$ so that with $u_i=\sum_{j\in E_i} a_j x_j$, $\|Qu_i\|<\varepsilon_i$, where $\varepsilon_i\downarrow 0$ rapidly.  Then $(u_i)$ is a $K$-$\ell_p^\xi$ spreading model and, if $y_i\in Y$ is such that $\|u_i-y_i\|<\varepsilon_i$ for all $i\in \nn$, $((1+\varepsilon)^{-1}y_i)\subset Y$ will be a $C$-$\ell_p^\xi$ spreading model for an appropriate choice of $(\varepsilon_i)$.

(iv) The proof is the obvious combination of the methods of (ii) and (iii).

(v) Again, we need only consider the case in which $X$ does not contain $c_0$.  Fix $\xi<\omega_1$ and assume $(x_n)\subset X$ is a $c_0^\xi$ spreading model.  If $\xi=0$, there is nothing to prove.  Otherwise, $(x_n)$ is weakly null.  If there exists $N\in \nn$ so that $(Qx_n)_{n>N}$ is bounded away from zero, then some subsequence is a $c_0^\xi$ spreading model.  This is because some subsequence is basic, so the lower $c_0$ estimates come automatically.  The upper $c_0$ estimates come from comparison to $(x_n)_{n>N}$.  Otherwise, there exists a subsequence of $(x_n)$, which we can assume is the whole sequence, so that $\|Qx_n\|\to 0$.  Then some subsequence is an arbitrarily small perturbation of a sequence in $Y$.  By taking the perturbations small enough, we guarantee the existence of a $c_0^\xi$ spreading model in $Y$.

\end{proof}

There is a natural exception to the remark above about uniform perturbations of $\ell_p$ sequences not necessarily exhibiting $\ell_p$ behavior.  This exception is uniform perturbations of sequences with uniformly bounded length. If $I_p(X/Y), I_p(Y)\leqslant \omega$, the estimate above only yields $I_p(X)\leqslant \omega^2$.    In this case one can perform the finite version of the arguments above to easily see that if $I_p(X)>\omega$, $\max\{I_p(X/Y), I_p(Y)\}>\omega$.   This is because if $I_p(X)>\omega$ while $I_p(X/Y), I_p(Y)$, then for any $\varepsilon>0$ and $N\in \nn$, we can find $(x_i)_{i=1}^N\in T_p(X, 1-\varepsilon)$. Fix $n\in \nn$ arbitrary.  With $m= I_p(X/Y, n/\varepsilon)$ and $N=mn$, we deduce that for each $1\leqslant j\leqslant n$, there exists $(a_i)_{i=(j-1)m+1}^{jm}\in S_{\ell_p^m}$ so that $\|Q\sum_{i=(j-1)m+1}^{jm} a_ix_i\|<\varepsilon/n$.  With $u_i=\sum_{i=(j-1)m+1}^{jm}a_ix_i$ and $y_i\in Y$ so that $\|u_i- y_i\|<\varepsilon/n$, $$1-2\varepsilon \leqslant \Bigl\|\sum_{i=1}^n b_iy_i\Bigr\|\leqslant 1+\varepsilon$$ for any $(b_i)_{i=1}^n \in S_{\ell_p^n}$.  Since we can do this for any $n\in \nn$, we deduce $I_p(Y)>\omega$.

We now deduce a few natural consequences of our coloring lemmas which help us estimate the $I_p$ index of infinite direct sums.  Recall that if $U$ is a Banach space with $1$-unconditional (not necessarily ordered) basis $E$, and if for each $e\in E$, $X_e$ is a Banach space, $$\bigl(\oplus X_e\bigr)_{e\in E}=\{(x_e)_{e\in E}: x_e\in X_e, \sum_{e\in E} \|x_e\|e\in U\}$$ is a Banach space with the norm $\|(x_e)\|= \|\sum_{e\in E}\|x_e\|e\|_U$.  We denote this direct sum by $X_E$.  Analogously to the case of an FDD, we let $\supp_E((x_e)_{e\in E})=\{e\in E: x_e\neq 0\}$ and we let $c_{00}(E)$ denote the vectors in $X_E$ which have finite support.  We note that such vectors are dense in $X_E$.  For $S\subset E$ finite, we let $P_S:X_E\to X_E$ be the projection given by $P_S(x_e)= (1_S(e) x_e)$.  We let $\pi_{e_0}:\bigl(\oplus X_e\bigr)_{e\in E}\to X_{e_0}$ be the coordinate projection into the summand $X_{e_0}$ defined by $\pi_{e_0}((x_e)_{e\in E})= y_{e_0}$.  Last, we let $\Pi:X_E\to U$ be the map defined by $\Pi(x)=\sum_{e\in E} \|\pi_e(y)\|e$.  We observe that if $(x_i)$ is a sequence of vectors with pairwise disjoint supports in $X_E$, then $(x_i)$ is isometrically equivalent to the sequence $(\Pi x_i)$ in $U$.  

\begin{proposition}

Fix $1\leqslant p\leqslant \infty$.  Suppose that $U$ is a Banach space with $1$-unconditional basis $E$ and $\alpha\in \emph{\ord}$ is such that for each $e\in E$, $X_e$ is a Banach space with $I_p(X_e)< \omega^{\omega^\alpha}$ (or $I_p(X_e)\leqslant \omega^{\omega^\alpha}$ in the case that $p=1$).  Then for any $1\leqslant K<C$, $I_p(X_E, K) \leqslant \omega^{\omega^\alpha}I_p(U, C).$  

\label{infinite direct sum}

\end{proposition}

\begin{proof}  (i) Let $\zeta=\omega^{\omega^\alpha}$.  If $I_p(U, C)=\infty$, there is nothing to prove.  So assume $\xi=I_p(U,C)<\infty$ and $I_p(Y, K)>\zeta\xi$.  By replacing any $K$ with any strictly larger number which is still less than $C$, we deduce there exists $(x_t)_{t\in \ttt_{\zeta\xi}}$ so that for each $t\in \ttt_{\zeta\xi}$, $(x_{t|_i})_{i=1}^{|t|}\in T_p(X, K)\cap c_{00}(E)^{<\omega}$.  For $c\in C(\ttt_{\zeta\xi})$, let $S_c=\cup_{s\preceq \max c}\supp_E(x_s)$, and note that this is a finite subset of $E$. Let $S_\varnothing=\varnothing$. For each $c\in C(\ttt_{\zeta\xi})\cup \{\varnothing\}$, define $f_c:C(\ttt_{\zeta\xi})\to \rr$ by $$f_{c}(c')= \min \{\|P_{S_{c}} x\|: x\in \text{co}_p(x_s:s\in c')\}.$$

 Note that for any order preserving $i:\ttt_\zeta\to \ttt_{\zeta\xi}$ and for any $c\in C(\ttt_{\zeta\xi})\cup \{\varnothing\}$, $$\inf \{f_c\circ i(c'): c'\in C(\ttt_\zeta)\}=0.$$ If this were not so, then with $0<\varepsilon= \inf \{f_c\circ i(c'):c'\in C(\ttt_\zeta)\}$, we deduce $(P_{S_c}x_{i(s)})_{s\in \ttt_\zeta}\subset X_{S_c}$ witnesses the fact that $I_p(X_{S_c}, \varepsilon^{-1})>\zeta$.  But since $X_{S_c}$ is a finite direct sum, this contradicts Theorem \ref{three space}.  

Fix $(\varepsilon_n)$ rapidly decreasing to zero and fix an order preserving $j:\mt_\xi\to C(\ttt_{\zeta\xi})\cup \{\varnothing\}$ so that for each $s\prec t\in \ttt_\xi$, $f_{j(s)}(j(t))<\varepsilon_{|t|}$, which we can do by Lemma \ref{strong coloring}.  Fix $t\in \ttt_\xi$ and let $s$ be the immediate predecessor of $t$ in $\mt_\xi$.  Choose $y_t\in \text{co}_p(x_r: r\in j(t))$ so that $\|P_{S_{j(s)}} y_t\|<\varepsilon_{|t|}$, which we can do since $f_{j(s)}(j(t))<\varepsilon_{|t|}$.  Let $A_t=\supp_E(y_t)\setminus S_{j(s)}$.  Let $z_t= P_{A_t} y_t$ and note that $\|z_t- y_t\| = \|P_{S_{j(s)}}y_t\|<\varepsilon_{|t|}.$ Thus for an appropriate choice of $\varepsilon_n\downarrow 0$ and an appropriate scaling of $(z_t)$, we deduce that $(z_t)_{t\in \ttt_\xi}\in T_p(X_E, C)$. But for any $s'\preceq s\prec t$, $\supp_E(z_{s'})\subset \supp_E(y_{s'})\subset S_{j(s)}$, so that $z_{s'}$ and $z_t$ have disjoint supports.  Thus for each $t\in \ttt_\xi$, $(\Pi z_{t|_i})_{i=1}^{|t|}\in T_p(U, C)$, since it is isometrically equivalent to $(z_{t|_i})_{i=1}^{|t|}$.  This contradiction finishes the proof.

\end{proof}

\begin{proposition} If $U$ is a reflexive Banach space with $1$-unconditional basis $E$ and if $\alpha\leqslant \omega_1$ is such that for each $e\in E$, $X_e$ is a Banach space with $\ssm_p(X_e)< \omega^\alpha$ (or $\ssm_p(X_e)\leqslant \omega^\alpha$ if $p=1$), then for any $1\leqslant K<C$, $$\ssm_p(X_E, K) \leqslant \omega^\alpha \ssm_p(U, C).$$

\end{proposition}

\begin{proof} First, suppose that $(x_n)$ $K$-dominates and is $1$-dominated by the $\ell_p$ basis.  By replacing $K$ with a strictly larger number, we can assume that for each $n\in \nn$, $S_n= \cup_{i=1}^n\supp_E(x_i)$ is finite.   Note that for any finite $S\subset E$, any $\varepsilon>0$,  and any $(m_n)=M\in \infin$, there exists $F\in \fin$ and $(a_n)_{n\in F}\in S_{\ell_p^{|F|}}$ so that $\|P_S\sum_{n\in F}a_n x_{m_n}\|< \varepsilon$.  If it were not so, then $(P_S x_n)_{n\in M}$ would be equivalent to the $\ell_p$ basis in the finite direct sum $X_S$.  But since no space $X_e$ contains $\ell_p$, the finite direct sum $X_S$ cannot.  Thus for any $\varepsilon_n\downarrow 0$, we can recursively choose a $p$-absolutely convex block $y_n=\sum_{i\in F_n} a_i x_i$ so that for each $n\in \nn$, $\|P_{S_{\max F_{n-1}}}y_n\|< \varepsilon_n$.  Then with $z_n= y_n- P_{S_{\max F_{n-1}}}y_n$, we obtain a small perturbation of $(y_n)$ which is disjointly supported in $X_E$ and $(K+\varepsilon)$-equivalent to the $\ell_p$ basis.  Then $(\Pi z_n)\subset U$ is $(K+\varepsilon)$-equivalent to the $\ell_p$ basis in $U$.  Therefore we deduce that $\ssm_p(U, C)=\infty$.  

We must consider the case in which $\ssm_p(X_E, K)\leqslant \omega_1$. To obtain a contradiction, assume $\xi= \ssm_p(U, C)$ and $\zeta= \omega^\alpha$ are such that $\zeta\xi<\ssm_p(X_E, K)$, and in particular, $\zeta, \xi$ are countable.  The argument in this case proceeds as in the previous paragraph: For any finite $S\subset E$, any $\varepsilon>0$, and any $M\in \infin$, there exists $F\in \sss_\zeta$ and $(a_n)_{n\in F}\in S_{\ell_p^{|F|}}$ so that $\|P_S\sum_{n\in F}a_n x_{m_n}\|< \varepsilon$.  If this were not so, then $(P_Sx_n)_{n\in M}$ would be an $\ell_p^\zeta$ spreading model in the finite direct sum $X_S$.  But since $\ssm_p(X_e)<\omega^\alpha= \zeta$, for any finite $S\subset E$, $\ssm_p(X_S) \leqslant \sum_{e\in S}\ssm_p(X_e)<\omega^\alpha$.  Therefore if we choose $M\in \infin$ so that $\sss_\xi[\sss_\zeta](M)\subset \sss_{\zeta+\xi}$, we can choose a $p$-absolutely convex block $(y_n)$ of $(x_{m_n})$ and disjointly supported vectors $(z_n)$ in $X_E$ so that $\|y_n-z_n\|<\varepsilon_n$ for all $n\in \nn$ and so that $y_n=\sum_{i\in F_n} a_i x_{m_i}$, where $F_n\in \sss_\zeta$.  Thus the sequence $(y_n)$ is a $K$-$\ell_p^\xi$ spreading model, and $(z_n)$ is a $C$-$\ell_p^\xi$ spreading model for an auspicious choice of $\varepsilon_n$.  Therefore $(\Pi z_n)$ is a $C$-$\ell_p^\xi$ spreading model in $U$, a contradiction.

To treat the $p=1$ case with the assumption that $\ssm_1(X_e)\leqslant \omega^\alpha$ for each $e\in E$, we replace the subadditivity with the fact that admitting no $\ell_1^{\omega^\alpha}$ spreading model is a three space property.  Thus any finite direct sum of spaces admitting no $\ell_1^{\omega^\alpha}$ admits no $\ell_1^{\omega^\alpha}$ spreading model.

\end{proof}

We come to one final application of these methods.  

\begin{proposition} Fix $\xi\in \emph{\ord}$.  If for each $k\in \nn$, $X_k$ is a Banach space with $\ssm_1(X_k)\leqslant \omega^\xi$, then $\ssm_p\bigl((\oplus X_k)_{\ell_2}\bigr)\leqslant\omega^\xi$.  

\end{proposition}

\begin{proof} Choose $\eta, \mu \in  (0,1)$ so that $\eta^2+\mu^2=1$ and $2\eta^2-1>\cos \theta,$ where $\theta$ is the angle between $(\eta, \mu)$ and $(\mu, \eta)$. If $(\oplus X_k)_{\ell_2}$ admits an $\ell_1^{\omega^\xi}$ spreading model, it admits one with constant $1\leqslant K<\eta^{-1}$.  Assume that $(x_n)$ is a $K$-$\ell_1^{\omega^\xi}$ spreading model in $(\oplus X_k)_{\ell_2}$.  Let $v_n=\sum_k \|\pi_{e_k}(x_n)\|e_k\in B_{\ell_2}$.  By passing to a subsequence, we can assume that there exists $v\in B_{\ell_2}$ and a bounded block sequence $(b_n)$ in $\ell_2$ so that $\|v_n- (v+b_n)\|\to 0$.  Then for any free ultrafilter $\uuu$ on $\nn$ and any $l\in \nn$, \begin{align*} l^2 K^{-2} & \leqslant \underset{n_1\in \uuu}{\lim}\ldots \underset{n_l\in \uuu}{\lim}\hspace{2mm} \Bigl\|\sum_{i=1}^l x_{n_i}\Bigr\|^2 \\ &  = \underset{n_1\in \uuu}{\lim}\ldots \underset{n_l\in \uuu}{\lim} \sum_k \Bigl\|\sum_{i=1}^l \pi_{e_k}(x_{n_i})\Bigr\|^2 \\ & \leqslant \underset{n_1\in \uuu}{\lim}\ldots \underset{n_l\in \uuu}{\lim} \sum_k \Bigl(\sum_{i=1}^l \|\pi_{e_k}(x_{n_i})\|\Bigr)^2\\ & = \underset{n_1\in \uuu}{\lim}\ldots \underset{n_l\in \uuu}{\lim}\hspace{2mm} \Bigl\|\sum_{i=1}^l v_{n_i}\Bigr\|_{\ell_2}^2 \\ & = l^2\|v\|^2+l \hspace{2mm} \underset{n\in \uuu}{\lim}\|b_n\|^2. \end{align*} Since this holds for any $l\in \nn$, $\eta <K^{-1}\leqslant  \|v\|$.  Choose $m\in \nn$ so that with $v=\sum c_ke_k$, $\eta^2< \sum_{k=1}^m c_k^2$.  

Let $f:\widehat{\sss}_{\omega^\xi}\to \{0,1\}$ be defined by letting $f(E)=1$ if $\min \{\|P_{[1,m]}x\|: x\in \text{co}_1(x_n:n\in E)\}\geqslant \mu$ and $f(E)=0$ otherwise.  If there exists $M\in \infin$ so that $f|_{\widehat{\sss}_{\omega^\xi}(M)}\equiv 1$, then $(P_{[1,m]}x_n)_{n\in M}$ is a $\mu^{-1}$-$\ell_1^{\omega^\xi}$ spreading model in the finite direct sum $\bigl(\oplus_{i=1}^m X_k\bigr)_{\ell_2^m}$, a contradiction.  Therefore by Lemma \ref{easy coloring} there exists $F_1<F_2<\ldots$ so that $f(F_i)=0$ for all $i\in \nn$ and so that for each $F\in \sss_{\omega^\xi}$, $\cup_{i\in F}F_i\in \sss_{\omega^\xi}$.  Choose $x\in \text{co}_1(x_n:n\in F_2)$ so that $\|P_{[1,m]}x\|<\mu$ and for each $i>2$, choose $n_i\in F_i$.  Then since $(2, i)\in \sss_{\omega^\xi}$, $F_2\verb!^!n_i\subset F_2\verb!^!F_i\in \sss_{\omega^\xi}$ for all $i>2$.  This means that for all $i>2$, \begin{align*} 4\eta^2 & \leqslant \|x+x_{n_i}\|^2 = \|P_{[1,m]}(x+x_{n_i})\|^2+ \|P_{(m, \infty)}(x+x_{n_i})\|^2 \\ & \leqslant \|P_{[1,m]}x\|^2+ 2\|P_{[1,m]}x\|\|P_{[1,m]}x_{n_i}\|+\|P_{[1,m]}x_{n_i}\|^2 \\ & + \|P_{(m, \infty)}x\|^2+2\|P_{(m, \infty)}x\|\|P_{(m, \infty)}x_{n_i}\|+\|P_{(m, \infty)}x_{n_i}\|^2 \\ & \leqslant 2 + 2 u\cdot v.  \end{align*}
where $$u= \bigl( \|P_{[1,m]}x\|, (1- \|P_{[1,m]}x\|)^2\bigr), \hspace{2mm} v=  \bigl( \|P_{[1,m]}x_{n_i}\|,(1-\|P_{[1, m]}x_{n_i}\|^2)^{1/2}\bigr)\in \rr^2.$$ But since $\|P_{[1,m]}x\|<\mu $ and for sufficiently large $i$, $\|P_{[1,m]}x_{n_i}\|\approx \bigl(\sum_{k=1}^m c_k^2\bigr)^{1/2}>\eta$, the angle $\phi$ between $u$ and $v$ is greater than the angle $\theta$ between $(\mu, \eta)$ and $(\eta, \mu)$.  Then we deduce $$2\eta^2-1 \leqslant u\cdot v = \cos \phi < \cos \theta < 2\eta^2-1,$$ and this contradiction finishes the proof.

\end{proof}

With this, we deduce that for any $0<\xi<\omega_1$ and $\xi_n\uparrow \omega^\xi$, the direct sum $\bigl(\oplus X_{\xi_n}\bigr)_{\ell_2}$ of the Schreier spaces $X_{\xi_n}$ admits $1$-$\ell_1^\zeta$ spreading models for all $\zeta<\omega^\xi$, but does not admit an $\ell_1^{\omega^\xi}$ spreading model.  As usual, we can deduce that the dual admits $1$-$c_0^\zeta$ spreading models for all $\zeta<\omega^\xi$ with no $c_0^{\omega^\xi}$ spreading model, and the $p$-convexification admits a $1$-$\ell_p^\zeta$ spreading model for each $\zeta<\omega^\xi$ but no $\ell_p^{\omega^\xi}$ spreading model.  Thus we see that the supremum $\ssm_p(X)=\sup_{K\geqslant 1}\ssm_p(X, K)$ may be attained, and in fact the examples here give for each $1\leqslant p\leqslant \infty$ an example of some $X$ so that $\omega^\xi = \ssm_p(X)=\ssm_p(X, 1)$.

\end{document}